\documentclass[11pt]{article}
\usepackage{amsmath, amsthm, amsfonts, amssymb, color}
\usepackage{mathabx}
\usepackage[pagebackref]{hyperref}
\usepackage{graphicx}
\usepackage{bm}
\usepackage{esint}

\usepackage{bbm, dsfont}
\usepackage[utf8x]{inputenc}
\usepackage{ulem}
\usepackage{enumitem}

\usepackage{xpatch}
\def\thissectiontitle{}
\def\thissectionnumber{}
\newtoggle{noTodos}
\makeatletter
\def\@sect#1#2#3#4#5#6[#7]#8{%
\ifnum #2>\c@secnumdepth
\let\@svsec\@empty
\else
\refstepcounter{#1}%
\protected@edef\@svsec{\@seccntformat{#1}\relax}%
\fi
\@tempskipa #5\relax
\ifdim \@tempskipa>\z@
\begingroup
#6{%
\@hangfrom{\hskip #3\relax\@svsec}%
\interlinepenalty \@M #8\@@par}%
\endgroup
\csname #1mark\endcsname{#7}%
\addcontentsline{toc}{#1}{%
\ifnum #2>\c@secnumdepth \else
\protect\numberline{\csname the#1\endcsname}%
\fi
#7}%
\else
\@xsect
\def\@svsechd{%
#6{\hskip #3\relax
\@svsec #8}%
\csname #1mark\endcsname{#7}%
\addcontentsline{toc}{#1}{%
\ifnum #2>\c@secnumdepth \else
\protect\numberline{\csname the#1\endcsname}%
\fi
#7}}%
\fi
\@xsect{#5}%
\ifnum#2=1\relax
\global\def\thissectiontitle{#8}
\global\def\thissectionnumber{\thesection}
\fi%
}
\pretocmd{\section}{\global\toggletrue{noTodos}}{}{}
\AtBeginDocument{%
  \xpretocmd{\todo}{%
    \iftoggle{noTodos}{
     \addtocontents{tdo}{\protect\contentsline {section}%
        {\protect\numberline{\thissectionnumber}{\thissectiontitle}}{}{} }
      \global\togglefalse{noTodos}
        }{}
    }{}{}%
  }
\makeatother
\def\mc{\mathcal}
\def\mf{\mathfrak}
\def\lesim{\lesssim}

\def\C{\mathbb{C}}
\def\D{\mathbb{D}}

\def\N{\mathbb{N}}
\def\B{\mathbb{B}}
\def\R{\mathbb{R}}
\def\T{\mathbb{T}}
\def\W{\mathbb{W}}
\def\Z{\mathbb{Z}}

\newcommand{\bfx}{{\bf x}}
\newcommand{\bfy}{{\bf y}}
\newcommand{\bfz}{{\bf z}}

\newcommand{\bfT}{{\bf T}}

\newcommand{\bfY}{{\bf Y}}

\newcommand\avsuminner[2]{%
  {\sbox0{$\m@th#1\sum$}%
   \vphantom{\usebox0}%
   \ooalign{%
     \hidewidth
     \smash{\vrule height\dimexpr\ht0+1pt\relax depth\dimexpr\dp0+1pt\relax}%
     \hidewidth\cr
     $\m@th#1\sum$\cr
   }%
  }%
}

\makeatletter

\newcommand*\barredprod{%
  \DOTSB\mathop{%
      \@rodriguez@mathpalette \@rodriguez@overprint@bar \prod
    }\slimits@
}
\newcommand*\@rodriguez@mathpalette[2]{%
  \mathchoice
    {#1\displaystyle      \textfont         {#2}}%
    {#1\textstyle         \textfont         {#2}}%
    {#1\scriptstyle       \scriptfont       {#2}}%
    {#1\scriptscriptstyle \scriptscriptfont {#2}}%
}

\newcommand*\@rodriguez@overprint@bar[3]{%
  \sbox\z@{$#1#3$}%
  \dimen@   = \ht\z@   \advance \dimen@   \p@
  \dimen@ii = \dp\z@   \advance \dimen@ii \p@
  \dimen4 = 1.25\fontdimen 8 #2\thr@@ \relax
  \ooalign{% the resulting box has the same...
    \@rodriguez@bar \dimen@ \z@ \cr   % ... height as the first row
    $\m@th #1#3$\cr
    \@rodriguez@bar \z@ \dimen@ii \cr % ... depth as the last row
  }%
}

\newcommand*\@rodriguez@bar[2]{%
  \hidewidth \vrule \@width \dimen4 \@height #1\@depth #2\hidewidth
}

\makeatother

\def\pnorm#1{ \Big(  #1 \Big) }
\def\bnorm#1{ \Big[  #1 \Big] }
\def\anorm#1{ \Big|  #1 \Big| }
\def\Norm#1{ \Big\|  #1 \Big\| }
\def\norm#1{\big\|  #1 \big\| }
\def\inn#1#2{\langle#1,#2\rangle}
\def\set#1{ \Big\{ #1 \Big\} }
\def\supp{\mathrm{supp}}
\def\vector{\vec}
\def\dim{\text{dim}}
\def\rank{\mathrm{rank}}
\newcommand{\ang}{\measuredangle}
\def\spa{\mathrm{span}}
\def\dist{\mathrm{dist}}
\def\hessian{\mathrm{Hessian}}

\def\sing{\mathrm{Sing}}
\newcommand{\bomega}{\bm{\Omega}}

\newcommand{\BLka}{\mathrm{BL}_{k, A}}
\newcommand{\operat}{T^{\lambda}}

\newcommand{\cell}{\mathrm{cell}}
\newcommand{\trans}{\mathrm{trans}}
\newcommand{\tang}{\mathrm{tang}}
\newcommand{\mfy}{\mathfrak{Y}}

\newcommand{\omegatwo}{\Omega^{(2)}}
\newcommand{\omegaone}{\Omega^{(1)}}
\newcommand{\avetwo}{L^2_{\mathrm{avg}}}
\numberwithin{equation}{section}
\theoremstyle{plain}
\newtheorem{theorem}{Theorem}[section]

\newtheorem{corollary}[theorem]{Corollary}
\newtheorem{lemma}[theorem]{Lemma}
\newtheorem{definition}[theorem]{Definition}
\newtheorem{remark}[theorem]{Remark}

\newtheorem{claim}[theorem]{Claim}

\begin{document}
	
%%%%%%%%%%%%%%%%%%%%%%%%%%%%%%%%%%%%%%%%%%%%%%%%%%%%Title_abstract%%%%%%%%%%%%%%%%%%%%%%%%%%%%%%%%%%%%
\title{A dichotomy for H\"ormander-type \\
oscillatory integral operators}

\author{Shaoming Guo, Hong Wang, Ruixiang Zhang}

%	
%\begin{abstract}
%
%\end{abstract}
\date{ }	
\maketitle

\begin{abstract}
    In this paper, we first generalize the work of Bourgain \cite{MR1132294} and state a curvature condition for H\"ormander-type oscillatory integral operators, which we call Bourgain's condition. This condition is notably satisfied by the phase functions for the Fourier restriction problem and the Bochner-Riesz problem. We conjecture that for H\"ormander-type oscillatory integral operators satisfying Bourgain's condition, they satisfy the same $L^p$ bounds as in the Fourier Restriction Conjecture. To support our conjecture, we  show that whenever Bourgain's condition fails, then the $L^{\infty} \to L^q$ boundedness always  %$L^p$ bounds for H\"ormander-type oscillatory integral operators 
    fails for some $q= q(n) > \frac{2n}{n-1}$, extending Bourgain's three-dimensional result \cite{MR1132294}. On the other hand, if Bourgain's condition holds, then we prove $L^p$ bounds for H\"ormander-type oscillatory integral operators for a range of $p$ that extends the currently best-known range for the Fourier restriction conjecture in high dimensions, given by Hickman and Zahl \cite{hickman2020note}. This gives new progress on the Fourier restriction problem and the Bochner-Riesz problem.
\end{abstract}

%\date{\vspace{-5ex}}
%%%%%%%%%%%%%%%%%%%%%%%%%%%%%%%%%%%%%%%%%%%%%%%%%%%%%%%%%%%%%%%%%%%%%%%%
%\setcounter{tocdepth}{1}
%\tableofcontents

%\listoftodos

\section{Introduction}

Let us recall H\"ormander's problem in \cite{MR340924}. Let $n\ge 2$. Let $B^n$ be the unit ball in $\R^n$ and $a: \R^n\times \R^{n-1}\mapsto \R$ be a smooth function supported on $B^n\times B^{n-1}$. Let $\phi: B^n\times B^{n-1}\to \R$ be a smooth function satisfying the following conditions: 
\begin{enumerate}
\item[(H1)] $\rank\ \nabla_{\bfx}\nabla_{\xi}\phi(\bfx; \xi)=n-1$ for all $(\bfx; \xi)\in B^{n}\times B^{n-1}$;
\item[(H2)] with the map $G_0: B^n\times B^{n-1}\to \R^{n}$ defined by 
\begin{equation}
G_0(\bfx; \xi):=\bigwedge_{j=1}^{n-1} \partial_{\xi_j}\nabla_{\bfx} \phi(\bfx; \xi),
\end{equation}
the curvature condition 
\begin{equation}\label{curvcondeq}
\det \nabla_{\xi}^2 \inn{\nabla_{\bfx}\phi(\bfx; \xi)}{G_0(\bfx; \xi_0)}\big|_{\xi=\xi_0}\neq 0
\end{equation}
holds for all $(\bfx; \xi_0)\in \supp(a)$. 
\end{enumerate}
 Define the oscillatory integral operator
\begin{equation}\label{220221e2.3}
T_N f(\bfx):=\int e^{iN\phi(\bfx; \xi)} f(\xi)a(\bfx; \xi)d\xi.
\end{equation}
If one takes $\phi(\bfx; \xi)=\inn{x}{\xi}+t|\xi|^2$ where $\bfx=(x, t)$, then one can check easily that Hypothesis (H1) and (H2) are satisfied, and $T_N$ becomes the standard Fourier extension operator for the paraboloid. H\"ormander \cite{MR340924} asked the question whether $T_N$ satisfies similar $L^p$-boundedness properties to those of the Fourier extension operator. To be more precise, he asked whether it holds that 
\begin{equation}\label{211003e1.4}
\|T_N f\|_q \lesim_{n, q} N^{-n/q} \|f\|_{\infty},
\end{equation}
for all $q>\frac{2n}{n-1}$. \\

In dimension $n=2$, the answer to H\"ormander's question is affirmative, see for instance Carleson and Sj\"olin \cite{MR361607}, H\"ormander \cite{MR340924} and Fefferman \cite{MR320624}. 
However, in dimension $n=3$, Bourgain \cite{MR1132294} showed that if one takes 
\begin{equation}
\phi(\bfx; \xi)=x_1\xi_1+x_2\xi_2+t\xi_1\xi_2+\frac{1}{2}t^2\xi_1^2,\ \ \bfx=(x_1, x_2, t),
\end{equation}
then \eqref{211003e1.4} may fail for every $q<4$. Indeed, he showed that even if one replaces (H2) by the following stronger assumption 
\begin{equation}
(\mathrm{H2}^+)\ \  \nabla_{\xi}^2 \inn{\nabla_{\bfx}\phi(\bfx; \xi)}{G_0(\bfx; \xi_0)}\big|_{\xi=\xi_0} \text{ is positive definite,}
\end{equation}
the estimate \eqref{211003e1.4} may still fail for some $q>3$, and it may even fail generically. Let us be more precise about this generic failure. By some elementary change of variables, phase functions $\phi(\bfx; \xi)$ satisfying (H1) and (H2) can be taken to be
\begin{equation}\label{211003e1.7}
\phi(\bfx; \xi)=\inn{x}{\xi}+t\inn{A\xi}{\xi}+O(|t||\xi|^3+|\bfx|^2 |\xi|^2),
\end{equation}
where $A$ is a symmetric non-degenerate matrix. Condition $(\mathrm{H2}^+)$ amounts to that $A$ is positive definite. The form \eqref{211003e1.7} is called a normal form of the phase function $\phi(\bfx; \xi)$ at the origin (see \cite[page 323]{MR1132294}). Bourgain \cite{MR1132294} proved in dimension $n=3$ that, if 
\begin{equation}\label{211003e1.7zzz}
\nabla_{\xi}^2\partial_t^2\phi\big|_{\bfx=0, \xi=0} \text{ is not a multiple of } \nabla_{\xi}^2\partial_t\phi\big|_{\bfx=0, \xi=0}
\end{equation}
where $\phi(\bfx; \xi)$ is as in \eqref{211003e1.7}, then \eqref{211003e1.4} fails for every $q< \frac{118}{39}$, which is $>3$.

\normalem On the other hand, for phase functions $\phi(\bfx; \xi)$ satisfying (H1) and $(\mathrm{H2}^+)$, Guth, Hickman and Iliopoulou \cite{MR4047925} (basing on earlier work Guth \cite{guth2018}) obtained the \emph{optimal} range of $q$ for which \eqref{211003e1.4} holds. Denote 
\begin{equation}\label{220717e1_9}
    q_{n, \mathrm{GHI}}:=
    \begin{cases}
    \frac{2(3n+1)}{3n-3}, & \hfill \text{ if $n$ is odd,} \\
    \frac{2(3n+2)}{3n-2}, & \hfill \text{ if $n$ is even.} 
    \end{cases}
\end{equation}
Then \eqref{211003e1.4} holds for all $q>q_{n, \mathrm{GHI}}$.\\

In the first result of the paper, we show that generic failure in the spirit of Bourgain \cite{MR1132294} occurs in every dimension $n\ge 3$. Let us first introduce some terminology. At a given point $(\bfx_0; \xi_0)$, consider a new phase function $\phi'(\bfx; \xi):=\phi(\bfx_0+\bfx; \xi_0+\xi)$. Let us use $\phi''(\bfx; \xi)$ to denote a normal form of $\phi'(\bfx; \xi)$ at the origin $\bfx=0, \xi=0$. We say that Bourgain's condition holds at $(\bfx_0; \xi_0)$ if 
\begin{equation}\label{Bourgaincond}
\nabla_{\xi}^2\partial_t^2\phi''\big|_{\bfx=0, \xi=0} \text{ is a multiple of } \nabla_{\xi}^2\partial_t\phi''\big|_{\bfx=0, \xi=0},
\end{equation}
where the implicit constant is allowed to depend on $\bfx_0$ and $\xi_0$. Otherwise, we say that Bourgain's condition fails at this point. As normal forms are not unique, we need to show that Bourgain's condition is well-defined, and this is done in Corollary \ref{bourgain_defined}. 
\begin{theorem}[Generic failure]\label{main_thm_1}
Let $n\ge 3$ and
\begin{equation}
    q_{n, 1}:=\frac{2(2n^2+n-1)}{2n^2-n-2}.
\end{equation}
Let $\phi: B^n\times B^{n-1}\to \R$ be a smooth function satisfying (H1) and (H2). If Bourgain's condition fails at some $(\bfx_0; \xi_0)\in \supp(a)$, then \eqref{211003e1.4} may fail for every $q<q_{n, 1}$. 
\end{theorem}

Note that when $n=3$, $q_{n, 1}=40/13$, which is slightly better than Bourgain's exponent $118/39$. \\

Based on the above theorem, we think it is very natural to conjecture that \eqref{211003e1.4} holds for every $q>\frac{2n}{n-1}$ if Bourgain's condition holds at every point. The following positive results provide some further evidence for such a conjecture.

For $\delta>0$, we define $\delta$-tubes. Fix a dyadic cube $\theta\subset B^{n-1}$ of side length $\delta$ and let $\xi_{\theta}$ be the center of $\theta$. For $v\in B^{n-1}$ with $v\in \delta\Z^{n-1}$, let $X_{t}(\xi_{\theta}, v)\in B^{n-1}$ denote the unique solution in the $x$ variable to 
\begin{equation}
\nabla_{\xi}\phi(x, t; \xi_{\theta})=v.
\end{equation}
By a $\delta$-tube, we mean 
\begin{equation}
T_{\theta, v}:=\{(x, t): |x-X_{t}(\xi_{\theta}, v)|\le \delta, |t|\le 1\}.
\end{equation}
For a collection $\T$ of tubes $\{T_{\theta, v}\}$, we say that the tubes in $\T$ point in different directions if for two different tubes $T_{\theta_1, v_1}$ and $T_{\theta_2, v_2}$ from $\T$, we always have $\theta_1\neq \theta_2$.

\begin{theorem}[Polynomial Wolff Axiom for $\phi$]\label{main_thm_2}
Let $n\ge 3$. If Bourgain's condition holds for the phase function $\phi$ at every $(\bfx_0; \xi_0)\in \supp(a)$, then the following polynomial Wolff axiom for $\phi$ holds: Let $E\ge 2$ be an integer. For every $\epsilon>0$, there exists $C(n, E, \epsilon)>0$ such that for every collection $\T$ of $\delta$-tubes pointing in different directions, 
\begin{equation}
\#\{T\in \T: T \subset S\}\le C(n, E, \epsilon)|S| \delta^{1-n-\epsilon}
\end{equation}
whenever $S \subset B^n$ is a semialgebraic set of complexity $\le E$.
\end{theorem}

The above polynomial Wolff axiom for $\phi$ satisfying Bourgain's condition is a generalization of that for $\phi(\bfx; \xi)=\inn{x}{\xi}+t|\xi|^2$, proven by Katz and Rogers \cite{MR3881832}. Hickman and Rogers \cite{HR2019} used the polynomial Wolff axiom of Katz and Rogers and proved that \eqref{211003e1.4} holds with $\phi(\bfx; \xi)=\inn{x}{\xi}+t|\xi|^2$ for 
\begin{equation}\label{220728e1_15}
    q>2+\frac{\lambda_{\mathrm{HR}}}{n}+O(n^{-2}),
\end{equation}
where 
\begin{equation}
    \lambda_{\mathrm{HR}}=4/(5-2\sqrt{3})=2.60434\dots
\end{equation}
After verifying the polynomial Wolff axiom for general $\phi$ satisfying Bourgain's condition, one can expect to combine the argument of \cite{MR4047925} and \cite{HR2019}, and prove \eqref{211003e1.4} for all $q$ satisfying \eqref{220728e1_15}. We will indeed prove something stronger. Before stating the next theorem, we first recall the result of Hickman and Zahl \cite{hickman2020note}. Let $\nu^{1/2}$ be the real number that solves the equation 
\begin{equation}
    2 x^3+3 x^2-2=0.
\end{equation}
Denote 
\begin{equation}
    \lambda_{\mathrm{HZ}}:=\frac{4}{2-\nu}=2.59607\dots
\end{equation}
Hickman and Zahl \cite{hickman2020note} used the strong polynomial Wolff axiom by Hickman-Rogers-Zhang \cite{HRZ} and independently Zahl \cite{MR4205111} to further improve the result in \cite{HR2019} and obtained that  \eqref{211003e1.4} holds for 
\begin{equation}\label{220926e1_19}
    q> 2+\frac{\lambda_{\mathrm{HZ}}}{n}+O(n^{-2}),
\end{equation}
with $\phi(\bfx; \xi)=\inn{x}{\xi}+t|\xi|^2$. This result gives the best asymptotic formula (as $n\to \infty$) in the literature for the Fourier restriction conjecture. 

\begin{theorem}\label{220130thm1.2}
Let $\phi: B^n\times B^{n-1}\to \R$ be a smooth function satisfying (H1) and (H2$^+$). If Bourgain's condition holds for the phase function $\phi(\bfx; \xi)$ at every point $(\bfx; \xi)\in \supp(a)$, then \eqref{211003e1.4} holds for 
\begin{equation}\label{220717e1_13}
    q>q_{n, 2}:=2+\frac{2.5921}{n}+O(n^{-2}).
\end{equation}
\end{theorem}

As $\phi(\bfx; \xi)=\inn{x}{\xi}+t|\xi|^2$ also satisfies Bourgain's condition, we obtain the following immediate corollary of Theorem \ref{220130thm1.2}. 
\begin{corollary}[Improved Fourier restriction estimate]
For $q>q_{n, 2}$, it holds that 
\begin{equation}
    \Norm{
    \int_{[0, 1]^{n-1}} e^{iN(\inn{x}{\xi}+t|\xi|^2)}f(\xi)d\xi
    }_q \lesim_{n, q} N^{-n/q} \norm{f}_{\infty}.
\end{equation}
\end{corollary}

%the above theorem improves the asymptotic of Hickman and Zahl \cite{hickman2020note}. 
Recall Bourgain's observation \cite[Remark 3.43]{MR1132294} that the phase function for the Bochner-Riesz problem also satisfies Bourgain's condition, we see Theorem \ref{220130thm1.2} also gives the currently best known bounds for the Bochner-Riesz problem. More precisely, for $\alpha \geqslant 0$, the Bochner-Riesz multiplier of order $\alpha$ is defined by
$$
m^\alpha(\xi):=\left(1-|\xi|^2\right)_{+}^\alpha, \ \ \xi\in \R^n.
$$
As a corollary of Theorem \ref{220130thm1.2} (also Theorem \ref{220704theorem3_1} below), we obtain 
\begin{corollary}[Improved Bochner-Riesz estimate on $\R^n$]
For $q>q_{n, 2}$, it holds that 
\begin{equation}
    \norm{m^{\alpha}(D) f}_{L^q(\R^n)} \lesim_{n, \alpha, q} \norm{f}_{L^q(\R^n)},
\end{equation}
whenever $\alpha> n(\frac{1}{2}-\frac{1}{p})-\frac{1}{2}$. 
\end{corollary}

Before stating the next corollary, let us discuss a prior attempt in breaking the critical range of exponents in \eqref{220717e1_9} for general H\"ormander operators. In \cite{gao2022type}, Gao, Li and Wang, under the assumptions (H1), $(\mathrm{H2})^+$, and the additional assumption
\begin{equation}\label{230129e1_23}
    G_0(\bfx; \xi)/|G_0(\bfx; \xi)| \text{ is constant in } \bfx,
\end{equation}
proved that \eqref{211003e1.4} holds for $q$
satisfying the range of Hickman and Zahl \eqref{220926e1_19}. 

By Theorem \ref{211016thm2.1} and Lemma \ref{211016rem2.2}, one can check directly that phase functions satisfying \eqref{230129e1_23} also satisfy Bourgain's condition. Therefore the bounds Gao, Li and Wang obtained in  \cite{gao2022type} can also be improved to the one stated in Theorem \ref{220130thm1.2}. 

The assumption \eqref{230129e1_23} appears naturally in several interesting applications, including the generalized Bochner-Riesz problem for non-degenerate hyper-surfaces, the local smoothing estimates for fractional Schr\"odinger equations and sharp resolvent estimates outside of the uniform boundedness range. These applications were worked out carefully in \cite{gao2022type}. We include the latter two here. \\

Let $u: \R^{n-1}\times \R\to \C$ be the solution to the equation 
\begin{equation}\label{230129e1_24}
    \begin{cases}
        i\partial_t u+(-\Delta)^{\frac{\alpha}{2}}u=0, & (x, t)\in \R^{n-1}\times R,\\
        u(x, 0)=f(x),  & x\in \R^{n-1},
    \end{cases}
\end{equation}
where $\alpha>1$ and $f$ is a Schwartz function. 
\begin{corollary}[Local smoothing estimates for fractional Schr\"odinger equations]
    Let $\alpha>1$ and let $u$ be a solution to \eqref{230129e1_24}. Then 
    \begin{equation}\label{230129e1_25}
        \norm{u}_{L^q(\R^{n-1}\times [1, 2])}
        \lesim_{\alpha, \beta, n, \epsilon, p}
        \norm{f}_{L^q_{\beta}(\R^{n-1})},
    \end{equation}
    whenever 
    \begin{equation}
        \beta> (n-1)\alpha \pnorm{
        \frac{1}{2}-\frac{1}{q}
        }-\frac{\alpha}{q},
    \end{equation}
    and $q$ satisfies \eqref{220717e1_13}. Moreover, for each fixed $q$, the range of $\beta$ is sharp. 
\end{corollary}
The conjectured range for \eqref{230129e1_25} to hold is $q>\frac{2n}{n-1}$, the same as the range in the Fourier restriction conjecture. \\

The resolvent estimate for the Laplacian on $\R^n$ is of the form 
\begin{equation}\label{230129e1_27}
    \norm{
    (-\Delta-z)^{-1} f
    }_{L^q(\R^n)} \le C_{p, q, n}(z) \norm{f}_{L^p(\R^n)}, \ \ z\in \C\setminus [0, \infty). 
\end{equation}
Here $C_{p, q, n}(z)$ is a constant that is allowed to depend on $p, q, n$ and $z$. We are particularly interested in tracking the dependence on $z$, for fixed $n, p$ and $q$. 
\begin{corollary}[Resolvent estimates]
    For $q$ satisfying \eqref{220717e1_13}, we have 
    \begin{equation}\label{230129e1_28}
        \norm{
    (-\Delta-z)^{-1} f
    }_{L^q(\R^n)} 
    \lesim_{q, n} 
    |z|^{-1+\gamma_q} 
    \mathrm{dist}(z, [0, \infty))^{-\gamma_q} \norm{f}_{L^q(\R^n)},
    \end{equation}
    where
    \begin{equation}
        \gamma_q:=\frac{n+1}{2}-\frac{n}{q}.
    \end{equation}
    Moreover, for fixed $q$, the bound is optimal in $z$. 
\end{corollary}
The study of resolvent estimates in \eqref{230129e1_27} has a long history and can be dated back to the work of Kenig, Ruiz and Sogge \cite{MR894584}. The authors there proved \eqref{230129e1_27} for optimal ranges of $(p, q)$ for which the constant $C_{p, q, n}(z)$ is independent of $z$. These are called uniform Sobolev inequalities, and have found numerous applications including unique continuation properties, limiting absorption principles, etc. 

The conjectured range for \eqref{230129e1_28} to hold is $q>\frac{2n}{n-1}$, the same as the Fourier restriction exponent. Moreover, if this conjecture turns out to be true, then by interpolation with known results, it would imply \eqref{230129e1_27} for all combinations of $(p, q)$, with an optimal dependence of $C_{p, q, n}(z)$ on $z$. \\

In the last corollary, we discuss the connection of Theorem \ref{220130thm1.2} to Sogge's work \cite{MR835795} and \cite{MR1639543}. In \cite{MR1639543}, Sogge studied the Nikodym problem on general Riemannian manifolds and proved that the Nikodym maximal operator satisfies better bounds on manifolds of constant scalar curvature than on general Riemannian manifolds. This is a strong indication that distance functions on manifolds of constant curvature satisfy Bourgain's condition. In the next corollary, we show that this is indeed the case for $S^n$, the $n$-dimensional Euclidean sphere. 
\begin{corollary}\label{230129coro1_8}
    Let $a: S^n\times S^n\to\R$ be a smooth function supported away from the diagonal. Let $\mathrm{dist}$ be the distance function on $S^n$. Then 
    \begin{equation}\label{230129e1_30}
        \Norm{
        \int_{S^n} 
        e^{
        iN
        \mathrm{dist}(x, y)
        }
        a(x, y)f(y) dy
        }_{L^q(S^n)} 
        \lesim 
        N^{-n/q} \norm{f}_{L^q(S^n)},
    \end{equation}
    for every $q$ satisfying \eqref{220717e1_13}. 
\end{corollary}
\begin{proof}[Proof of Corollary \ref{230129coro1_8}]
By cutting the support of $a$ into finitely many pieces, we without loss of generality assume that the support of $a(x, y)$ is such that $x$ is around the north pole and $y$ is slightly away from the north pole. Write 
\begin{equation}
    x=(x_1, \dots, x_n, \sqrt{1-|x'|^2}), \ \ y=(y_1, \dots, y_n, \sqrt{1-|y'|^2}),
\end{equation}
where $x'=(x_1, \dots, x_n), y'=(y_1, \dots, y_n)$. Note that $\dist(x, y)=\arccos(x\cdot y)$. When integrating in $y$ on the left hand side of \eqref{230129e1_30}, we apply Fubini's theorem and integrate in $y$ with $\sqrt{1-|y'|^2}=r$ for some $r$ and then integrate in $r$. For a fixed $r$, our distance function can be written as 
\begin{equation}\label{230130e1_32}
    \arccos(
    x_1y_1+\dots +x_{n-1} y_{n-1}
    + x_n \sqrt{1-r^2-|y''|^2}+r\sqrt{1-|x'|^2}
    ),
\end{equation}
where $y'':=(y_1, \dots, y_{n-1})$. Therefore, to prove Corollary \ref{230129coro1_8}, it suffices to show that \eqref{230130e1_32} satisfies Bourgain's condition in $x'$ and $y''$ variables. We will prove this by checking the definition of Bourgain's condition as in \eqref{Bourgaincond}. By rotation symmetry, it suffices to consider the phase function \eqref{230130e1_32} near the north pole $x'=0$ and $y''=0$. Next, we apply a Taylor expansion for \eqref{230130e1_32} about $x'=0$ and $y''=0$. After the Taylor expansion, note that all linear terms in $y''$ can be written as 
\begin{equation}
    (x_1+O(|x'|^2))y_1+\dots+(x_{n-1}+O(|x'|^2))y_{n-1}.
\end{equation}
We therefore apply the change of variables 
\begin{equation}
    x_1+O(|x'|^2)\mapsto x_1, \dots, x_{n-1}+O(|x'|^2)\mapsto x_{n-1},
\end{equation}
to turn our phase function into a normal form. Let $\phi''$ denote this normal form. In the end, we just need to note that both matrices 
\begin{equation}
    \nabla_{y''}^2 \partial^2_t \phi''\big|_{x'=0, y''=0}, \ \ \nabla_{y''}^2 \partial_t \phi''\big|_{x'=0, y''=0}
\end{equation}
are multiples of the identity matrix. This finishes the proof. 
\end{proof}

It is conjectured (see Sogge \cite{MR835795}) that \eqref{230129e1_30} holds for all $q>2n/(n-1)$. The operator studied in \eqref{230129e1_30} appears in the study of the Bochner-Riesz problem on Euclidean spheres $S^n$. Let $\Delta_g$ denote the Laplace-Beltrami operator, and 
\begin{equation}
0<\lambda_1\le \lambda_2\le \dots
\end{equation}
the eigenvalues of $-\Delta_g$. Let $E_j$ be the one-dimensional eigenspace for $-\Delta_g$ with eigenvalue $\lambda_j$, and $e_j: L^2(S^n)\to L^2(S^n)$ be the projection operator onto the eigenspace $E_j$. We define the Riesz means of index $\alpha\ge 0$ as 
\begin{equation}
S^{\alpha}_L(f):=\sum_{j=1}^{\infty}
\pnorm{
1-\frac{\lambda_j}{L}
}^{\alpha}_+ e_j(f).
\end{equation}
One can follow the work Sogge \cite{MR835795} and Huang and Sogge \cite{MR3275105}, and deduce the following corollary from Corollary \ref{230129coro1_8}. 
\begin{corollary}[Bochner-Riesz for spheres]\label{220112corollary1_6}
    Assume that $q$ satisfies \eqref{220717e1_13}. We have that 
    \begin{equation}
        \norm{S^{\alpha}_L (f)}_{L^q(S^n)}
        \lesim \norm{f}_{L^q(S^n)} \ \text{ uniformly in } L,
    \end{equation} 
    whenever
    $\alpha> n(\frac{1}{2}-
    \frac{1}{p}
    )-\frac{1}{2}$.  Moreover, the range of $\alpha$ is sharp for fixed $q$.
\end{corollary}

\medskip

At the end of the introducion, we discuss the proof of Theorem \ref{220130thm1.2}. To prove Theorem \ref{220130thm1.2}, we first prove a strong Polynomial Wolff Axiom (SPWA) for the phase function $\phi$ satisfying Bourgain's condition, see Section \ref{220704section4}. It is a nested version of the Polynomial Wolff Axiom and generalizes the SPWA by Hickman and Zahl \cite{hickman2020note}, which was built on  the work of Hickman-Rogers-Zhang \cite{HRZ} and Zahl \cite{MR4205111}. One can combine this strong polynomial Wolff axiom with the argument in \cite{hickman2020note}, and prove \eqref{211003e1.4} for $q$ satisfying \eqref{220926e1_19}.

 To improve the range in \cite{hickman2020note}, we further develop the idea of ``brooms'' in dimension $n=3$ in Wang \cite{wang2018restriction} for the Fourier restriction problem. In \cite{wang2018restriction}, the second author introduced a notion of \emph{brooms}. They enable one to exploit the feature that if the sum of a collection of wave packets is highly concentrated locally in space, then this collection must spread out on the far end, leading to new improvements on the range of exponent for the Fourier restriction conjecture in $\R^3$.

However, a key  geometric argument in \cite{wang2018restriction} regarding brooms relies heavily on the space being three-dimensional. Even in the Fourier restriction setting, it was not clear how one can most efficiently generalize the notion and construction of brooms in  \cite{wang2018restriction} to higher dimensions. In the current paper, we come up with a slightly different notion of brooms, which works in all dimensions and also in the setting of oscillatory integral operators satisfying H\"ormander's conditions. This is done in Subsection \ref{220926sub7_1}. When proving the relevant broom estimates (see Theorem \ref{220615thm8_7} in Subsection \ref{220926sub7_3}), we use an argument that can be viewed as a generalized pseudo-conformal transformation (Lemma \ref{220608lemma8_9}). By this transformation and a counting lemma (Lemma \ref{lem: counting} below) the key broom estimate is then reduced to what we call a \emph{Variety Uncertainty Principle}. We  state it in the Fourier transform case below as we will use it to prove the general case. 
\begin{lemma}[Variety Uncertainty Principle]\label{220614lemma8_8}
Given two $(m-1)$-dimensional algebraic varieties $Y_1, Y_2$ in $\R^{n-1}$ that are transverse complete intersections (see Definition \ref{221010defi5_1} below). Let $Z_i\subset Y_i$ be the part of $Y_i$ where every point is non-singular and the angle formed by $T_{\bfz_i}(Z_i)$ and the space spanned by $\{\vector{e}_1, \dots, \vector{e}_{m-1}\}$ is $\le 1/(100n)$, for every $i=1, 2$ and every $\bfz_i\in Z_i$. Here $T_{\bfz_i}(Z_i)$ refers to the tangent space and $\vec{e}_j$ refers to a coordinate vector. Let $1\le R_1\le R_2$. Denote\footnote{Here $\mc{N}$ means neighborhood in $\R^{n-1}$; in this lemma there is no $\R^n$. }
\begin{equation}
    \Omega_1=\mc{N}_{\sqrt{R_1}}(Z_1), \ \ \Omega_2=\mc{N}_{1/\sqrt{R_2}}(Z_2).
\end{equation}
Assume that $F: \R^{n-1}\to \C$ satisfies $\supp(F)\subset \Omega_2$. Then 
\begin{equation}
    \norm{\widehat{F}}_{L^2(\Omega_1)}^2 \lesim \pnorm{\frac{R_1}{R_2}}^{\frac{n-m}{2}-\delta} \norm{F}_{L^2}^2,
\end{equation}
for every $\delta>0$, where the implicit constant depends on $n, m$, $\deg(Z_1)$,  $\deg(Z_2)$ and $\delta$. 
\end{lemma}

Lemma \ref{220614lemma8_8} may also be of independent interest as it can be viewed as a variant of the generalized Mizohata-Takeuchi Conjecture by Jonathan Bennett and Tony Carbery, as stated in (9) in \cite{bennett2022tomographic}. For the original Mizohata-Takeuchi Conjecture and its influences, see e.g. the references in \cite{bennett2022tomographic}. As one sees from the proof, Lemma \ref{220614lemma8_8} is proved by using the geometric information of neighborhoods of both the spatial set and the hypersurface in the Fourier space at many scales, and can be viewed as a result in the vein of Mizohata-Takeuchi but with much stronger assumptions about the neighborhoods of the underlying sets at many scales.\\ %for bounded degree hypersurfaces in the frequency space and the Lebesgue measure of a neighborhood of a hypersurface of bounded degree in the physical space.\\

\noindent {\bf Structure of the paper.} In Section \ref{220717section2} we first give an equivalent characterization of Bourgain's condition, which is more straightforward to check, and then prove Theorem \ref{main_thm_1}. In dimension $n=3$, the improvement of our result over Bourgain's \cite{MR1132294} comes from a slightly more efficient way of constructing (curved) tubes that have high overlapping. 

In Section \ref{220728section3}, we show that for phase functions $\phi(\bfx; \xi)$ satisfying Bourgain's condition, the corresponding tubes satisfy the polynomial Wolff axiom. We follow largely the argument of Katz and Rogers \cite{MR3881832}.

In Section \ref{220717section3}, we introduce the standard wave packet decomposition and standard reduction of Theorem \ref{220130thm1.2} to a broad norm estimate (Theorem \ref{201204thm5_1}).

In Section \ref{220717section5} we apply a polynomial partitioning algorithm to decompose the broad norm in Theorem \ref{201204thm5_1}. The algorithm is a slight variant of that in Hickman and Rogers \cite{HR2019}, with one difference that we need to have a better control of how fast cells shrink. 

In Section \ref{220704section4}, we prove the strong polynomial Wolff axiom (mentioned below Theorem \ref{220130thm1.2}) for phase functions satisfying Bourgain's condition.

In Section \ref{220706section6}, we define brooms and prove the broom estimate (Theorem \ref{220615thm8_7}), which is key to the proof of Theorem \ref{201204thm5_1}. It is worth noting that the broom estimate holds for all phase functions $\phi(\bfx; \xi)$ satisfying (H1) and $(\mathrm{H2}^+)$, and does not rely on Bourgain's condition. 

In Section \ref{220717section7}, we define bushes and prove bush estimates. They are used to handle ``small" grain resulting from the polynomial partitioning algorithm in Section \ref{220717section5}. 

In Section \ref{220706section8}, we put all the ingredients together and finish the proof of Theorem \ref{201204thm5_1}, the broad norm estimate. \\

\noindent {\bf Notation.} We use $\bfx=(x, t)$ to refer to a spatial points in $\R^n$, and $\xi$ or $\omega$ for a frequency point in $\R^{n-1}$. Denote $\partial_i=\partial_{x_i}$ if $1\le i\le n-1$ and $\partial_n=\partial_t$.

We use $d$ for the degree of polynomials that we will use in the polynomial partitioning lemmas. It is a large constant and is not allowed to depend on parameters like $R$ or $\lambda$.

We will use a few admissible parameters
\begin{equation}\label{admissable_parameters}
\epsilon^{C} \leqslant \delta \ll_{\epsilon} \delta_{n} \ll_{\epsilon} \delta_{n-1} \ll_{\epsilon} \cdots \ll_{\epsilon} \delta_{1}\ll_{\epsilon} \delta_0 \ll_{\epsilon} \epsilon_{\circ} \ll_{\epsilon} \epsilon .
\end{equation}
Here $C$ is some dimensional constant and the notation $A \ll_{\epsilon} B$ indicates that $A \leqslant C_{n, \epsilon}^{-1} B$ for some large admissible constant $C_{n, \epsilon} \geqslant 1$. These parameters have exactly the same meaning as their counterparts $\delta, \delta_n, \dots, \delta_0, \epsilon$ in Hickman-Rogers' work \cite{HR2019}. In the current paper, we need more of these parameters. For each $1\le n'\le n$, let $\delta_{n'-1/2}$ be such that 
\begin{equation}\label{220707e1_12}
    \delta_{n'}\ll_{\epsilon} \delta_{n'-1/2}\ll_{\epsilon} \delta_{n'-1}.
\end{equation}
These new parameters will be used when we modify the polynomial partitioning algorithm of \cite{HR2019}. 

For two positive constant $A, B$, by $A\lessapprox B$ we mean $A\le R^{O(\delta)} B$. 

For a function $F: \R^n\to \C$ and a region $\Omega\subset \R^n$, we say that $F$ is essentially supported on $\Omega$ if it decay rapidly outside $\Omega$. In other words, for $\bfx\in \R^n\setminus \Omega$, and every $N\in \N$, it holds 
\begin{equation}
    |F(\bfx)|\le C_{n, N} R^{-N},
\end{equation}
for some constant $C_{n, N}$. Here $R>1$ is as in Theorem \ref{201204thm5_1}.\\

\noindent {\bf Acknowledgement.} Guo is partly supported by  NSF DMS-1800274 and NSF DMS-2044828. Wang is partly supported by NSF DMS-2055544 and an AMS-Simons travel grant. Zhang is partly supported by NSF  DMS-2207281, NSF DMS-2143989 and the Sloan Research Fellowship. The authors would like to thank Larry Guth and Changkeun Oh for numerous inspiring discussions. They also thank Chris Sogge for explaining to them the connection of Theorem \ref{220130thm1.2} to the Bochner-Riesz problem on $S^n$.

\section{Bourgain's condition and proof of Theorem \ref{main_thm_1}}\label{220717section2}

\subsection{An equivalent formulation of Bourgain's condition}

In this subsection, we will provide an equivalent formulation of Bourgain's condition. This equivalent formulation will be used in a few places below; for instance, it will play a crucial role in Section \ref{220728section3} and Section \ref{220704section4} when proving polynomial Wolff axioms for H\"ormander's operators satisfying Bourgain's condition. \\

Define 
\begin{equation}\label{defnT}
\vector{T_j}(\bfx; \xi)=\partial_{\xi_j} \nabla_{\bfx} \phi(\bfx; \xi),
\end{equation}
and 
\begin{equation}\label{211017e2.2}
\vector{V}(\bfx; \xi)=\vector{T}_1(\bfx; \xi)\wedge\cdots \wedge \vector{T}_{n-1} (\bfx; \xi). 
\end{equation}
The main result in this subsection is 
\begin{theorem}\label{211016thm2.1}
For a phase function $\phi(\bfx; \xi)$ satisfying conditions (H1) and (H2), it satisfies Bourgain's condition at $(\bfx_0; \xi_0)$ if and only if 
\begin{equation}\label{211016e2.3}
((\vector{V}\cdot \nabla_{\bfx})^2 \nabla^2_{\xi} \phi)(\bfx_0; \xi_0) \text{ is a multiple of } (( \vector{V}\cdot \nabla_{\bfx}) \nabla^2_{\xi} \phi)(\bfx_0; \xi_0).
\end{equation}
The constant is allowed to depend on $\bfx_0$ and $\xi_0$.
\end{theorem}
\begin{corollary}\label{bourgain_defined}
Let $\phi(\bfx; \xi)$ be a phase function satisfying conditions (H1) and (H2). That Bourgain's condition holds at $(\bfx_0; \xi_0)$ is independent of the choice of normal forms. 
\end{corollary}
Here by $\vector{V}\cdot \nabla_{\bfx}$, we mean
\begin{equation}
    V_1(\bfx; \xi)\partial_{1}+\dots +V_{n-1}(\bfx; \xi)\partial_{n-1}+ V_n(\bfx; \xi)\partial_{n},
\end{equation}
where $\vector{V}=(V_1, \dots, V_n)^T$ and for the sake of simplicity we introduced the notation $\partial_i=\partial_{x_i}$ if $i\le n-1$ and $\partial_n=\partial_t$. It is perhaps worthy emphasizing that because of the dependence of $\vector{V}$ on $\bfx$, the differential operator $(\vector{V}\cdot \nabla_{\bfx})^2$ is equal to 
\begin{equation}\label{211022e2.5}
    \sum_{1\le i, j\le n}V_i V_j\cdot \partial_i \partial_j+\sum_{1\le i\le n}V_i\sum_{1\le j\le n}\partial_i V_j\cdot  \partial_j.
\end{equation}

\begin{proof}[Proof of Theorem \ref{211016thm2.1}.]

We first observe that if \eqref{211016e2.3} is satisfied everywhere, then it is also satisfied everywhere if one replaces $V$ by $\lambda(\bfx; x) V$ where $\lambda(\bfx; x)$ is any smooth scalar function. This can be seen by the straightforward computation
\begin{equation}
    (\lambda\vector{V}\cdot \nabla_{\bfx})^2 f= \lambda^2 (\vector{V}\cdot \nabla_{\bfx})^2 f+ \lambda\cdot  (\vector{V}\cdot \nabla_{\bfx} \lambda)\vector{V}\cdot \nabla_{\bfx} f
\end{equation}
for every smooth function $f = f(\bfx; \xi)$.

Back to the proof of the Theorem. We first prove that if $\phi$ satisfies \eqref{211016e2.3} everywhere, then it also satisfies \eqref{211016e2.3} everywhere after any diffeomorphism in $\bfx$ only or in $\xi$ only. Without loss of generality, we only need to verify this at $({\bf0}, 0)$ and can assume the diffeomorphism always preserves the origin.

In order to show that \eqref{211016e2.3} continues to hold after any diffeomorphism in $\bfx$, in light of the above property, we just need to show the ``direction'' of $\vector{V}\cdot \nabla_{\bfx}$ is invariant. More specifically, if $h$ is a diffeomorphism that preserves the origin and write $h(\bfx) = \bfy$, we use $h_*$ to denote the tangent map of $h$ at the origin and only need to prove
\begin{equation}\label{xdiffinv}
    h_* (\vector{T}_1 \wedge \cdots \wedge \vector{T}_{n-1}) \parallel \vector{S}_1 \wedge \cdots \wedge \vector{S}_{n-1}
\end{equation}
where $\vector{T}_j$ is similarly the original $\vector{T}_j ({\bf0}; 0) =\nabla_{\bfx} \partial_{\xi_j} \phi({\bf0}; 0)$ as before and $\vector{S}_j=\nabla_{\bfy} \partial_{\xi_j} \phi({\bf0}; 0)$ is defined similarly in the tangent space of $({\bf0}; 0)$ in $\bfy$ coordinates.

Now everything can be computed in the tangent space of $({\bf0}; 0)$ in terms of $n$-variate functions $\partial_{\xi_j} \phi (\cdot; 0)$ and we now check \eqref{xdiffinv} using linear algebra. View all vectors $\vector{T}, \vector{S}$ as column vectors as before. Let $J$ denote the Jacobian $\frac{\partial \bfx}{\partial \bfy}|_{\bf0}$ with $J_{ij} = \frac{\partial x_i}{\partial y_j}|_{\bf0}$. Then
\begin{equation}
    \vector{S}_j = J^T \cdot \vector{T}_j, 1 \leq j \leq n-1.
\end{equation}
It is then easy to compute by considering the $(n-1)$-th tensor product of $J$ acting on the $(n-1)$-fold wedge algebra that
\begin{equation}\label{wedgeidentity}
    \wedge_{j=1}^{n-1} \vector{S}_j = \det(J)\cdot (J^{-1}) \cdot (\wedge_{j=1}^{n-1} \vector{T}_j).
\end{equation}
Finally, note that the matrix of $h_*$ is $\frac{\partial \bfy}{\partial \bfx}|_{\bf0} = J^{-1}$. By \eqref{wedgeidentity} we see both sides of \eqref{xdiffinv} is parallel to $(J^{-1}) \cdot (\wedge_{j=1}^{n-1} \vector{T}_j)$ and thus  \eqref{xdiffinv} holds.

Next we show that \eqref{211016e2.3} continues to hold after any diffeomorphism in $\xi$. First note that this property is preserved if we do any invertible diffeomorphism in $\xi$. Indeed, this will multiply a nonzero scalar (equal to the determinant of the linear change of variable in $\xi$) to the whole vector field $\vector{V}$, and will result in a constant congruent transformation in $\nabla_{\xi}^2 \phi$ everywhere.

Hence, it suffices to show that \eqref{211016e2.3} gets preserved if one does a change of variables of the following shape:
\begin{equation}\label{changeofvarkeepinglinearterms}
    \xi_j \mapsto \xi_j + \xi^T A_j \xi + \text{ higher order terms}
\end{equation}
where $A_j (1 \leq j\leq n-1)$ is a symmetric $(n-1)\times (n-1)$ matrix. Assume the Taylor expansion of $\phi$ near the origin is
\begin{align}\label{oldexpressionofV}
    \phi (\bfx; \xi) = & a (\bfx)  + b (\xi) + \sum_{i, j} c_{i,j} x_i \xi_j + \sum_{i}  x_i\cdot \xi^T D_i \xi +\sum_{j} \xi_j \cdot \bfx^T E_j \bfx \nonumber\\
    + &  \sum_{i_1, i_2 , j_1, j_2} f_{i_1, i_2, j_1, j_2} x_{i_1} x_{i_2} \xi_{j_1} \xi_{j_2} + g(\bfx; \xi)
\end{align}
where to simplify notation we wrote $t=x_n$, the functions $a, b$ are polynomials of degree $\leq 3$ and $g$ is a sum of terms of order $\geq 3$ in $\bfx$ and terms of order $\geq 3$ in $\xi$. These terms %along with the $d_{i, j_1, j_2}$ terms 
will not play roles in verifying \eqref{211016e2.3} at the origin. Here $D_i$ and $E_j$ refer to square matrices.

We use $\tilde{\phi}$ to denote the new expression of $\phi$ under the change of variable \eqref{changeofvarkeepinglinearterms}, and use $\widetilde{\vector{T}}_j$ and $\widetilde{\vector{V}}$ to denote the counterpart of $\vector{T}_j$ and $\vector{V}$ after the change of variable \eqref{changeofvarkeepinglinearterms}.

Note that $\tilde{\phi}$ has the Taylor expansion
\begin{align}\label{newexpressionofV}
    \tilde{\phi} (\bfx; \xi) = & \sum_{i, j} c_{i,j} x_i \xi_j + \sum_{i}  x_i \cdot \xi^T (D_i+ \sum_{j} c_{i, j} A_j) \xi +\sum_{j} \xi_j \cdot \bfx^T E_j \bfx \nonumber\\
    &+  \sum_{i_1, i_2 , j_1, j_2} f_{i_1, i_2, j_1, j_2} x_{i_1} x_{i_2} \xi_{j_1} \xi_{j_2} +\sum_{j} \xi^T A_j \xi \cdot \bfx^T E_j \bfx\\
    & + \text{ terms playing no role }\nonumber
\end{align}

Comparing \eqref{oldexpressionofV} and \eqref{newexpressionofV}, we see $\vector{V}$ and $\widetilde{\vector{V}}$ differ at the origin by terms of order at least $1$ in $\bfx$ or $\xi$. Denote $\vector{V}_{0} = (V_{0, 1}, \ldots, V_{0, n})$ to be the common value of $\vector{V}$ and $\widetilde{\vector{V}}$ at $({\bf0}, 0)$. %Denote $\widetilde{\vector{V}} = (\tilde{V}_1, \ldots, \tilde{V}_n)^T$
Now
\begin{align}\label{comparing3rdderivative}
& (\widetilde{\vector{V}}\cdot \nabla_{\bfx}) \nabla^2_{\xi} \tilde{\phi} ({\bf0}; 0)-(\vector{V}\cdot \nabla_{\bfx}) \nabla^2_{\xi} \phi({\bf0}; 0)\nonumber\\
= & \sum_i \sum_j V_{0, i} c_{ij} (2A_j)\nonumber=  0
\end{align}
where the last equality is because by definition, $\vector{V}_{0}$ is orthogonal to each $(c_{1j}, \ldots, c_{nj})^T$.

Next we compare $(\widetilde{\vector{V}}\cdot \nabla_{\bfx})^2 \nabla^2_{\xi} \tilde{\phi}$ and $(\vector{V}\cdot \nabla_{\bfx})^2 \nabla^2_{\xi} \phi$. We will show that they are also equal by using \eqref{oldexpressionofV} and \eqref{newexpressionofV} to compute their difference. We begin by showing that near the origin,
\begin{equation}\label{noxindifference}
    (\widetilde{\vector{V}}- \vector{V})(\bfx; \xi) = O(|\xi| + |\bfx|^2),
\end{equation}
which will greatly simplify our computation. Indeed, use the definition \eqref{211017e2.2} of $\widetilde{\vector{V}}$ and $\vector{V}$, we see that both vectors have the same constant terms. Moreover, observe that in the wedge definition \eqref{211017e2.2} of $\widetilde{\vector{V}}$ and $\vector{V}$, all $\vector{T}$ have the same linear term in $\bfx$, since the coefficients of all $x_{i_1} x_{i_2} \xi_j$ terms are the same for $\phi$ and $\tilde{\phi}$. Hence $\widetilde{\vector{V}}$ and $\vector{V}$ also have the same linear term in $\bfx$ and \eqref{noxindifference} is seen to hold.

By \eqref{noxindifference}, if we write
\begin{equation}\label{TaylorV}
    \vector{V} (\bfx; \xi) = \vec{V}_0 + \sum_{i=1}^n x_i \vec{U}_i + O(|\xi| + |\bfx|^2),
\end{equation}
we can reduce the effect of both $(\widetilde{\vector{V}}\cdot \nabla_{\bfx})^2$ and $(\vector{V}\cdot \nabla_{\bfx})^2$ at the origin to the action of a constant coefficient differential operator $\mathcal{D}_0 = (\vec{V}_0\cdot \nabla_{\bfx})^2 + \sum_j V_{0, j} (\vec{U}_j\cdot \nabla_{\bfx})$. Hence
\begin{align}\label{comparing4thderivative}
& (\widetilde{\vector{V}}\cdot \nabla_{\bfx})^2 \nabla^2_{\xi} \tilde{\phi}({\bf0}; 0)-(\vector{V}\cdot \nabla_{\bfx})^2 \nabla^2_{\xi} \phi({\bf0}; 0)\nonumber\\
= & \mathcal{D}_0(\tilde{\phi} - \phi) ({\bf0}; 0)=  \sum_j 4(\vector{V}_0^T E_j \vector{V}_0) A_j + 2\sum_j \vector{C}_j^T U \vector{V}_0 A_j
\end{align}
where the column vector $\vec{C}_j = (c_{1, j}, \ldots, c_{n, j})^T$ and the $n\times n$ matrix $U$ has the $k$-th column equal to $\vector{U}_k$ and $(\cdot)_j$ means the $j$-th component.

In order to show \eqref{comparing4thderivative} gives $0$, it suffices to show the stronger statement
\begin{equation}\label{cancelcond4th}
    2E_j \vector{V}_0 + U^T \vector{C}_j = 0, \forall 1 \leq j\leq n-1.
\end{equation}

We prove \eqref{cancelcond4th} entrywisely. Take an arbitrary $1 \leq k \leq n$, we need to prove
\begin{equation}\label{cancelcond4thentry}
    2\vector{E}_{j; k}\cdot \vector{V}_0 + \vector{U}_k \cdot  \vector{C}_j = 0
\end{equation}
where $\vector{E}_{j; k}$ is the $k$-th row (or column) of $E_j$. To show this we recall how one obtain $\vector{V}_0$ and $\vector{U}_k$. Recall from \eqref{211017e2.2} and \eqref{oldexpressionofV},
\begin{equation}
    \vector{V}_0 = \vector{C}_1 \wedge \vector{C}_2 \wedge \cdots \wedge \vector{C}_{n-1}.
\end{equation}

To compute $\vector{U}_k$, note that it is the coefficient of $x_k$ in \eqref{TaylorV}. Its computation boils down to expanding the $\vector{T}_j$ in \eqref{defnT} and  \eqref{211017e2.2} into the constant term, the $x_k$ term and higher terms. We see
\begin{equation}
\vector{U}_k = 2\vector{E}_{1; k} \wedge \vector{C}_2 \wedge \cdots \wedge \vector{C}_{n-1} + 2\vector{C}_1 \wedge \vector{E}_{2; k} \wedge \cdots \wedge  \vector{C}_{n-1} + \cdots + 2\vector{C}_1 \wedge \cdots \wedge  \vector{C}_{n-2} \wedge \vector{E}_{n-1; k}.
\end{equation}

Now there is only one nonzero term in the above expression, namely $2\vector{C}_1 \wedge \cdots \wedge \vector{E}_{j; k} \wedge \cdots \wedge \vector{C}_{n-1}$, that contributes to $\vector{U}_k\cdot \vector{C}_j$. Hence
\begin{equation}
    \vector{U}_k \cdot  \vector{C}_j = 2\det (\vector{C}_1, \ldots, \vector{C}_{j-1}, \vector{E}_{j; k}, \vector{C}_{j+1}, \ldots, \vector{C}_{n-1}, \vector{C}_j).
\end{equation}

But we also have
\begin{equation}
    2\vector{E}_{j; k} \cdot \vector{V}_0 = 2\det (\vector{C}_1, \ldots, \vector{C}_{n-1}, \vector{E}_{j; k}).
\end{equation}

Since the two add up to $0$, we see \eqref{cancelcond4thentry}, and thus \eqref{cancelcond4th} holds. This concludes the proof that both terms in \eqref{211016e2.3} at $({\bf0}; 0)$ remain the same under the change of variables \eqref{changeofvarkeepinglinearterms}, finishing the proof that the property \eqref{211016e2.3} is preserved under every diffeomorphism in $\bfx$ only or in $\xi$ only.

Now we just need to prove that if the phase function is in the normal form \eqref{211003e1.7}, Bourgain's condition \eqref{Bourgaincond} at the origin coincides with \eqref{211016e2.3}. Indeed, in the normal form, the expression of every $\vector{T_j}$ has no linear term in $\bfx$ and thus the action of $\vector{V}$ at the origin is the same as $\partial_t$ and  the action of $\vector{V}^2$ at the origin is the same as $\partial_t^2$, verifying the above claim. 
\end{proof}

The first part of the above proof immediately implies the following useful corollary.%Next, let us record one simple lemma that will be useful later. 

\begin{lemma}\label{211016rem2.2}
Let $\lambda(\bfx; \xi)$ be a smooth scalar function that does not take value zero. Then for a phase function $\phi(\bfx; \xi)$ satisfying conditions (H1) and (H2), it satisfies Bourgain's condition at $(\bfx_0; \xi_0)$ if and only if \eqref{211016e2.3} holds with $\vector{V}$ replaced by $\lambda\cdot \vector{V}$.
\end{lemma}

\subsection{Proof of Theorem \ref{main_thm_1}}

Given a phase function $\phi(\bfx; \xi)$, we would like to show that \eqref{211003e1.4} may fail for some $q>\frac{2n}{n-1}$ and some $f\in L^{\infty}$. Let us turn to the dual form of it: 
\begin{equation}\label{211005e3.1}
\int \anorm{\int g(\bfx) e^{iN\phi(\bfx; \xi)}a(\bfx; \xi)d\bfx}d\xi\lesim N^{-n/q}\|g\|_{q'},
\end{equation}
for $q>\frac{2n}{n-1}$. Let $\delta\simeq N^{-1/2}$. Consider a $\delta$-net $\{\xi_{\alpha}\}$ in the $\xi$ variable. For each $\alpha$, we will introduce a curved tube $T_{\alpha}$, whose bottom is a disc of radius $\delta$ and length is about $\delta^{\lambda}$ with $\lambda$ to be determined. Let $\bfT$ denote the collection of tubes $\{T_{\alpha}\}$ and let $\#\bfT$ denote the number of tubes. Moreover, we will find a function $\bomega=(\Omega_1(\xi), \dots, \Omega_{n-1}(\xi))$ such that for every $\alpha$ and every $\bfx\in T_{\alpha}$, we have 
\begin{equation}\label{211005e3.2}
\anorm{\nabla_{\xi} \phi(\bfx; \xi_{\alpha})-\bomega(\xi_{\alpha})}\le \delta. 
\end{equation}
Afterwards, let us set 
\begin{equation}
g(\bfx)=g_{\epsilon}(\bfx)=\sum_{\alpha}\epsilon_{\alpha} e^{-iN\phi(\bfx; \xi_{\alpha})}\chi_{T_{\alpha}}(\bfx)
\end{equation}
where $\chi_{T_{\alpha}}$ is the indicator function of $T_{\alpha}$ and $\epsilon_{\alpha}$ takes $\pm 1$ randomly. If \eqref{211005e3.1} holds, then 
\begin{equation}\label{220609e3_33}
\begin{split}
& \int \max_{\alpha} \anorm{\int_{T_{\alpha}} e^{iN[\phi(\bfx; \xi)-\phi(\bfx; \xi_{\alpha})]} a(\bfx; \xi) d\bfx
} d\xi\\
& \le \int \pnorm{\sum_{\alpha} \anorm{\int_{T_{\alpha}} e^{iN[\phi(\bfx; \xi)-\phi(\bfx; \xi_{\alpha})]} a(\bfx; \xi) d\bfx
}^2}^{1/2} d\xi
\end{split}
\end{equation}
which, by Khintchine's inequality, is bounded by 
\begin{equation}\label{220609e3_34}
N^{-n/q}\pnorm{\int (\sum_{\alpha} \chi_{T_{\alpha}})^{q'/2}d\bfx }^{1/q'}.
\end{equation}
We apply H\"older's inequality, and obtain 
\begin{equation}\label{220609e3_35}
    \eqref{220609e3_34} \lesim N^{-n/q} \big|\bigcup_{\alpha}T_{\alpha}\big|^{\frac{1}{2}-\frac{1}{q}} (\sum |T_{\alpha}|)^{\frac{1}{2}}. 
\end{equation}
Recall the function $\bomega$. Write 
\begin{equation}
    \begin{split}
        \phi(\bfx; \xi)& =\phi(\bfx; \xi_{\alpha})+\inn{\nabla_{\xi}\phi(\bfx; \xi_{\alpha})}{\xi-\xi_{\alpha}}+O(\delta^2)\\
        & =\phi(\bfx; \xi_{\alpha})+\inn{\bomega(\xi_{\alpha})}{\xi-\xi_{\alpha}}+O(\delta^2), \ \bfx\in T_{\alpha}. 
    \end{split}
\end{equation}
Therefore the left hand side of \eqref{220609e3_33} is at least 
\begin{equation}\label{220609e3_37}
    \delta^{n-1} \sum |T_{\alpha}|. 
\end{equation}
We combine \eqref{220609e3_35} and \eqref{220609e3_37}, and obtain 
% \begin{equation}
%     (\sum |T_{\alpha}|)^{\frac{1}{2}}\lesim N^{-n/q} \big|\bigcup_{\alpha}T_{\alpha}\big|^{\frac{1}{2}-\frac{1}{q}} \delta^{-(n-1)}. 
% \end{equation}
\begin{equation}\label{211005e3.6}
\delta^{\lambda/2}\lesim N^{-n/q}\delta^{-(n-1)}\big|\bigcup_{\alpha}T_{\alpha}\big|^{1/2-1/q}.
\end{equation}
In the remaining part, we will construct $\{T_{\alpha}\}$ so that the union of these tubes is small. So far we have been following Bourgain's framework in \cite{MR1132294}. The improvement over Bourgain's result in dimension $n=3$ comes from the construction of the tubes $\{T_{\alpha}\}$. \\

Let us write our phase function in its normal form at the origin, that is, 
\begin{equation}\label{211015e4.7}
\phi(\bfx; \xi)=\inn{x}{\xi}+t\inn{A\xi}{\xi}+O(|t||\xi|^3+|\bfx|^2 |\xi|^2).
\end{equation}
The following lemma is the key for the construction of $\{T_{\alpha}\}$. 
\begin{lemma}\label{211005lem3.1}
If Bourgain's condition fails at the origin, then we can find $\bomega$ with $\bomega(0)=0$ such that the following holds: Let $X_t(\xi): \R^{n-1}\mapsto \R^{n-1}$ denote the unique solution to $\nabla_{\xi}\phi(x, t; \xi)=\bomega(\xi)$ in the $x$ variable, then 
\begin{equation}\label{211016e4.8}
\big|\det\nabla_{\xi} X_t\big|=O(|(t, \xi)|^{n}),
\end{equation}
for $t, \xi$ small. 
\end{lemma}
Let us assume Lemma \ref{211005lem3.1} and finish the proof of Theorem \ref{main_thm_1}. As a consequence of this lemma, if we define 
\begin{equation}
T_{\alpha}=\{(x, t): |x-X_t(\xi_{\alpha})|\le \delta, 0\le t\le \delta^{\lambda}\},
\end{equation}
then \eqref{211005e3.2} holds by mean value theorems. Here the value that $\lambda>0$ takes is not relevant, that is, $\lambda$ can even be very close to zero. Next, pick $\lambda=1/(n+1)$. Lemma \ref{211005lem3.1} then says that the union of the tubes is small: 
\begin{claim}\label{220130claim2.4}
\begin{equation}
\anorm{\bigcup_{\alpha: |\xi_{\alpha}|\le \delta^{\lambda}} T_{\alpha}}\lesim_{\lambda} \delta^{2n\lambda}.
\end{equation}
\end{claim}
Let us first accept the above claim. For $|\xi_{\alpha}|\ge \delta^{\lambda}$, we will construct tubes in the same way. Therefore 
\begin{equation}
\anorm{\bigcup_{\alpha} T_{\alpha}}\lesim \delta^{2n\lambda} \delta^{-(n-1)\lambda}.
\end{equation}
Substituting this into \eqref{211005e3.6} will give us 
\begin{equation}
q\ge \frac{2(2n^2+n-1)}{2n^2-n-2}>\frac{2n}{n-1}.
\end{equation}
This finishes the proof of the lemma, modulo the proof of Lemma \ref{211005lem3.1} and the proof of Claim \ref{220130claim2.4}. \\

\begin{proof}[Proof of Lemma \ref{211005lem3.1}.] Let us start with \eqref{211015e4.7}. Write 
\begin{equation}
\phi(\bfx; \xi)=\inn{x}{\xi}+t\inn{A\xi}{\xi}+t^2Q_2(\xi)+\phi_4(\bfx; \xi),
\end{equation}
where $Q_2(\xi)$ is a quadratic form in $\xi$. The assumption that Bourgain's condition fails at the origin is then equivalent to saying that 
\begin{equation}
\hessian(Q_2) \text{ is not a multiple of } A. 
\end{equation}
Let $\bomega=(\Omega_1(\xi), \dots, \Omega_{n-1}(\xi))$ be smooth with $\bomega(0)=0$. We need to solve 
\begin{equation}\label{211015e4.15}
x+tA\xi+t^2\nabla_{\xi} Q_2(\xi)+\nabla_{\xi}\phi_4(\bfx; \xi)=\bomega(\xi).
\end{equation}
It is not difficult to see that when $\bfx$ and $\xi$ are small, the solution is unique. 
We solve \eqref{211015e4.15} iteratively and write the solution as \begin{equation}\label{211015e4.16}
\begin{split}
X_t(\xi)& =-tA\xi-t^2B\xi+tP_1(\xi)+t^2P_2(\xi)\\
& +\sum_{j=3}^n t^j P_j(\xi)+\widetilde{\bomega}(\xi)+ O(|(t, \xi)|^{n+1}),
\end{split}
\end{equation}
where $B$ is a $(n-1)\times (n-1)$ matrix that is not a constant multiple of $A$, $B\xi=\nabla_{\xi} Q_2(\xi)$, $\widetilde{\bomega}$ depends on $\bomega$, $P_i(\xi)$ is a polynomial of degree $n-i$ with lowest order term of degree 2 for $i=1, 2$, and $P_j(\xi)$ is a polynomial of degree $n-i$ with lowest order term of degree 1 for $j\ge 3$. Here $P_i (1\le i\le n)$ depends on $\bomega$, but it is important that the matrix $B$ does not depend on $\bomega$.  Before computing $\nabla_{\xi} X_t$, let us do the change of variables 
\begin{equation}\label{211015e4.17}
-A\xi+P_1(\xi)\mapsto A\eta.
\end{equation}
Write the right hand side of \eqref{211015e4.16} in the $\eta$ variable: 
\begin{equation}
tA\eta+t^2 B\eta+t^2P_2'(\eta)+\sum_{j=3}^n t^jP'_j(\eta)+\widetilde{\bomega}'(\eta)+ O(|(t, \eta)|^{n+1})=:X'_t(\eta),
\end{equation}
where $P'_2(\eta)$ is a polynomial of degree $n-2$ with lowest order term of degree 2, and $P'_j(\eta)$ is a polynomial of degree $n-i$ with lowest order term of degree 1 for $j\ge 3$. As the change of variables in \eqref{211015e4.16} is non-degenerate, in order to guarantee \eqref{211016e4.8}, we just need to show that 
\begin{equation}\label{211016e4.19}
\big|\det\nabla_{\eta} X'_t\big|=O(|(t, \eta)|^{n}).
\end{equation}
When computing \eqref{211016e4.19}, we will see more clearly why it is convenient to do the change of variables in \eqref{211015e4.17}. Write 
\begin{equation}
X'_t(\eta)=tA\eta+t^2 B\eta+\widetilde{\bomega}'(\eta)+\phi'_4(t, \eta).
\end{equation}
Note that the lowest order term in $\phi'_4$, jointly in $t$ and $\eta$ variables, is four. Compute 
\begin{equation}
\nabla_{\eta} X'_t=tA+t^2B+\nabla_{\eta} \widetilde{\bomega}'(\eta)+\nabla_{\eta}\phi'_4(t, \eta).
\end{equation}
Next we will compute the determinant. As $A$ is non-degenerate, when computing the determinant, we can without loss of generality assume that $A$ is the identity matrix. Moreover, by using Jordan normal forms for $B$, and using the fact that $B$ is not a multiple of $A$, we can therefore without loss of generality assume that $B$ is of the form $[b_{ij}]_{1\le i, j\le n-1}$ with $b_{i1}=b_{i2}=0$ for every $i\ge 3$, and the leading principle minor of order $2$ is one of the following forms 
\begin{equation}\label{220609e3_54}
    \begin{bmatrix}
    \gamma, & 1\\
    0, & \gamma
    \end{bmatrix}
    \text{ or }
    \begin{bmatrix}
    \gamma', & 0\\
    0, & 0
    \end{bmatrix}
    \text{ or } 
    \begin{bmatrix}
    \gamma_1, & -\gamma_2\\
    \gamma_2, & \gamma_1
    \end{bmatrix}
\end{equation}
where $\gamma, \gamma', \gamma_1, \gamma_2\in \R$ and $\gamma'\neq 0, \gamma_2\neq 0$. 
Write
\begin{equation}
\nabla_{\eta} X'_t(\eta)=\nabla_{\eta}\widetilde{\bomega}'(\eta)+
    \begin{bmatrix}
*, & *, & 1_{t, y}, & 1_{t, y}, & \dots\\
*, & *, & 1_{t, y}, & 1_{t, y}, & \dots\\
3_{t, y}, & 3_{t, y}, & 1_{t, y}, & 1_{t, y}, & \dots\\
3_{t, y}, & 3_{t, y}, & 1_{t, y}, & 1_{t, y}, & \dots\\
\dots, & \dots, & \dots, & \dots, & \dots
\end{bmatrix}
\end{equation}
where $i_{t, y}$ means that the lowest order in $t, y$ is $i$, for $i=1, 3$. If we are in the first case in \eqref{220609e3_54}, then 
we pick $\widetilde{\bomega}'$ such that 
\begin{equation}
\nabla_{\eta}\widetilde{\bomega}'=
 \begin{bmatrix}
0, & 0, & 0, & \dots\\
1, & 0, & 0, &   \dots\\
0, & 0,   & 0, & \dots\\
\dots, & \dots, & \dots,  & \dots
\end{bmatrix}
\end{equation}
If we are in the second case, then we pick 
$\widetilde{\bomega}'$ such that 
\begin{equation}
\nabla_{\eta}\widetilde{\bomega}'=\frac{1}{\gamma'}
 \begin{bmatrix}
1, & 1, & 0, & \dots\\
-1, & -1, & 0, &   \dots\\
0, & 0,   & 0, & \dots\\
\dots, & \dots, & \dots,  & \dots
\end{bmatrix}
\end{equation}
The last case in \eqref{220609e3_54} can be handled in the same way. In the end, to find $\bomega$ from $\widetilde{\bomega}'$, we just need to revert the change of variables in \eqref{211015e4.17}. 
\end{proof}

\begin{proof}[Proof of Claim \ref{220130claim2.4}]
Let us start by sketching the ideas in the proof. To begin with, we replace the set $\bigcup_{\alpha: |\xi_{\alpha}|\le \delta^{\lambda}} T_{\alpha}$ by a larger set $\mc{N}_{2\delta} (X (B))$, where $X$ is the map  $(\xi, t) \mapsto (X_t (\xi), t)$, %locally differentiable around the origin
$B$ is a ball of radius $O(\delta^{\lambda})$ around the origin and $\mc{N}_{2\delta}$ refers to the $2\delta$ neighbourhood. Now intuitively the volume of $\mc{N}_{2\delta} (X (B))$ depends on the volume of $X (B)$ and the ``surface area'' of $X (B)$. We will control the volume of $X (B)$ by Lemma \ref{211005lem3.1}. For its    ``surface volume'', we first observe that if $\phi$ and $\bomega$ are polynomials, $\partial X (B)$ is contained in a nice semialgebraic set of dimension $<n$ and hence has a controlled ``surface volume'' by tools in real algebraic geometry. Finally, the general situation can be reduced to the above polynomial situation by a Taylor series approximation. To establish the semialgebracity above, we will use quantifier elimination based on the Tarski-Seidenberg theorem. For a recent application of quantifier elimination in Kakeya and restriction that also helped us to motivate the present proof, see \cite{MR3881832}. The tools we need from real algebraic geometry can be found in references \cite{MR2248869, MR2041428}.\\

We now present the proof details. First we claim that without loss of generality, one may assume $\phi$ and all components of $\bomega$ are polynomials of degree $O_{\lambda}(1)$. To see this, let $M\geq 5$ be a large positive integer to be determined later and replace $\phi$ and $\bomega$ by their degree $M$ Taylor approximations. Since $M\geq 5$, $\phi$ will stay as a legitimate phase function and Bourgain's condition continues to fail at the origin. Moreover, whenever $|t|, |\xi_{\alpha}| \leq \delta^{\lambda}$, the change of $\bomega (\xi_{\alpha})$ is $O(\delta^{(M+1)\lambda})$. Thus for these $t$ and $\xi_{\alpha}$, the distance between the new $X_t(\xi_{\alpha})$ from the old one is $O(\delta^{(M+1)\lambda})$ by the nondegeneracy of $\phi$. Since we only care about the volume of the union of $\delta$-neighborhoods, it suffices to choose $M > \frac{1}{\lambda}$ so that each old $T_{\alpha}$ is contained in the twice-thickening of the corresponding new $T_{\alpha}$. Now the old situation is reduced to the new situation where $\phi$ and all components of $\bomega$ are polynomials of degree $O_{\lambda}(1)$.\\

Take $B$ to be a ball of radius $O(\delta^{\lambda})$ centered at the origin in the $(\xi, t)$ space containing all $(\xi, t)$ with $|\xi|, |t| \leq \delta^{\lambda}$. Let $X$ denote the map $(\xi, t) \mapsto (X_t (\xi), t)$. $X$ is smooth near the origin by the implicit function theorem. By definition, $\bigcup_{\alpha: |\xi_{\alpha}|\le \delta^{\lambda}} T_{\alpha}$ is contained in $\mc{N}_{2\delta} (X (B))$. It suffices to prove
\begin{equation}\label{nbhdvolupperbd}
\anorm{\mc{N}_{2\delta} (X (B))}\lesim_{\lambda} \delta^{2n\lambda}.
\end{equation}
Let us understand the geometry of $X(B)$. For a point in $\partial X(B)$, either it is in $X(B)$ and hence in $X(\sing(X; B))$ where $\sing(X; B)$ is the singular set of $X$ inside $B$, or it is outside of $X(B)$ and hence by a compactness argument it is in $X(\partial B)$. Since
\begin{equation}
    \mc{N}_{2\delta} (X (B)) \in X (B) \bigcup \mc{N}_{2\delta} (\partial X (B)),
\end{equation}
we have
\begin{equation}\label{nbhdinthreesets}
\mc{N}_{2\delta} (X (B)) \subseteq X (B) \bigcup \mc{N}_{2\delta} (X(\sing(X; B))) \bigcup \mc{N}_{2\delta} (X(\partial B))
\end{equation} and will next bound the measures of all three sets on the right-hand side from above.\\

First we bound $\big|{X (B)}\big|$. By the definition of $X$ and Lemma \ref{211005lem3.1}, $\big|\det\nabla X\big| = \big|\det\nabla_{\xi} X_t\big|=O(|(t, \xi)|^{n}).$ Integrating on $B$ we get
\begin{equation}\label{upperbdofXB}
\anorm{{X (B)}} \lesssim |B|\sup_{(\xi, t) \in B}|(\xi, t)|^{n} \lesssim \delta^{2n\lambda}.
\end{equation}
Next we bound $\big|\mc{N}_{2\delta}( X(\sing(X; B)))\big|$. By the chain rule and the non-degeneracy of $\nabla_x\nabla_{\xi}\phi$ near $0$, we can rewrite
\begin{equation}
\begin{split}
    X(\sing(X; B)) =\{(x, t): \exists (\xi, t) \in B \text{ s.t. } \nabla_{\xi}\phi(x, t; \xi)=\bomega(\xi) \\
    \text{ and } \nabla^2_{\xi} \phi(x, t; \xi)  = \nabla_{\xi}\bomega(\xi)\}.
    \end{split}
    \end{equation}
We will analyze this set using tools in real algebraic geometry and first do some setup. \normalem We recall a subset of some $\R^N$ is \emph{semialgebraic} if it can be obtained by finitely many steps of taking unions, intersections, or complements from algebraic sets. The \emph{complexity} of a semialgebraic set is the smallest possible sum of the degrees of all polynomials appearing in a complete description of it. Section 2 of \cite{MR3881832} has a good introduction to basic properties of the above notions, as well as a quantitative quantifier elimination (or a quantitative Tarski-Seidenberg theorem) that we will use below, from analysts' viewpoint. A semialgebraic set in $\R^N$ has a \emph{dimension} that is a non-negative integer $\leq N$. See Chapter 5 of \cite{MR2248869} for its basic properties.

Note that the ball $B$ is a semialgebraic set of complexity $O(1)$. Moreover, $\phi$ and components of $\bomega$ are already polynomials of degree $O_{\lambda}(1)$. Hence by quantitative quantifier elimination (or the quantitative Tarski-Seidenberg theorem, see Theorem 14.16 of \cite{MR2248869}), $X(\sing(X; B))$  is a semialgebraic set of complexity $O_{\lambda}(1)$. By Sard's theorem, $X(\sing(X; B))$ has measure zero and thus has dimension $<n$ (by Proposition 5.53 of \cite{MR2248869}). Now by Corollary 5.7 of \cite{MR2041428}, $X(\sing(X; B))$ can be covered by $O_{\lambda}(1) \times (\delta^{\lambda -1})^{(n-1)}$ many $\delta$-balls. Hence
\begin{equation}\label{upperbdofNXSingB}
\anorm{\mc{N}_{2\delta} (X(\sing(X; B)))} \lesssim_{\lambda} \delta^{1+ (n-1)\lambda}.
\end{equation}
We remark that Corollary 5.7 of \cite{MR2041428} can be viewed qualitatively as a generalization of Wongkew's theorem \cite{Wongkew} to the semialgebraic setting.

Finally we bound $\big|\mc{N}_{2\delta} (X(\partial B))\big|$ via a similar application of real algebraic geometrical tools. First write $$X(\partial B) = \{(x, t): \exists (\xi, t) \in \partial B \text{ s.t. } \nabla_{\xi}\phi(x, t; \xi)=\bomega(\xi)\}$$ and we  see $X(\partial B)$ is a semialgebraic set of complexity $O_{\lambda} (1)$. Since $X$ is smooth on the dilation $2B$, $X(\partial B)$ has zero measure and thus has dimension  $<n$. Applying Corollary 5.7 of \cite{MR2041428} as before we get 
\begin{equation}\label{upperbdofNXpartialB}
\anorm{\mc{N}_{2\delta} (X(\partial B))} \lesssim_{\lambda} \delta^{1+ (n-1)\lambda}.
\end{equation}
Combining \eqref{nbhdinthreesets}, \eqref{upperbdofXB}, \eqref{upperbdofNXSingB} and \eqref{upperbdofNXpartialB} and noticing that our $\lambda = \frac{1}{n+1}$, we finish the proof of the claim.
\end{proof}

\section{Polynomial Wolff axiom: Proof of Theorem \ref{main_thm_2}}\label{220728section3}

In order to prove this theorem, we first state and prove a generalized version of Polynomial Wolff Axiom by Katz-Rogers \cite{MR3881832} as follows. We first introduce more notation. Let $n\ge 2$. Suppose the map
\begin{equation}
    \Phi: \R^{n-1} \times \R \times \R^{n-1}  \to  \R^{n-1}
\end{equation}
is smooth on a neighborhood of $[-1, 1]^{2n-1}$ with $\|\Phi\|_{C^k} \lesssim_k 1, \forall k \geq 1$.\footnote{Depending on the choice of $\epsilon$, we will only use the boundedness of finitely many derivatives of $\Phi$ in the proof.}. %and that
%\begin{equation}\label{Jacobianvlowerbdeqn}
%    |\det(\nabla_v \Phi)| \gtrsim 1, \forall (v, t, \xi) \in [-1, 1]^{2n-1}.
%\end{equation}

\normalem By a \emph{$\delta$-tube for cap $\theta$ with respect to $\Phi$}, we mean some
\begin{equation}\label{defnofTthetavphi}
T_{\xi_\theta, v, \Phi} (\delta, 1):=\{(x, t) \in \R^n: |x-\Phi (v, t, \xi_{\theta})|\le \delta, |t|\le 1\}
\end{equation}
where the $\delta$ in the name indicates the ``thickness'' and the $1$ in the name indicates the time span of the tube. For a collection $\T$ of tubes $\{T_{\xi_\theta, v, \Phi} (\delta, 1)\}$, we say that the tubes in $\T$ point in different directions if all the underlying $\theta$ for them are distinct.

\begin{theorem}[Generalized Polynomial Wolff Axiom]\label{GPWAthm}

Suppose that for every choice of $v \in [-1, 1]^{n-1}$ and $\xi \in [-1, 1]^{n-1}$, \begin{equation}\label{averagelargenessofJacobianeq}
\int_{-\delta^{\epsilon}}^{\delta^{\epsilon}} |\det (\nabla_{v} \Phi(v, t, \xi)\cdot M+\nabla_{\xi} \Phi(v, t, \xi))|\mathrm{d} t \gtrsim_{\epsilon} \delta^{C \epsilon}, \forall M \in \mathrm{Mat}_{(n-1)\times (n-1)} (\R)
\end{equation}
for some constant $C$ that depends only on the dimension $n$, where $\mathrm{Mat}_{(n-1)\times (n-1)} (\R)$ stands for the set of $(n-1)\times (n-1)$ real matrices and the bound is uniform and independent of the choices of $v, \xi$ and $M$. Then for every collection $\T$ of $\delta$-tubes pointing in different directions, 
\begin{equation}\label{numberofdirectionseq}
\#\{T\in \T: T \subset S\}\le C(n, E, \epsilon)|S| \delta^{1-n-\epsilon}
\end{equation}
whenever $S \subset B^n$ is a semialgebraic set of complexity $\le E$.

Moreover the implied constant only depends on bounds of finitely many (depending on $n, E, \epsilon$) derivatives of $\Phi$.
\end{theorem}

\begin{proof}[Proof of Theorem \ref{GPWAthm}] When first reading the proof we recommend fixing $\Phi$, and it will be easy to see from the proof that the constant only depends on bounds of finitely many derivatives of $\Phi$ when we are allowed to change it.

In the proof we always think of $\epsilon$ as fixed and always assume $\delta$ is sufficiently small (that can depend on $\epsilon$) since otherwise the conclusion is easily seen to hold. %We will suppress all dependencies on $\epsilon$ henceforth.
Without loss of generality, we always assume $|S| \gtrsim \delta^{n-1}$ with a suitable absolute constant, since otherwise no $T$ can lie in $S$. Our proof will largely follow that of Theorem 3.1 in \cite{MR3881832}. 

Our $\Phi$ is not necessarily a semialgebraic map and we would like to first make it semialgebraic by a Taylor approximation. Fix a $0< \epsilon_1 <\min\{0.1, \epsilon\}$ and fix a large $K>\frac{2n^3}{\epsilon_1}$ (with more constraints to be determined on both parameters). Let $N$ be the quantity on the left hand side of \eqref{numberofdirectionseq}. We may assume $N \geq 1$. Without loss of generality we can assume
\begin{equation}
    \#\{T_{\xi_\theta, v, \Phi} (\delta, 1)\in \T: T_{\xi_\theta, v, \Phi} (\delta, 1) \subset S, |\xi_\theta| \leq \delta^{\epsilon_1}, |v|\leq \delta^{\epsilon_1}\} \gtrsim \delta^{2(n-1)\epsilon_1} N
\end{equation}
and will focus on giving an upper bound of the number of these $T_{\xi_\theta, v, \Phi} (\delta, 1)$. Without loss of generality we assume $\Phi(0) = 0$.

Now replace $\Phi$ by its $K$-th Taylor approximation at the origin, called $\Phi_1$. The new map $\Phi_1$ is still in $C^{\infty}$. For each $T_{\xi_{\theta}, v, \Phi} (\delta, 1)$ in $\T$ such that $|\xi_\theta| \leq \delta^{\epsilon_1}$ and that $|v|\leq \delta^{\epsilon_1}$, we form a corresponding shrunken tube $T_{\xi_{\theta}, v, \Phi_1} (\delta/2, 2\delta^{\epsilon_1})$ that is defined similar to \eqref{defnofTthetavphi} but have thickness $\delta/2$ and time span $|t| \leq 2\delta^{\epsilon_1}$. By Taylor approximation it is easy to see the entire
\begin{equation}
    T_{\xi_\theta, v, \Phi_1} (\delta/2, 2\delta^{\epsilon_1}) \subset T_{\xi_\theta, v, \Phi} (\delta, 1) \subset S.
\end{equation}
Moreover, by Taylor approximation, we see that the following analogue of \eqref{averagelargenessofJacobianeq} continues to hold for each $|\xi| \leq \delta^{\epsilon_1}$ and  $|v|\leq \delta^{\epsilon_1}$ as long as $K \gtrsim 1$ (when $\delta$ is sufficiently small) and $M \in \mathrm{Mat}_{(n-1)\times (n-1)} (\R)$ are such that all entries of $M$ are $\leq \delta^{-n}$:
\begin{equation}\label{avgbigJacrestreq}
\int_{-\delta^{\epsilon_1}}^{\delta^{\epsilon_1}} |\det (\nabla_{v} \Phi_1(v, t, \xi)\cdot M+\nabla_{\xi} \Phi_1(v, t, \xi))|\mathrm{d} t
\gtrsim_{\epsilon_1} 
\delta^{C \epsilon_1}
\end{equation}
Like Katz-Rogers did in \cite{MR3881832}, we define a set
\begin{equation}\label{tuebinSpropeq}
    L = \{(\xi, v): T_{\xi, v, \Phi_1} (\delta/3, 2\delta^{\epsilon_1}) \subset S\}.
\end{equation}
Now the map $\Phi_1$ is algebraic, by definition $L$ is semialgebraic with complexity $O_{n, E, \epsilon_1, K} (1)$ (for the definition of semialgebraic sets and their complexity and how to arrive at the present claim, see Section 2 of \cite{MR3881832}). Note that if a tube $T_{\xi_\theta, v, \Phi_1} (\delta/2, 2\delta^{\epsilon_1}) \subset S$, then keeping $v$ and perturbing $\xi_{\theta}$ by an arbitrary small distance proportional to $\delta$, the resulting $(\xi, v)$ will end up in $L$ by definition. Hence we know the measure
\begin{equation}\label{allthexieq}
    |\{\xi: \exists v \text{ s.t. } (\xi, v) \in L\}| \gtrsim \delta^{2(n-1)\epsilon_1} N \cdot \delta^{n-1} \simeq \delta^{(n-1)+2(n-1)\epsilon_1} N.
\end{equation}
Like in the proof of Theorem 3.1 in \cite{MR3881832}, next we apply the Tarski-Seidenberg theorem to obtain a semialgebraic section $L' \subset L$ of complexity $O_{n, E, \epsilon_1, K} (1)$ consisting of a single $(\xi, v)$ for each $\xi$ appearing in the set in \eqref{allthexieq}. Arguing like \cite{MR3881832}, we see $L'$ is an $(n-1)$-dimensional subset of $\R^{2n-2}$. Using Gromov's lemma (cited as Lemma 2.3 in \cite{MR3881832}) in the identical way as in pages 1711-1712 of \cite{MR3881832}, we find two polynomial maps $F$ and $G: [0, \delta^{\epsilon_1}]^{n-1} \to \R^{n-1}$ with $\deg F, \deg G = O_{n, E, \epsilon_1, K} (1)$ and $\|F\|_{C^1}, \|G\|_{C^1} \leq 1$ such that
\begin{equation}\label{imofcubelargeeq}
    |G([0, \delta^{\epsilon_1}]^{n-1})| \gtrsim_{n, E, \epsilon_1, K} \delta^{(n-1)+C\epsilon_1} N
\end{equation}
and that
\begin{equation}
    (\Phi_1 (F(x), t, G(x)), t) \in S, \forall x \in [0, \delta^{\epsilon_1}]^{n-1} \text{ and } \forall |t| \leq \delta^{\epsilon_1}.
\end{equation}
For technical reasons that we replaced $\Phi$ by $\Phi_1$ which does not need to satisfy \eqref{averagelargenessofJacobianeq} but instead only satisfies the weaker \eqref{avgbigJacrestreq}, we now pass to a subset $B$ of $[0, \delta^{\epsilon_1}]^{n-1}$ where $G$ has reasonably large Jacobian. Define
\begin{equation}
    B = \{x \in [0, \delta^{\epsilon_1}]^{n-1}: |\det\nabla_x G|\gtrsim_{n, E, \epsilon_1, K} \delta^{(n-1)+C\epsilon_1} N\}.
\end{equation}
We see that if the implied constant above is carefully chosen, by Chebyshev we continue to have the following inequality similar to \eqref{imofcubelargeeq}:
\begin{equation}\label{imofBlargeeq}
    |G(B)| \gtrsim_{n, E, \epsilon_1, K} \delta^{(n-1)+C\epsilon_1} N.
\end{equation}
Now like in \cite{MR3881832}, we look at the volume of
\begin{equation}
    \Delta = \{(\Phi_1 (F(x), t, G(x)), t): x \in B, |t| \leq \delta^{\epsilon_1}\}.
\end{equation}
On one hand, $\Delta$ is contained in $S$ and thus
\begin{equation}\label{upperbdofDeltaeq}
    |\Delta| \leq |S|.
\end{equation}
On the other hand, we can bound $|\Delta|$ below by calculus. Note that $\Phi_1$ is a polynomial of degree $K$ and that $F$ and $G$ are polynomials of degree $O_{n, E, \epsilon_1, K} (1)$. Hence by B\'{e}zout's theorem for every fixed $t$ the map
\begin{equation}
    x \mapsto \Phi_1 (F(x), t, G(x))
\end{equation}
is $O_{n, E, \epsilon_1, K} (1)$ to $1$. Thus
\begin{align}\label{integralineq}
    & |\Delta| \simeq_{n, E, \epsilon_1, K} \int_{-\delta^{\epsilon_1}}^{\delta^{\epsilon_1}} \int_B |\nabla_x (\Phi_1 (F(x), t, G(x)))|  \mathrm{d}x\mathrm{d}t\nonumber\\
    = & \int_{-\delta^{\epsilon_1}}^{\delta^{\epsilon_1}} \int_B |\det(\nabla_v \Phi_1 (F(x), t, G(x))\cdot \nabla_x F + \nabla_{\xi} \Phi_1 (F(x), t, G(x))\cdot \nabla_x G)|  \mathrm{d}x\mathrm{d}t\nonumber\\
    = & \int_B |\det(\nabla_x G)|\nonumber\\
    \cdot & \int_{-\delta^{\epsilon_1}}^{\delta^{\epsilon_1}}  |\det(\nabla_v \Phi_1 (F(x), t, G(x))\cdot (\nabla_x F\cdot (\nabla_x G)^{-1}) + \nabla_{\xi} \Phi_1 (F(x), t, G(x)))|  \mathrm{d}t\mathrm{d}x.
\end{align}
Note that $\|F\|_{C^1}, \|G\|_{C^1} \leq 1$ and that for all $x \in B$,
\begin{equation}
    |\det\nabla_x G|\gtrsim_{n, E, \epsilon_1, K} \delta^{(n-1)+C\epsilon_1} N \geq \delta^{(n-1)+C\epsilon_1},
\end{equation}
we see that each entry of $(\nabla_x F\cdot (\nabla_x G)^{-1})$ is
\begin{equation}
    \lesssim_{n, E, \epsilon_1, K} \delta^{-(n-1)-C\epsilon_1} \le \delta^{-n}
\end{equation}
if $\delta$ is sufficiently small. This allows us to invoke \eqref{avgbigJacrestreq} to obtain
\begin{equation}
    |\Delta|\gtrsim_{n, E, \epsilon_1, K} \delta^{C\epsilon_1} \int_B |\det \nabla_x G(x)|\mathrm{d}x.
\end{equation}
Use B\'{e}zout again and notice \eqref{imofBlargeeq}, the right hand side is
\begin{equation}
    \simeq_{n, E, \epsilon_1, K} |G(B)| \gtrsim_{n, E, \epsilon_1, K} \delta^{(n-1)+C\epsilon_1} N.
\end{equation}
Hence
\begin{equation}\label{lowerbdofDeltaeq}
    |\Delta|\gtrsim_{n, E, \epsilon_1, K} \delta^{(n-1)+C\epsilon_1} N.
\end{equation}
Combine \eqref{upperbdofDeltaeq} and \eqref{lowerbdofDeltaeq}, we obtain
\begin{equation}
    N \lesssim_{n, E, \epsilon_1, K} |S|\delta^{1-n-C\epsilon_1}.
\end{equation}
It suffices to take $\epsilon_1$ to be a suitable multiple of $\epsilon$ depending on the above constant $C$ (and fix $K$ accordingly as in the beginning of the proof).
\end{proof}

Before proving Theorem \ref{main_thm_2}, we also state and prove an elementary lemma on the averaged size of determinants.

\begin{lemma}\label{polymatriceslem}
For $A, B \in \text{Mat}_{k \times k} (\R)$ and a measurable $E \subset \R$, we have
\begin{equation}\label{unifestimateofmatriceseq}
    \int_E |\det (tA+B)|\mathrm{d}t \gtrsim_k |E|^{k+1}|\det (A)|.
\end{equation}
\end{lemma}

\begin{proof}[Proof of Theorem \ref{polymatriceslem}]
Without loss of generality we can assume $\det A \neq 0$. We may further assume $A = I$ since otherwise we can replace $A$ by $I$ and replace $B$ by $BA^{-1}$ and notice $tA+B = A\cdot  (tI + BA^{-1})$.

Now $\det (tI+B)$ is a monic polynomial $g_B(t)$ in variable $t$ of degree $k$. Factorize $g_B$ over $\C$ and notice that the set of $t \in E$ with distance $\geq \frac{|E|}{2k}$ against each root of $g_B$ has measure at least $\frac{|E|}{2}$. For each such $t$,
\begin{equation}
    |\det (tI+B)| \geq \frac{|E|^k}{(2k)^k}.
\end{equation}
This finishes the proof. 
\end{proof}

Note that the key point of Lemma \ref{polymatriceslem} is that the estimate \eqref{unifestimateofmatriceseq} is independent of $B$. With the above preparation we now prove Theorem \ref{main_thm_2}.

\begin{proof}[Proof of Theorem \ref{main_thm_2}]
By the non-degenerate assumption of $\phi$, for sufficiently small $(v, t, \xi)$ one can find a unique $\Phi = \Phi(v, t, x)$ near $0$ such that
\begin{equation}\label{defnofPhi}
    (\nabla_{\xi} \phi)(\Phi, t; \xi)=v.
\end{equation}

Let us assume the above can be done for all $(v, t, \xi)\in [-1.5, 1.5]^{2n-1}$ without loss of generality since otherwise we can perform a constant rescaling. %The non-degenerate assumption of $\phi$  also  implies \eqref{Jacobianvlowerbdeqn}.
It suffices to show that our $\Phi$ satisfies the condition \eqref{averagelargenessofJacobianeq} since we can then use Theorem \ref{GPWAthm} to conclude the proof.\\

By \eqref{defnofPhi}, we get
\begin{equation}
    (\nabla_{\xi} \phi)(\Phi(v, t, \xi), t; \xi)=v. 
\end{equation}
Differentiating with respect to $\xi$ and $v$ respectively, we deduce
\begin{equation}\label{nablaxiPhieq}
    \nabla_x \nabla_{\xi} \phi \cdot \nabla_{\xi} \Phi + \nabla_{\xi}^2 \phi = 0
\end{equation}
and
\begin{equation}\label{nablavPhieq}
    \nabla_x \nabla_{\xi} \phi \cdot \nabla_{v} \Phi = I.
\end{equation}
Note that we have adopted the abbreviation that $\phi$ is evaluated at $(\Phi(v, t, \xi), t; \xi)$. By the non-degeneracy of $\phi$, we know $|\nabla_x \nabla_{\xi} \phi| \simeq 1$. Hence
\begin{align}\label{reductioninPWAeq}
& \int_{-\delta^{\epsilon}}^{\delta^{\epsilon}} |\det (\nabla_{v} \Phi(v, t, \xi)\cdot M+\nabla_{\xi} \Phi(v, t, \xi))|\mathrm{d} t\nonumber\\
\simeq & \int_{-\delta^{\epsilon}}^{\delta^{\epsilon}} |\det (\nabla_x \nabla_{\xi} \phi \cdot \nabla_{v} \Phi(v, t, \xi)\cdot M+\nabla_x \nabla_{\xi} \phi \cdot \nabla_{\xi} \Phi(v, t, \xi))|\mathrm{d} t\nonumber\\
\simeq & \int_{-\delta^{\epsilon}}^{\delta^{\epsilon}} |\det (M-\nabla_{\xi}^2 \phi (\Phi(v, t, \xi), t; \xi))|\mathrm{d} t.
\end{align}
Recall we only need to verify \eqref{averagelargenessofJacobianeq} for $\Phi$. In light of \eqref{reductioninPWAeq}, it now suffices to show that the right hand side of \eqref{reductioninPWAeq} is $\gtrsim_{\epsilon} \delta^{n\epsilon}$ (independent of the choice of $M$, $v \in [-1, 1]^{n-1}$ and $\xi \in [-1, 1]^{n-1}$).\\

From now on we fix $v$ and omit it from place to place, and all the estimates will be uniform  in $v$. For simplicity we use $X_t (\xi)$ to denote $\Phi (v, t, \xi)$ below. Denote 
\begin{equation}\label{defnofA}
A(t; \xi)= \nabla^2_{\xi} \phi(X_t(\xi), t; \xi).
\end{equation}
We claim that for all $t \in [0, 1]$, $\xi \in [0, 1]^{n-1}$, $A(t; \xi) - A(0; \xi)$ as a matrix is proportional to a matrix $B(\xi)$ independent of $t$. This will be refered to as Claim ($\ast$). We will prove this claim by  finding $B(\xi)$ explicitly. Compute 
\begin{equation}
\begin{split}
\partial_t A(t; \xi)& =\pnorm{\nabla_x\cdot \partial_t X_t} \nabla^2_{\xi}\phi(X_t(\xi), t; \xi)+\partial_t  \nabla^2_{\xi}\phi(X_t(\xi), t; \xi)\\
&=\bnorm{\pnorm{\vector{V}\cdot \nabla_{\bfx}}\nabla^2_{\xi}\phi}(X_t(\xi), t; \xi)
\end{split}
\end{equation}
where $\vector{V}:=(\partial_t X_t, 1)$. Note that if we differentiate both side of \eqref{defnofPhi} in $t$, then we obtain 
\begin{equation}
\partial_t X_t \cdot \nabla_{x}\nabla_{\xi}\phi(X_t(\xi), t; \xi)+\partial_t \nabla_{\xi}\phi(X_t(\xi), t; \xi)=0.
\end{equation}
This means at the point $(X_t(\xi), t; \xi)$, the vector field $\vector{V}$ is parallel to the one defined in \eqref{211017e2.2}. Moreover by a similar computation, we obtain that 
\begin{equation}
\partial^j_t A(t; \xi)=\bnorm{\pnorm{\vector{V}\cdot \nabla_{\bfx}}^j\nabla^2_{\xi}\phi}(X_t(\xi), t; \xi)
\end{equation}
for every $j\ge 1$.  By Lemma \ref{211016rem2.2}, $\partial^2_t A(t; \xi)$ is always parallel to $\partial_t A(t; \xi)$. Hence by considering the time derivative of the quotient of two entries we see $\partial_t A(t; \xi)$ is always parallel to $\partial_t A(0; \xi)$. By our non-degeneracy assumption \eqref{curvcondeq} for $\phi$, we see $\partial_t A(0; \xi)$ has norm and determinant $\simeq 1$ and call it $B(\xi)$. Since $\partial_t A(t; \xi)$ is always parallel to  $B(\xi)$, we see that Claim ($\ast$) holds.

Now by Claim ($\ast$), we assume
\begin{equation}\label{scalarcondofA}
    A(t; \xi) = f(t; \xi)B(\xi) + A(0; \xi)
\end{equation}
for some scalar function $f$. Since \begin{equation}
    \partial_t A(t; \xi)|_{t=0} = B(\xi),
\end{equation}
we see that the time derivative of $f$ is $1$ at $t=0$. By compactness, the range of $f(t; \xi)$ for $|t| \leq \delta^{\epsilon}$ has measure $\geq C \delta^{\epsilon}$ for some universal $C>0$ independent of $\xi \in [-1, 1]$ and $\epsilon$. 

We are in a position to apply Lemma \ref{polymatriceslem} to the right hand side of \eqref{reductioninPWAeq}. Note that
\begin{equation}
    M-\nabla_{\xi}^2 \phi (\Phi(v, t, \xi), t; \xi) =  M - A(t; \xi) = (M-A(0; \xi)) - f(t; \xi) B(\xi).
\end{equation}
We see the right hand side of \eqref{reductioninPWAeq} is bounded below by $C\delta^{n\epsilon}\det B(\xi) \geq C\delta^{n\epsilon}$, thus concluding the proof.
\end{proof}

\section{Preliminaries for the proof of Theorem \ref{220130thm1.2}}\label{220717section3}

For $\lambda\ge 1$, denote 
\begin{equation}
    \phi^{\lambda}(\bfx; \xi)=\lambda\phi(\bfx/\lambda; \xi), \ \ a^{\lambda}(\bfx; \xi)=a(\bfx/\lambda; \xi).
\end{equation}
Define an operator 
\begin{equation}\label{220221e4.2}
    T^{\lambda} f(\bfx):=\int e^{i\phi^{\lambda}(\bfx; \xi)} a^{\lambda}(\bfx; \xi)f(\xi)d\xi. 
\end{equation}
Note that $T^{\lambda} f$ is just a rescaled version of the operator in  Theorem \ref{220130thm1.2}, and we use this rescaled version as we will use the wave packet decomposition and uncertainty principles to bound $T^{\lambda}$.  In the rest of the paper, we will prove 
\begin{theorem}\label{220704theorem3_1}
If $\phi$ is  assumed to satisfy (H1), $(\mathrm{H2}^+)$ and Bourgain's condition at every point, then 
\begin{equation}\label{211026e5.3}
    \norm{T^{\lambda}f}_{L^p(B_{\lambda})}\lesim_{\epsilon, p, \phi, a} \lambda^{\epsilon} \|f\|_{L^{p}}
\end{equation}
for every $p>q_{n, 2}$, defined in \eqref{220717e1_13}, ball $B_{\lambda}\subset \R^n$ of radius $\lambda\ge 1$ and every $\epsilon>0. $
\end{theorem}

For the sake of simplicity, we will assume that our phase function $\phi$ is of the normal form. To prove \eqref{211026e5.3}, it suffices to prove
\begin{equation}
    \norm{T^{\lambda}f}_{L^p(B_{R})}\lesim_{\epsilon, p, \phi, a} R^{\epsilon} \|f\|_{L^p},
\end{equation}
for every $1\le R\le \lambda^{1-\epsilon}$ and every cube $B_R\subset B_{\lambda}$. We will run an induction on both parameters $\lambda$ and $R$. The base case of the induction $\lambda=R=1$ is trivial. Let us assume that we have proven 
\begin{equation}\label{220222e4.5}
    \norm{T^{\lambda'}f}_{L^p(B_{R'})}\lesim_{\epsilon, p, \phi, a} (R')^{\epsilon} \|f\|_{L^p},
\end{equation}
for every $\lambda'\le \lambda/2$, $R'\le (\lambda')^{1-\epsilon}$ and every cube $B_{R'}\subset B_{\lambda'}$. Our goal is to prove that the same holds with $\lambda$ and $R$.

\subsection{Wave packet decomposition}\label{220619subsection4_2}

Let $1\le r\le R$ and take a collection $\Theta_{r}$ of dyadic cubes of side length $\frac{9}{11} r^{-1 / 2}$ covering the ball $B^{n-1}(0,2)$. We take a smooth partition of unity $\left(\psi_{\theta}\right)_{\theta \in \Theta_{r}}$ with $\operatorname{supp} \psi_{\theta} \subset \frac{11}{10} \theta$ for the ball $B^{n-1}(0,2)$ such that
$$
\left\|\partial_{w}^{\alpha} \psi_{\theta}\right\|_{L^{\infty}} \lesssim_{\alpha} r^{|\alpha| / 2}
$$
for any $\alpha \in \mathbb{N}_{0}^{n-1}$. We denote by $\omega_{\theta}$ the center of $\theta$. Given a function $g$, we perform a Fourier series decomposition to the function $g \psi_{\theta}$ on the region $\frac{11}{9} \theta$ and obtain
$$
g(w) \psi_{\theta}(w) \cdot \mathbbm{1}_{\frac{11}{10} \theta}(\omega)=\left(\frac{r^{1 / 2}}{2 \pi}\right)^{n-1} \sum_{v \in r^{1 / 2} \mathbb{Z}^{n-1}}\left(g \psi_{\theta}\right)^{\wedge}(v) e^{2 \pi i v \cdot w} \mathbbm{1}_{\frac{11}{10} \theta}(\omega) .
$$
Let $\widetilde{\psi}_{\theta}$ be a non-negative smooth cutoff function supported on $\frac{11}{9} \theta$ and equal to 1 on $\frac{11}{10} \theta$. We can therefore write
$$
g(w) \psi_{\theta}(w) \cdot \widetilde{\psi}_{\theta}(\omega)=\left(\frac{r^{1 / 2}}{2 \pi}\right)^{n-1} \sum_{v \in r^{1 / 2} \mathbb{Z}^{n-1}}\left(g \psi_{\theta}\right)^{\wedge}(v) e^{2 \pi i v \cdot w} \widetilde{\psi}_{\theta}(\omega)
$$
If we also define
$$
g_{\theta, v}(w):=\left(\frac{r^{1 / 2}}{2 \pi}\right)^{n-1}\left(g \psi_{\theta}\right)^{\wedge}(v) e^{2 \pi i v \cdot \omega} \widetilde{\psi}_{\theta}(\omega)
$$
then we have
$$
g=\sum_{(\theta, v) \in \Theta_{r} \times r^{1 / 2} \mathbb{Z}^{n-1}} g_{\theta, v}
$$
For $\omega\in B^{n-1}$, $z'\in B^{n-1}$ and $t\in [0, 1]$, let us define a function $\Phi=\Phi(z', t; \omega)$ by 
\begin{equation}
    \partial_{\omega}\phi(\Phi(z', t; \omega), t; \omega)=z'.
\end{equation}
We refer to equation (4.6) in \cite[page 275]{MR4047925} for a discussion on the definition of $\Phi$. For $\theta\in B^{n-1}$ and $v\in B^{n-1}$, define the curve $\gamma^1_{\theta, v}: [0, 1]\to \R^{n-1}$ by 
\begin{equation}
    \gamma^1_{\theta, v}(t):=\Phi(v, t; \omega_{\theta}), 
\end{equation}
where $\omega_{\theta}$ is the center of $\theta$. Moreover, for given $(\theta, v)$ let us define the rescaled curve \begin{equation}
    \gamma^{\lambda}_{\theta, v}(t):=\lambda \gamma^1_{\theta, v/\lambda} \pnorm{\frac{t}{\lambda}}. 
\end{equation}
Let $\Gamma^{\lambda}_{\theta, v}$ be the map 
\begin{equation}
    \Gamma^{\lambda}_{\theta, v}:=(\gamma^{\lambda}_{\theta, v}(t), t),
\end{equation}
where $t\in [0, \lambda]$. Define the curved $r^{\frac12+\delta}$-tube as 
\begin{equation}
    T_{\theta, v}:=\set{(x, t): |x-\gamma^{\lambda}_{\theta, v}(t)|\le r^{1/2+\delta}, \ t\in [0, r]
    }.
\end{equation}
The curve $\Gamma^{\lambda}_{\theta, v}$ is referred to as the core of $T_{\theta, v}$. This finishes our wave packet decomposition for a ball of radius $r$ centered at the origin. \\

Next, let us define the wave packet decomposition for a ball not centered at the origin. Fix $\bfx_0\in B(0, \lambda)$ and consider the ball $B(\bfx_0, r)$. For $g: B^{n-1}\to \C$ integrable, define $$\tilde{g}(\omega):=e^{2 \pi i \phi^{\lambda}(\bfx_0; \omega)} g(\omega)$$ so that
$$
T^{\lambda} g(\bfx)=\tilde{T}^{\lambda} \tilde{g}(\tilde{\bfx}) \quad \text { for } \tilde{\bfx}=\bfx-\bfx_0,
$$
where $\widetilde{T}^{\lambda}$ is the H\"ormander-type operator with phase $\tilde{\phi}^{\lambda}$ and amplitude $\tilde{a}^{\lambda}$ given by
$$
\tilde{\phi}(\bfx; \omega):=\phi\left(\bfx+\frac{\bfx_0}{\lambda} ; \omega\right)-\phi\left(\frac{\bfx_0}{\lambda} ; \omega\right) \text { and } \tilde{a}(\bfx; \omega):=a\left(\bfx+\frac{\bfx_0}{\lambda} ; \omega\right) .
$$
If $\bfx\in B(\bfx_0, r)$, then $\tilde{\bfx}\in B(0, r)$, and we can therefore apply the wave packet decomposition above to $\tilde{T}^{\lambda} \tilde{g}$. Moreover, notice that the core curve of $\tilde{T}^{\lambda} (\tilde{g})_{\theta, v}(\tilde{\bfx})$ is given by the collection of $\tilde{\bfx}\in B(0, r)$ satisfying
\begin{equation}
    \partial_{\omega}\phi\left(\frac{\tilde{\bfx}}{\lambda}+\frac{\bfx_0}{\lambda}; \omega_{\theta}\right)=\frac{v}{\lambda}+\partial_{\omega}\phi\left(\frac{\bfx_0}{\lambda}; \omega_{\theta}\right).
\end{equation}
Set 
\begin{equation}
    v^{\lambda}(\bfx_0; \omega):=\partial_{\omega}\phi^{\lambda}\left(\bfx_0; \omega\right)
\end{equation}
Under this notation, the core curve of $\tilde{T}^{\lambda} (\tilde{g})_{\theta, v}(\tilde{\bfx})$ can be written as the image of the map
\begin{equation}
    \Gamma^{\lambda}_{\theta, v+v^{\lambda}(\bfx_0; \omega_{\theta})}=\left(\gamma^{\lambda}_{\theta, v+v^{\lambda}(\bfx_0; \omega_{\theta})}(t), t \right)
\end{equation}
with $t\in [0, \lambda]$. Define curved tubes
\begin{equation}
    T_{\theta, v}(\bfx_0):=\bfx_0+\set{(x, t): |x-\gamma^{\lambda}_{\theta, v+v^{\lambda}(\bfx_0; \omega_{\theta})}(t)|\le r^{1/2+\delta}, \ t\in [0, r]
    }.
\end{equation}
Thus the function 
\begin{equation}
    \tilde{T}^{\lambda}(\tilde{g})_{\theta, v}(\bfx-\bfx_0)
\end{equation}
is essentially supported on $T_{\theta, v}(\bfx_0)$ if we restrict $t\in [t_0, t_0+r]$ where $\bfx_0=(x_0, t_0)$. We will use $\T[B(\bfx_0, r)]$ to denote the collection $\{T_{\theta, v}(\bfx_0)\}_{(\theta, v)}$. Moreover, we write $\theta(T)=\theta$ for a tube $T=T_{\theta, v}(\bfx_0)$. To simplify notation, we also define 
\begin{equation}
    g_{T_{\theta, v}(\bfx_0)}(\omega):=e^{-2\pi i\phi^{\lambda}(\bfx_0; \omega)} (\tilde{g})_{\theta, v}(\omega).
\end{equation}
Under this notation, we can write 
\begin{equation}
    T^{\lambda} g(\bfx)=\sum_{T\in \T[B(\bfx_0, r)]} T^{\lambda} g_T(\bfx). 
\end{equation}
This finishes our wave packet decomposition associated to the ball $B(\bfx_0, r)$.

\subsection{Reducing to broad term estimates}

We define Gauss maps and rescaled Gauss maps. Define 
\begin{equation}
    G_0(\bfx; \xi):=\partial_{\xi_1}\nabla_{\bfx}\phi\wedge \dots \wedge \partial_{\xi_{n-1}}\nabla_{\bfx}\phi.
\end{equation}
Moreover, define 
\begin{equation}
    G(\bfx; \xi):=\frac{G_0(\bfx; \xi)}{|G_0(\bfx; \xi)|}.
\end{equation}
Define the rescaled Gauss map 
\begin{equation}
    G^{\lambda}(\bfx; \xi):=G(\bfx/\lambda; \xi).
\end{equation}
Let $K\ge 1$.   We divide $B^{n-1}$ into caps $\tau$ of side length $K^{-1}$. Let $g_{\tau}$ denote $g\cdot \mathbbm{1}_{\tau}$. For $\bfx\in B_{R}$, denote 
\begin{equation}
    G^{\lambda}(\bfx; \tau):=\{G^{\lambda}(\bfx; \xi): \xi\in \tau\}. 
\end{equation}
Let $V\subset \R^n$ be a linear subspace. Let $\ang(G^{\lambda}(\bfx; \tau), V)$ denote the smallest angle between any non-zero vector $v\in V$ and $v'\in G^{\lambda}(\bfx; \tau)$. Moreover, we say that $\tau\notin_{\bfx, K}V$ if $\ang(G^{\lambda}(\bfx; \tau), V)\ge K^{-1}$; otherwise, we say $\tau\in_{\bfx, K} V$. If  $\bfx$ and $K$ are clear from the context, we often abbreviate $\tau\notin_{\bfx, K} V$ to $\tau\notin V$. 
Next, let us introduce the notion of broad norms.  Fix $B_{K^2}\subset B_R$ centered at $\bfx_0$. Define 
\begin{equation}
    \mu_{T^{\lambda} g}(B_{K^2}):=\min_{V_1, \dots, V_A\in \text{Gr}(k-1, n)} \Big(\max_{\substack{\tau\notin V_a\\ \text{ for any } 1\le a\le A}} \|T^{\lambda} g_{\tau}\|_{L^p(B_{K^2})}^p\Big).
\end{equation}
Here $\mathrm{Gr}(k-1,n)$ is the Grassmannian of all $(k-1)$-dimensional subspaces in $\R^n$, and $k$ is to be determined, and $A$ is a parameter that is less important and its choice will become clear later. For $U\subset \R^n$, define 
\begin{equation}
    \|T^{\lambda} g\|_{\BLka^p(U)}:=\Big(\sum_{B_{K^2}} \frac{|B_{K^2}\cap U|}{|B_{K^2}|} \mu_{T^{\lambda} g}(B_{K^2})\Big)^{1/p}. 
\end{equation}
This is called the broad part of $\operat g$. 

For $2\le k\le n-1$, denote \begin{equation}
    \Gamma_{\mathrm{HZ}}(n, k):=\frac{n-1}{3}+\frac{k-1}{6}\prod_{i=k}^{n-1} \frac{2i}{2i+1}.
\end{equation}
We will prove 
\begin{theorem}[Broad norm estimate]\label{201204thm5_1}
Let $2n/5\le k\le n/2$, and
\begin{equation}
    p>p_n(k):=2+\frac{1}{
    \Gamma_{\mathrm{HZ}}(n, k)+\frac{n}{10^3}
    }.
\end{equation}
Then for every $\epsilon>0$, there exists $A$ such that 
\begin{equation}
\label{main-esti}
    \|T^{\lambda} g\|_{\BLka^p(B_{R})}\lesim_{K, \epsilon} R^{\epsilon} \|g\|_{L^2}^{2/p}\|g\|_{L^{\infty}}^{1-2/p},
\end{equation}
for every $K\ge 1$, $1\le R\le \lambda$, where $B_{R}\subset B_{\lambda}$ is a ball of radius $R$ . Moreover, the implicit constant depends polynomially on $K$. 
\end{theorem}

Recall that when $\phi(\bfx; \xi)=\inn{x}{\xi}+t|\xi|^2$, Hickman and Zahl \cite{hickman2020note} proved \eqref{main-esti} for all 
\begin{equation}
    p>2+\frac{1}{\Gamma_{\mathrm{HZ}}
    (n, k)}.
\end{equation}
Theorem \ref{201204thm5_1} provides a slight improvement in this case. When $k$ is outside the range $[2n/5, n/2]$, we also obtain improved broad norm estimates. As they are irrelevant for the asymptotic formula in \eqref{220717e1_13}, we do not state them here.

It is standard in the literature to reduce Theorem \ref{220704theorem3_1} to Theorem \ref{201204thm5_1}. For instance, by Proposition 11.1 in Guth-Hickman-Iliopoulou's work \cite{MR4047925}, if 
\begin{equation}
    2+\frac{4}{2n-k}\le p\le 2+\frac{2}{k-2},
\end{equation}
then Theorem \ref{201204thm5_1} (for some fixed $k$) implies Theorem \ref{220704theorem3_1} for the same $p$. To see the asymptotic formula in \eqref{220717e1_13}, we set $k=\nu n$ and 
\begin{equation}
    p_n(k)=2+\frac{4}{2n-k},
\end{equation}
solve $\nu$, and then obtain 
\begin{equation}
    q_{n, 2}=2+\frac{4}{2-\nu}\frac{1}{n}+O(n^{-2}).
\end{equation}
We refer to the appendix of \cite{hickman2020note} on how to control $\Gamma_{\mathrm{HZ}}(n, k)$.\\

When proving Theorem \ref{201204thm5_1}, we will apply a wave packet decomposition
\begin{equation}
    g=\sum_{T\in \T[B_R]} g_T.
\end{equation}
By pigeonholing, we can assume that $\norm{g_T}_2 \simeq \norm{g_{T'}}_2$ for every $T$ and $T'$. \footnote{This will be used in a counting lemma below (Lemma \ref{lem: counting}). }\\

\section{Polynomial partitioning}\label{220717section5}

\subsection{Preparatory work.} 

In this subsection, we state a few definitions that will be useful in the forthcoming polynomial partitioning algorithms. 
\begin{definition}[Transverse complete intersection]\label{221010defi5_1}
Let $P_{1}, \ldots, P_{n-m}: \mathbb{R}^{n} \rightarrow \mathbb{R}$ be polynomials. We consider the common zero set
\begin{equation}\label{220223e6.1}
    Z\left(P_{1}, \ldots, P_{n-m}\right):=\left\{x \in \mathbb{R}^{n}: P_{1}(x)=\cdots=P_{n-m}(x)=0\right\} .
\end{equation}
Suppose that for all $z \in Z\left(P_{1}, \ldots, P_{n-m}\right)$, one has
$$
\bigwedge_{j=1}^{n-m} \nabla P_{j}(z) \neq 0 .
$$
Then a connected branch of this set, or a union of connected branches of this set, is called an m-dimensional transverse complete intersection. Given a set $Z$ of the form \eqref{220223e6.1}, the degree of $Z$ is defined by
$$
\min \left(\prod_{i=1}^{n-m} \operatorname{deg}\left(P_{i}\right)\right)
$$
where the minimum is taken over all possible representations of $Z=Z\left(P_{1}, \ldots, P_{n-m}\right)$.
\end{definition}

\begin{definition}[Tangent tubes]\label{220613definition6_2}
Recall the parameters in \eqref{admissable_parameters}. Let $r \geqslant 1$ and $Z$ be an $m$-dimensional transverse complete intersection. A tube $T_{\theta, v}\left(\mathbf{x}_{0}\right) \in \mathbb{T}\left[B\left(\mathbf{x}_{0}, r\right)\right]$ is said to be $r^{-1 / 2+\delta_{m}}$-tangent to $Z$ in $B\left(\mathbf{x}_{0}, r\right)$ if it satisfies
\begin{enumerate}
    \item $T_{\theta, v}\left(\mathbf{x}_{0}\right) \subset \mc{N}_{r^{1 / 2+\delta_{m}}}(Z) \cap B\left(\mathbf{x}_{0}, r\right)$;
    \item For every $\bfz \in Z \cap B\left(\mathbf{x}_{0}, r\right)$, if there is $\bfy \in T_{\theta, v}\left(\mathbf{x}_{0}\right)$ with $|\bfz-\bfy| \lesssim r^{1 / 2+\delta_{m}}$, then one has
$$
\measuredangle \left(G^{\lambda}(\bfy; w_{\theta}), T_{\bfz} Z\right) \lesssim r^{-1 / 2+\delta_{m}} .
$$
Here, $T_{\bfz} Z$ is the tangent space of $Z$ at $\bfz$.
\end{enumerate}
\end{definition}

\begin{definition}
Given a function $f: B^{n-1}\to \C$, we say that it is concentrated on wave packets from $\W$ if 
\begin{equation}
    \norm{\sum_{T\notin \W} f_T}_{\infty} \lesim R^{-100n} \norm{f}_2.
\end{equation}
Here $R$ is from Theorem \ref{201204thm5_1}. 
\end{definition}

\subsection{Partitioning algorithms: Part I}\label{220607subsection6_2}

In this subsection, we run the first part of the polynomial partitioning algorithm. It is a variant of the algorithm in Hickman-Rogers' work \cite{HR2019} with two main differences. 

The first difference is that, after reaching an algebraic dominant case, we will not compare contributions from the tangential case and the transverse case, but instead keep both terms and continue to do polynomial partitioning for both terms. This will be needed in Section \ref{220706section6} when we construct brooms. 

Let us explain the second difference. In the first algorithm in \cite[page 247]{HR2019}, the authors there did not need to control how fast cells shrink. In other words, each time when they see a cellular case, they simply decrease the radius parameter $\rho_j$ by a factor of 2 (see for instance equation (31) in \cite[page 253]{HR2019}). If in the current paper we simply repeat their algorithm, then we will not have good control of the non-admissible parameters $D_n, D_{n-1}, \dots$ by $R$ (see Lemma \ref{220708lemma5_9} below), which was not needed in \cite{HR2019} and is crucial in our inductive argument in Section \ref{220706section8}. 

In order to control how fast cells shrink, in Lemma \ref{partitioninglemma} we require cells to have diameter at most $r/d$, instead of $r/2$. This change also brings in changes in how the algorithm runs. For instance, in the last equation in \cite[page 255]{HR2019}, the authors there simply let $d^{-\delta}$ absorb the constant $2$ from the above $r/2$, and this steps needs to be done more carefully in the current paper as $d^{-\delta}$ certainly can not absorb the factor $d$. \footnote{This will be addressed at the end of Subsection \ref{220610subsection6_4}.}\\

By pigeonholing, we can find a collection $\B_{K^2}$ of balls of radius $K^2$ such that 
\begin{equation}\label{220706e5_3}
    \frac{1}{2}\norm{T^{\lambda} g}_{\BLka^p(B'_{K^2})}\le \norm{T^{\lambda} g}_{\BLka^p(B_{K^2})}\le 2 \norm{T^{\lambda} g}_{\BLka^p(B'_{K^2})}
\end{equation}
for two arbitrary $B_{K^2}, B'_{K^2}\in \B_{K^2}$ and 
\begin{equation}
    \norm{T^{\lambda} g}_{\BLka^p(B_R)}^p\lesim (\log R)^{10} \sum_{B_{K^2}\in \B_{K^2}} \norm{T^{\lambda} g}_{\BLka^p(B_{K^2})}^p.
\end{equation}
Denote 
\begin{equation}
    \mfy=\bigcup_{B_{K^2}\in \B_{K^2}} B_{K^2}.
\end{equation}
Next, we apply polynomial partitioning to $T^{\lambda} g$ restricted to $\mfy$. 
\begin{lemma}[Polynomial partitioning, Guth \cite{guth2018}, Hickman and Rogers \cite{HR2019}]\label{partitioninglemma}
Fix $r \gg 1, d \in \mathbb{N}$ and suppose $F \in L^{1}(\R^n)$ is non-negative and supported on $B_r\cap \mc{N}_{r^{1/2+\delta_{\circ}}}(Z)$ for some $0<\delta_{\circ}\ll 1$, where $Z$ is an  $m$-dimensional transverse complete intersection of degree at most $d$. At least one of the following cases holds:\\

\noindent \underline{Cellular case.} There exists a polynomial $P: \mathbb{R}^{n} \rightarrow \mathbb{R}$ of degree $O(d)$ with the following properties:
\begin{itemize}
    \item[(1)] $\#\cell(P) \simeq d^{m}$ and each $O' \in \operatorname{cell}(P)$ has diameter at most $r/d$. 
    \item[(2)] If we define 
\begin{equation}\label{220706e5_6}
    \mc{O}:=\{O'\setminus \mc{N}_{r^{1/2+\delta_{\circ}}}: O'\in \cell(P)\},
\end{equation}
then 
\begin{equation}
    \int_{O} F \simeq d^{-m} \int_{\mathbb{R}^{n}} F \quad \text { for all } O \in \mathcal{O} .
\end{equation}
\end{itemize}
\noindent \underline{Algebraic case.} There exists an $(m-1)$-dimensional transverse complete intersection Y of degree at most $O(d)$ such that
$$
\int_{B_{r} \cap \mc{N}_{r^{1/2+\delta_{\circ}}}(Z)} F \lesssim \int_{B_{r} \cap \mc{N}_{r 1 / 2+\delta_{\circ}}(Y)} F
$$
\end{lemma}
Here the diameter of a cell in Lemma \ref{partitioninglemma} is $O(r/d)$ instead of $O(r/2)$ as in \cite{guth2018} and \cite{HR2019}. See the proof sketch of Theorem 2.12 in \cite{wang2018restriction} for a discussion.\\

We now start our polynomial partitioning algorithm. This algorithm will produce a tree consisting of many nodes. Each node will have zero child (algorithm for that node stops), one child (cellular case) or two children (algebraic case). For two nodes $\mf{n}$ and $\mf{n}'$, if $\mf{n}$ is a descendant of $\mf{n}'$, then we write $\mf{n}\preccurlyeq \mf{n}'$; similarly we define $\succcurlyeq.$ Here we make the convention that $\mf{n}\preccurlyeq \mf{n}$ and $\mf{n}\succcurlyeq \mf{n}$.   Recall the parameters in \eqref{admissable_parameters}. Moreover, define $\tilde{\delta}_{m-1}$ by 
\begin{equation}
    (1-\tilde{\delta}_{m-1})(1/2+\delta_{m-1})=1/2+\delta_m.
\end{equation}
Note that $\delta_{m-1}/2\le \tilde{\delta}_{m-1}\le 2\delta_{m-1}$. \\

\noindent \underline{Step 0.} In this step, we create the root of the tree.  Denote 
\begin{equation}
    \mathfrak{n}_0=\{O_{i_0}\},
\end{equation}
with $O_{i_0}=B_R\cap \mfy$. Moreover, define $\dim(\mathfrak{n}_0)=n$,  $\rho(\mathfrak{n}_0)=R$ and $\mf{j}(\mf{n}_0)=0$. Later we will define $\mf{j}$ for every node. It will play the role of the parameter $j$ in the recursive step of the first algorithm in \cite[page 247]{HR2019}. In the end of this step, define
\begin{equation}
    \#_a(\mf{j}(\mf{n}_0))=0, \ \ \#_c(\mf{j}(\mf{n}_0))=0.
\end{equation}
Here ``a" is short for ``algebraic" and ``c" is short of ``cellular", and we use $\#_a$ to record the number of algebraic cases we have so far when applying Lemma \ref{partitioninglemma}, and similarly, we use $\#_c$ to record the number of cellular cases. Initialize
\begin{equation}\label{220707e5_11a}
    \mf{M}'_{\ell'}=\emptyset,  \ \  \forall \ell'\in \N.
\end{equation}
This collection will appear at almost the end of the algorithm; we will keep adding elements to it as we run the algorithm. \\

\noindent\underline{Step 1.} Creating nodes at the first level. The root node $\mathfrak{n}_0$ has either one or two children, given as follows. Apply Lemma \ref{partitioninglemma} to the function $T^{\lambda} g\cdot \mathbbm{1}_{O_{i_0}}$, we obtain a collection of cells $\{O_{i_1}\}_{i_1}$ (as in \eqref{220706e5_6}) and a wall $W=\mc{N}_{(\rho(\mathfrak{n}_0))^{1/2+\delta_{m}}}(Z)$ with $m=\dim(\mathfrak{n}_0)$ for some variety $Z$. We without loss of generality assume that all these regions are unions of balls $B_{K^2}$. Compare 
\begin{equation}\label{220531e6_6}
    \#\{B_{K^2}\in \B_{K^2}: B_{K^2}\subset \bigcup_{i_1} O_{i_1}\} \text{ and } \#\{B_{K^2}\in \B_{K^2}: B_{K^2}\subset W\}.
\end{equation}
If the former term in \eqref{220531e6_6} is larger, then we say that we are in the cellular case of this step, and otherwise we say that we are in the algebraic case. In the cellular case, the node $\mf{n}_0$ has only one child, and will be denoted by $\mf{n}_{1, L}=\mf{n}_{1, L}(\mf{n}_0)$ and called the L-child; here $L$ refers to ``left". Define
\begin{equation}
    \rho(\mf{n}_{1, L})=\rho(\mf{n}_0)/d, \ \ \mf{n}_{1, L}=\{O_{i_1}\}_{i_1}, \ \ \dim(\mf{n}_{1, L})=m,
\end{equation}
and 
\begin{align}
& \mf{j}(\mf{n}_{1, L})=\mf{j}(\mf{n}_0)+1,\\ 
    & \#_c(\mf{j}(\mf{n}_{1, L}))=\#_c(\mf{j}(\mf{n}_{0}))+1, \ \ \#_a(\mf{j}(\mf{n}_{1, L}))=\#_a(\mf{j}(\mf{n}_{0})).
\end{align}
In other words, we have had one cellular cases so far, and zero algebraic cases.  Define $\mf{L}_1=\{\mf{n}_{1, L}\}$, that is, we use $\mf{L}_1$ to collect all the L-children in this step. Moreover, set $\mf{M}_1=\mf{R}_1=\emptyset$. This finishes defining the node $\mf{n}_{1, L}$ and its information. \\

If we are in the algebraic case, then the node $\mf{n}_0$ has two children, and will be denoted by $\mf{n}_{1, M}=\mf{n}_{1, M}(\mf{n}_0)$ and $\mf{n}_{1, R}=\mf{n}_{1, R}(\mf{n}_0)$ and be called the M-child and the R-child; here $M$ refers to ``middle", and $R$ to ``right".  Define 
\begin{equation}
     \rho(\mathfrak{n}_{1, M})=\rho(\mathfrak{n}_{1, R})=\rho(\mathfrak{n}_0)^{1-\tilde{\delta}_{m-1}}, 
\end{equation}
and
\begin{equation}
     \dim(\mathfrak{n}_{1, M})=n, \ \dim(\mathfrak{n}_{1, R})=n-1, \ \ \mathfrak{n}_{1, M}=\mathfrak{n}_{1, R}=\{O_{i'_1}\}_{i'_1},
\end{equation}
where each $O_{i'_1}$ is given by $W\cap B_{\rho(\mathfrak{n}_{1, M})}$ and we let $B_{\rho(\mathfrak{n}_{1, M})}$ run through all balls of radii $\rho(\mathfrak{n}_{1, M})$ inside $B_{\rho(\mathfrak{n}_0)}$. Moreover, define
\begin{align}
    & \mf{j}(\mf{n}_{1, M})=\mf{j}(\mf{n}_0)+1,\\
    & \#_c(\mf{j}(\mf{n}_{1, L}))=\#_c(\mf{j}(\mf{n}_{0})), \ \ \#_a(\mf{j}(\mf{n}_{1, L}))=\#_a(\mf{j}(\mf{n}_{0}))+1,
\end{align}
and 
\begin{equation}\label{220717e5_22}
    \mf{j}(\mf{n}_{1, R})=0, \ \ \#_c(\mf{j}(\mf{n}_{1, R}))=\#_a(\mf{j}(\mf{n}_{1, R}))=0.
\end{equation}
Here let us explain the rule of defining $\mf{j}$: It is always reset to be $0$ whenever we see an $R$-child, and otherwise its values is increased by $1$. 
% Define 
% \begin{equation}
%     \mathfrak{L}_1=\{\mathfrak{n}_{1, L}\}, \mathfrak{M}_1=\{\mathfrak{n}_{1, M}\}, \mathfrak{R}_1=\{\mathfrak{n}_{1, R}\},
% \end{equation}

Let $\mathfrak{M}_1, \mathfrak{R}_1$ collect all the M-children and R-children at Step 1, respectively. As $\mf{n}_0$ has no L-child, we set $\mf{L}_1=\emptyset$. 
This finishes the first step. \\

\noindent \underline{Step 2.} Creating nodes at the second level. Take a node $\mathfrak{n}_{1}$ from the previous step. It has one or two children. 

If $\mathfrak{n}_{1}\in \mathfrak{L}_1$ or $\mathfrak{M}_1$, then its children, which will be named either as $\mathfrak{n}_{2, L}$ or as  $\mathfrak{n}_{2, M}, \mathfrak{n}_{2, R}$, will be given as follows. For each $O_{i_1}\in \mathfrak{n}_1$, we apply Lemma \ref{partitioninglemma} with dimension parameter $\dim(\mathfrak{n}_1)$, and obtain a collection of cells $\{O_{i_2}\}_{i_2}$ and a wall $W_{m-1}=\mc{N}_{(\rho(\mathfrak{n}_1))^{1/2+\delta_m}}(Z_{m-1})$ with $m=\dim(\mathfrak{n}_1)$ for some variety $Z_{m-1}$. We make a comparison similar to \eqref{220531e6_6}. If we are in the cellular case, then define
\begin{equation}\label{220707e5_20}
    \rho(\mathfrak{n}_{2, L})=\rho(\mathfrak{n}_1)/d, \ \mathfrak{n}_{2, L}=\bigcup_{O_{i_1}\in \mf{n}_1}\{O_{i_2}\}_{i_2}, \ \dim(\mathfrak{n}_{2, L})=\dim(\mf{n}_1),
\end{equation}
and 
\begin{align}\label{220707e5_21}
    & \mf{j}(
    \mf{n}_{2, L}
    )=\mf{j}(
    \mf{n}_{1}
    )+1, \\
    & \#_c(\mf{j}(\mf{n}_{2, L}))=\#_c(\mf{j}(\mf{n}_{1}))+1, \ \ \#_a(\mf{j}(\mf{n}_{2, L}))=\#_a(\mf{j}(\mf{n}_{1})).
\end{align}
If we are in the algebraic case, then define 
\begin{equation}
    \rho(\mathfrak{n}_{2, M})=\rho(\mathfrak{n}_{2, R})=\rho(\mathfrak{n}_1)^{1-\tilde{\delta}_{m-1}}, 
\end{equation}
and
\begin{align}
         & \dim(\mathfrak{n}_{2, M})=\dim(\mf{n}_1), \ \dim(\mathfrak{n}_{2, R})=\dim(\mf{n}_1)-1,\\
         & \mathfrak{n}_{2, M}=\mathfrak{n}_{2, R}=\bigcup_{O_{i_1}\in \mf{n}_1}\{O_{i'_2}\}_{i'_2},
\end{align}
where each $O_{i'_2}$ is given by $W_{m-1}\cap B_{\rho(\mathfrak{n}_{2, M})}$ and we let $B_{\rho(\mathfrak{n}_{2, M})}$ run through all balls of radius $\rho(\mathfrak{n}_{2, M})$ inside $B_{\rho(\mathfrak{n}_1)}$. Moreover, define 
\begin{align}\label{220707e5_26}
    & \mf{j}(\mf{n}_{2, M})=\mf{j}(\mf{n}_1)+1,\\
    & \#_c(\mf{j}(\mf{n}_{2, L}))=\#_c(\mf{j}(\mf{n}_{1})), \ \ \#_a(\mf{j}(\mf{n}_{2, L}))=\#_a(\mf{j}(\mf{n}_{1}))+1,
\end{align}
and 
\begin{equation}\label{220707e5_28}
    \mf{j}(\mf{n}_{2, R})=0, \ \ \#_c(\mf{j}(\mf{n}_{2, R}))=\#_a(\mf{j}(\mf{n}_{2, R}))=0.
\end{equation}
Let $\mf{L}_2$ collection all the $L$-children at Step 2, and similarly, we define $\mf{M}_2$ and $\mf{R}_2$. 

Next, consider the remaining case $\mf{n}_1\in \mf{R}_1$. Its children are given as follows. For each $O_{i_1}\in \mf{n}_1$, we apply Lemma \ref{partitioninglemma} with dimension parameter $m=\dim(\mf{n}_1)$ and $\delta_{\circ}=\delta_m$, and obtain a collection of cells $\{O_{i_2}\}_{i_2}$ and a wall $W_{m-1}=\mc{N}_{(\rho(\mathfrak{n}_1))^{1/2+\delta_m}}(Z_{m-1})$ for some variety $Z_{m-1}$. If we are in the cellular case, then define $\rho(\mathfrak{n}_{2, L}),  \mathfrak{n}_{2, L},  \dim(\mathfrak{n}_{2, L})$ in the same way as in \eqref{220707e5_20} and $\mf{j}(
    \mf{n}_{2, L}
    ), \#_c(\mf{j}(\mf{n}_{2, L})), \#_a(\mf{j}(\mf{n}_{2, L}))$ in the same way as in \eqref{220707e5_21}. If we are in the algebraic case, then define 
\begin{equation}\label{220706e5_20}
     \rho(\mathfrak{n}_{2, M})=\rho(\mathfrak{n}_{2, R})=\rho(\mathfrak{n}_1)^{1-\tilde{\delta}_{m-1}}, 
\end{equation}
and
\begin{align}\label{220706e5_21}
     &\dim(\mathfrak{n}_{2, M})=\dim(\mf{n}_1), \ \dim(\mathfrak{n}_{2, R})=\dim(\mf{n}_1)-1,\\
    &  \mathfrak{n}_{2, R}=\bigcup_{O_{i_1}\in \mf{n}_1}\{O_{i'_2}\}_{i'_2},
\end{align}
where each $O_{i'_2}$ is given by $W_{m-1}\cap B_{\rho(\mathfrak{n}_{2, M})}$ and we let $B_{\rho(\mathfrak{n}_{2, M})}$ run through all balls of radius $\rho(\mathfrak{n}_{2, M})$ inside $B_{\rho(\mathfrak{n}_1)}$. Our choice of parameters guarantees that 
\begin{equation}
    \rho(\mf{n}_{2, R})^{1/2+\delta_{m-1}}=\rho(\mf{n}_1)^{1/2+\delta_m}. 
\end{equation}
We still need to define $\mf{n}_{2, M}$. Roughly speaking, for each $O_{i'_2}$ given by $W_{m-1}\cap B_{\rho(\mathfrak{n}_{2, M})}$, which is of thickness $\rho(\mf{n}_1)^{1/2+\delta_m}$, we will cut it into thinner layers $W_{m-1, b}$ of thickness $\rho(\mf{n}_{2, M})^{1/2+\delta_m}$. Then we set
\begin{equation}
    \mf{n}_{2, M}=\bigcup_{O_{i_1}\in \mf{n}_1} \bigcup_{O_{i'_2}} \{O_{i'_2}\cap W_{m-1, b}\}_b.
\end{equation}
To make this precise, we follow Hickman-Rogers' treatment \cite{HR2019}. For each $B_{\rho(\mf{n}_{2, M})}$, we follow page 258 in \cite{HR2019}, find a finite set of translates $\mf{B}\subset B(0, \rho(\mf{n}_{1})^{1/2+\delta_m})$ and then set (following the last equation in \cite[page 258]{HR2019})
\begin{equation}\label{220706e5_24}
    \mf{n}_{2, M}=
    \bigcup_{O_{i_1}\in \mf{n}_1}
    \bigcup_{O_{i'_2}} \{O_{i'_2}\cap \mc{N}_{\rho(\mf{n}_{2, M})^{1/2+\delta_m}}(Z_{m-1}+b): b\in \mf{B}\}.
\end{equation}
In the end, define $\mf{j}(\mf{n}_{2, M}), \mf{j}(\mf{n}_{2, R}), \#_a, \#_c$ in the same way as in \eqref{220707e5_26}--\eqref{220707e5_28}. 

Let $\mf{L}_2$ collection all the $L$-children at Step 2, and similarly, we define $\mf{M}_2$ and $\mf{R}_2$. This finishes Step 2.\\

\noindent \underline{Step $\ell$.} Creating nodes at the $\ell$-th level. How we proceed in a general step is similar to what we did in Step 2, with one difference mentioned at the beginning of this section that we need to control how fast cells shrink. We will sketch the part in this step that is similar to Step 2, and explain in more details the difference. \\

Take a node $\mf{n}_{\ell-1}$ from the previous step. There are a few parameters associated to it: A dimension parameter $\dim(\mf{n}_{\ell-1})=:m$, a radius parameter $\rho(\mf{n}_{\ell-1})$, the parameters $\mf{j}(\mf{n}_{\ell-1})$, $\#_c(\mf{j}(\mf{n}_{\ell-1}))$ and $\#_a(\mf{j}(\mf{n}_{\ell-1}))$ satisfying
\begin{equation}
    \mf{j}(\mf{n}_{\ell-1})=\#_c(\mf{j}(\mf{n}_{\ell-1}))+ \#_a(\mf{j}(\mf{n}_{\ell-1})).
\end{equation}
Before we proceed, we need to introduce new notation. Let $\mf{n}_{\ell-1}^{\uparrow}$ denote the closest ancestor (itself included) to $\mf{n}_{\ell-1}$ that belongs to $\mf{M}_{\ell'}$, $\mf{M}'_{\ell'}$ or $\mf{R}_{\ell'}$ for some $\ell'$. (Recall the initialization of $\mf{M}'_{\ell'}$ in \eqref{220707e5_11a}. ) Note that 
\begin{equation}
    \dim(
    \mf{n}_{\ell-1}
    )
    =\dim(
    \mf{n}_{\ell-1}^{\uparrow}
    ).
\end{equation}
For each $O_{i_{\ell-1}}\in \mf{n}_{\ell-1}$, we apply Lemma \ref{partitioninglemma} with dimension parameter $m$ and $\delta_{\circ}=\delta_{\circ}(\mf{n}_{\ell-1})$ satisfying
\begin{equation}\label{220709e5_38}
    \rho(\mf{n}_{\ell-1})^{\frac{1}{2}+\delta_{\circ}}=
    \rho(\mf{n}_{\ell-1}^{\uparrow})
    ^{\frac{1}{2}+\delta_{m}}
\end{equation}
and obtain a collection of cells $\{O_{i_{\ell}}\}_{i_{\ell}}$ and a wall
\begin{equation}\label{220709e5_39}
    W_{m-1}=\mc{N
}_{
(\rho(\mf{n}^{\uparrow}_{\ell-1}))^{1/2+\delta_m}
}(Z_{m-1})
\end{equation}
for some variety $Z_{m-1}$.  If we are in the algebraic case, then $\mf{n}_{\ell-1}$ has two children, called $\mf{n}_{\ell, M}$ and $\mf{n}_{\ell, R}$, and similarly to \eqref{220706e5_20}--\eqref{220706e5_24}, we define 
\begin{equation}\label{220710e5_40}
    \rho(\mf{n}_{\ell, M})=\rho(\mf{n}_{\ell, R})=\rho(\mf{n}_{\ell-1})^{1-\tilde{\delta}_{m-1}}, \ \mf{n}_{\ell, R}=\bigcup_{O_{i_{\ell-1}}\in \mf{n}_{\ell-1}} \{O_{i'_{\ell}}\}_{i'_{\ell}};
\end{equation}
moreover, define 
\begin{equation}\label{220707e5_39a}
    \mf{n}_{\ell, M}=\bigcup_{
    O_{i_{\ell-1}}\in \mf{n}_{\ell-1}
    }
    \bigcup_{
    O_{i'_{\ell}}
    } 
    \{
    O_{i'_{\ell}}\cap 
    \mc{N}_{
    \rho(\mf{n}_{\ell, M})^{1/2+\delta_m}
    }
    (Z_{m-1}+b)
    :
    b\in \mf{B}
    \},
\end{equation}
where $\mf{B}$ is a finite set of points in $B(0, 
\rho(
\mf{n}_{\ell-1}
)^{1/2+\delta_{\circ}}
)$, and 
\begin{equation}
    \dim(\mf{n}_{\ell, M})=\dim(\mf{n}_{\ell-1}), \ \dim(\mf{n}_{\ell, R})=\dim(\mf{n}_{\ell-1})-1.
\end{equation}
If we are in the cellular case,
% then $\mf{n}_{j-1}$ has only one child, called $\mf{n}_{j, L}$ and similarly to \eqref{220706e5_19}, we define 
% \begin{equation}
%     \rho(\mf{n}_{j, L})=\rho(\mf{n}_{j-1})/d, \mf{n}_{j, L}=\bigcup_{O_{i_{j-1}}\in \mf{n}_{j-1}} \{O_{i_j}\}_{i_j}, \ \dim(\mf{n}_{j, L})=\dim(\mf{n}_{j-1}). 
% \end{equation}
then we proceed differently. \footnote{Indeed Step 2 also falls into the same framework; the forthcoming difference only occurs when $\ell$ is large. } There are two further cases. If we are  in the case 
\begin{equation}\label{220707e5_39}
    \frac{\rho(\mf{n}_{\ell-1})}{d}\ge \rho(
    \mf{n}^{\uparrow}_{\ell-1}
    )
    ^{1-
    \delta_{m-1/2}
    },
\end{equation}
where $\delta_{m-1/2}$ is as in \eqref{220707e1_12}, then we define $\mf{n}_{\ell, L}, \rho(\mf{n}_{\ell, L}
)$ and $\dim(\mf{n}_{\ell, L}
)$ in the same way as in \eqref{220707e5_20}, and $\mf{j(\mf{n}_{\ell, L}
)
}, \#_c(\mf{j(\mf{n}_{\ell, L}
)
})$ and $\#_a(\mf{j(\mf{n}_{\ell, L}
)
})$ in the same way as in \eqref{220707e5_21}. If \eqref{220707e5_39} is violated, then we first update 
\begin{equation}
    \mf{M}'_{\ell}=\mf{M}'_{\ell}\bigcup \{\mf{n}_{\ell, L}(\mf{n}_{\ell-1})\}.
\end{equation}
The next step is to cut each $O_{i_{\ell}}$ into thinner layers, in a way that is essentially the same as in \eqref{220706e5_24}. Let us be more precise. By the way we run the partitioning algorithm, in particular, due to the choice of the parameter $\delta_{\circ}$ in \eqref{220709e5_38}, we know that 
\begin{equation}
    \rho(
    \mf{n}
    )^{
    \frac{1}{2}+
    \delta_{\circ}
    (\mf{n})
    }=
    \rho(
    \mf{n}_{\ell-1}^{\uparrow}
    )^{\frac{1}{2}+\delta_m}
\end{equation}
for every node $\mf{n}$ with $\mf{n}_{\ell-1}\preccurlyeq \mf{n}\preccurlyeq \mf{n}_{\ell-1}^{\uparrow}$. Therefore we have 
\begin{equation}\label{220707e5_43}
    O_{i_{\ell}}\subset B_{
    \rho
    (
    \mf{n}_{\ell, L}
    )
    }\cap \mc{N}_{
    \rho(
    \mf{n}_{\ell-1}^{\uparrow}
    )^{1/2+\delta_m}
    }(Z_m),
\end{equation}
where
\begin{equation}
    \rho(\mf{n}_{\ell, L}):=\rho(\mf{n}_{\ell-1})/d,
\end{equation}
and similarly as before we define $\rho(\mf{n}_{\ell, L})$ before defining $\mf{n}_{\ell, L}$, and $Z_m$ is an $m$-dimensional variety. Note that as \eqref{220707e5_39} is violated, we have 
\begin{equation}\label{220707e5_45}
    \rho(
    \mf{n}^{\uparrow}_{\ell-1}
    )
    ^{1-
    \delta_{m-1/2}
    }/d \le \rho(\mf{n}_{\ell, L})\le \rho(
    \mf{n}^{\uparrow}_{\ell-1}
    )
    ^{1-
    \delta_{m-1/2}
    }.
\end{equation}
We will cut the right hand side of \eqref{220707e5_43} into thinner layers of thickness $\rho(\mf{n}_{\ell, L})^{1/2+\delta_m}$ via transverse equidistribution properties (for instance Lemma 8.4 in  \cite{MR4047925}). This can be done in exactly the same way as in \eqref{220706e5_24} and \eqref{220707e5_39a}, with the only difference that the parameter $\tilde{\delta}_{m-1}$ that appears in the radius $\rho(\mf{n}_{2, M})$ is replaced by $\delta_{m-1/2}$ (which appears in the radius $\rho(\mf{n}_{\ell, L})$ because of the relation \ref{220707e5_45}).  Therefore, we follow \cite[page 258]{HR2019} and find a finite set of translates $\mf{B}\subset B(0, 
\rho(
\mf{n}_{\ell-1}^{\uparrow}
)^{1/2+\delta_m}
)$, and then set 
\begin{equation}\label{220707e5_49z}
\mf{n}_{\ell, L}=\bigcup_{
    O_{i_{\ell-1}}\in \mf{n}_{\ell-1}
    }
    \bigcup_{
    O_{i'_{\ell}}
    } 
    \{
    O_{i'_{\ell}}\cap 
    \mc{N}_{
    \rho(\mf{n}_{\ell, L})^{1/2+\delta_m}
    }
    (Z_{m-1}+b)
    :
    b\in \mf{B}
    \}.
\end{equation}
Moreover, define 
\begin{align}
    & \dim(\mf{n}_{\ell, L})=\dim(\mf{n}_{\ell-1}), \ \  \mf{j}
    (
    \mf{n}_{\ell, L}
    )=\mf{j}
    (
    \mf{n}_{\ell-1}
    )+1,\\
    & \#_c(\mf{j}
    (
    \mf{n}_{\ell, L}
    ))=\#_c(\mf{n}_{\ell-1})+1, \ \ \#_a(\mf{j}
    (
    \mf{n}_{\ell, L}
    ))=\#_a(\mf{n}_{\ell-1}).
\end{align}
In the end, we let $\mf{L}_{\ell}$ collect all the $L$-children at the $\ell$-th level, and similarly we define $\mf{M}_{\ell}$ and $\mf{R}_{\ell}$. This finishes the $\ell$-th step of the algorithm. \\

Before we proceed to the next step, we make a remark on the size of the parameter $\delta_{\circ}$ in \eqref{220709e5_38}. By \eqref{220709e5_38} and \eqref{220707e5_39}, we have 
\begin{equation}
     \rho(
     \mf{n}_{\ell-1}^{\uparrow}
     )^{
     \frac{1+2\delta_m}{1+2\delta_{\circ}}
     }
     =\rho(\mf{n}_{\ell-1})
     \ge d \rho(
    \mf{n}^{\uparrow}_{\ell-1}
    )
    ^{1-
    \delta_{m-1/2}
    },
\end{equation}
which further implies 
\begin{equation}
    \delta_m\le \delta_{\circ}\le \delta_{m-1/2}.
\end{equation}
In other words, $\delta_{\circ}$ is still quite close to $\delta_m$, and is very far from $\delta_{m-1}$. 
\begin{remark}
In the above \eqref{220710e5_40}, each element $O_{i'_{\ell}}$ in $\mf{n}_{\ell, R}$ is given by $B_{
\rho(\mf{n}_{\ell, R})
}
\cap 
W_{m-1},
$
where $B_{
\rho(\mf{n}_{\ell, R})
}$ is a ball of radius $\rho(\mf{n}_{\ell, R})$ and $W_{m-1}$ is given in \eqref{220709e5_39}. Their counterpart in \cite{HR2019} is given by $B\cap \mc{N}_{
\rho_j^{1/2+\delta_m}
}(\bfY)$ at the bottom of \cite[page 256]{HR2019}, where $B$ is a ball of radius $\rho_{j+1}$ and $\rho_{j+1}=\rho(
\mf{n}_{\ell, R}
)$ in our notation, $\rho_j=\rho(
\mf{n}_{\ell-1}
)$, $\bfY=W_{m-1}$ in our notation. The slight difference is that the neighborhood scale
$(\rho(
\mf{n}_{\ell-1}^{\uparrow}
))^{
1/2+\delta_m
}$
in \eqref{220709e5_39} is bigger than $\rho_j^{1/2+\delta_m}$. However, by \eqref{220707e5_39}, 
\begin{equation}
    \frac{
    \rho(\mf{n}_{\ell-1}^{\uparrow})
    }{\rho(
    \mf{n}_{\ell-1}
    )}\le \rho(\mf{n}_{\ell-1}^{\uparrow})^{\delta_{m-1/2}}.
\end{equation}
We will lose about $\delta_{m-1}^{-1}$ many of these multiplicative factors. As $\delta_{m-1/2}\ll_{\epsilon} \delta_{m-1}$, we see that they are harmless. 
\end{remark}

\noindent \underline{Stopping condition.} Suppose we have arrived at the $\ell_0$-th level. Take a node $\mf{n}_{\ell_0}$. The algorithm will not continue at this node (but may still continue at other nodes at the same level) if either $\rho(\mf{n}_{\ell_0})\le R^{\delta_0}$ or $\dim(\mf{n}_{\ell_0})\le k-1$. Here $\delta_0$ is given in \eqref{admissable_parameters}
 and $k$ is given in Theorem \ref{201204thm5_1}. In other words, we will not continue our algorithm if the radius of the node is too small, or the dimension is too small. \\

We state one lemma that will be used later. 
\begin{lemma}\label{220608lemma6_5}
Under the above notation, we have \begin{equation}\label{220717e5_57}
    \#\pnorm{\bigcup_{\iota} (\mf{M}_{\iota}\cup \mf{R}_{\iota}
    )}\lesim_{n, \delta} 1. 
\end{equation}
\end{lemma}
\begin{proof}[Proof of Lemma \ref{220608lemma6_5}]
Note that the left hand side of \eqref{220717e5_57} would not change if we assume that there is no cellular case in the algorithm. In this case, each node has either zero or two children. The total number of levels $\ell_0\le \delta^{-1}$. Moreover, note that the algorithm stops at a node $\mf{n}$ if $\dim(\mf{n})\le k-1$. As a consequence, we see that the left hand side of \eqref{220717e5_57}  is $\le n\delta^{-1}$. 
\end{proof}

\subsection{The related case}

The proof of Theorem \ref{201204thm5_1} relies on a two-ends argument. This requires a relation, denoted by $\sim$, which is defined between tubes $T\in \T[B_R]$ and balls $B_{\iota}\subset B_R$ of radius $R^{1-\delta}$. The definition of $\sim$ is a bit complicated and relies on the definition of brooms; it will be given in Section \ref{220706section6}.  At this point, we only need the fact that 
\begin{equation}\label{220706e5_31}
    \#\{B_{\iota}\subset B_R: B_{\iota}\sim T\}\lesim 1,
\end{equation}
for every tube $T\in \T[B_R]$. 

\begin{definition}
For each ball $B_{\iota}\subset B_R$ of radius $R^{1-\delta}$ and $\bfx\in B_{\iota}$, define 
\begin{equation}
    T^{\lambda} g^{\sim}(\bfx):=\sum_{T\in \T[B_R], T\sim B_{\iota}} T^{\lambda} g_{T}(\bfx),
\end{equation}
and define $T^{\lambda} g^{\not\sim}(\bfx)$ to be the difference of $T^{\lambda} g(\bfx)$ and $T^{\lambda} g^{\sim}(\bfx)$. Moreover, for a given $\iota$, define 
\begin{equation}\label{220610e6_25}
    g^{\not\sim}_{\iota}:=\sum_{T\in \T[B_R], T\not\sim B_{\iota}} g_{T}.
\end{equation}
\end{definition}
Under the above notation, it holds that 
\begin{equation}
    T^{\lambda} g^{\not\sim}(\bfx)=T^{\lambda} g^{\not\sim}_{\iota}(\bfx),
\end{equation}
whenever $\bfx\in B_{\iota}$. Recall the definition of $\B_{K^2}$ at the beginning of Subsection \ref{220607subsection6_2}. Denote 
\begin{equation}
    \B'_{K^2}:=\set{B_{K^2}\in \B_{K^2}: \norm{T^{\lambda} g^{\sim}}_{\BLka^p(B_{K^2})} \le \norm{T^{\lambda} g^{\not\sim}}_{\BLka^p(B_{K^2})}}.
\end{equation}
If $|\B'_{K^2}|\le |\B_{K^2}|/2$, then we say that we are in the related case (of Theorem \ref{201204thm5_1}), and otherwise we say that we are in the non-related case. Because of the pigeonholing step in \eqref{220706e5_3}, if we are in the related case, then the contribution to the broad-norm $\BLka^p(B_{R})$ from $\B_{K^2}\setminus \B'_{K^2}$ is bigger. In this case, we can use the induction hypothesis \eqref{220222e4.5} and the fact that there is only a small number of balls related to each tube, as in \eqref{220706e5_31}, to finish the proofs of Theorem \ref{201204thm5_1} and Theorem \ref{220704theorem3_1}. 
\begin{lemma}
If we are in the related case, then \eqref{main-esti} holds. 
\end{lemma}

The proof of this lemma is a standard induction-on-scales argument, and is the same as that of Lemma 2.20 in \cite{wang2018restriction}. We leave out the proof.

\subsection{Partitioning algorithm: Part II}\label{220610subsection6_4}

The rest of the paper is to handle the case that 
\begin{equation}
    |\B'_{K^2}|\ge |\B_{K^2}|/2
\end{equation}
that is, the unrelated component $T^{\lambda} g^{\not\sim}$ dominates. Recall that we need to bound 
\begin{equation}
    \begin{split}
        & \sum_{B_{K^2}\in \B_{K^2}} \norm{T^{\lambda} g^{\not\sim}}_{\BLka^p(B_{K^2})}^p=\sum_{B_{\iota}} \sum_{B_{K^2}\subset B_{\iota}} \Norm{T^{\lambda} g^{\not\sim}_{\iota}}_{\BLka^p(B_{K^2})}^p.
    \end{split}
\end{equation}
We run the previous algorithm again with $T^{\lambda} g$ replaced by $T^{\lambda} g^{\not\sim}$. Note that in the algorithm in Subsection \ref{220607subsection6_2}, we did not compare contributions between the transverse case and the tangential case, which is exactly because we will further run the algorithm below. In what follows, we often abbreviate $g^{\not\sim}_{\iota}$ to $g_{\iota}$. \\

\noindent \underline{Step 0.} Define $\mf{n}^*_0=\mf{n}_0$.\\

\noindent \underline{Step 1. } We will define three quantities; they correspond to contributions from the cellular case $C(\mathfrak{L}_1)$, the transverse case $C(\mathfrak{M}_1)$ and the tangential case $C(\mathfrak{R}_1)$. Take $B_{\iota}\subset B_R$, a ball of radius $R^{1-\delta}$, and $\mf{n}_1$, a child of $\mf{n}^*_0$ with $\mf{n}_1\in \mc{L}_1$, take $O_{i_1}\in \mf{n}_1$ with $O_{i_1}\subset B_{\rho_1}\subset B_{\iota}$ with $\rho_1=\rho(\mf{n}_1)$ for some $B_{\rho(\mf{n}_1)}$, denote 
\begin{equation}\label{220429e6_24}
    g_{\iota, O_{i_1}}=\sum_{T\in \T[B_R], T\cap O_{i_1}\neq \emptyset}(g_{\iota})_{T}.
\end{equation}
Define
\begin{equation}
    C(\mathfrak{L}_1):= \sum_{O_{i_1}\in \mathfrak{n}_{1}} \sum_{\iota} \Norm{T^{\lambda} g_{\iota, O_{i_1}} }_{\BLka^p(O_{i_1}\cap B_{\iota})}^p.
\end{equation}
If $\mf{n}_0^*$ does not have any children in $\mf{L}_1$, then we simply set $C(\mf{L})_1=0$.

To define the other two quantities, we need more notation. Let $\mathfrak{n}_1$ be a child of $\mf{n}^*_0$ with $\mf{n}_1\in \mf{M}_1$ or $\mf{R}_1$, and take  $O_{i'_1}\in \mathfrak{n}_1$ given by $O_{i'_1}=B_{\rho_1}\cap W$ for some $B_{\rho_1}\subset B_{\rho_0}$, with $\rho_1=\rho(\mf{n}_1)$ and $\rho_0=\rho(\mf{n}^*_0)$. Similarly to \eqref{220429e6_24}, we define $g_{\iota, O_{i'_1}}$.  Let $\T_{O_{i'_1}}$ denote the collection of all $T\in \T[B_{\rho_0}]$  for which 
\begin{equation}
    T\cap B_{\rho_1}\cap W\neq \emptyset. 
\end{equation}
Moreover, we will partition $\T_{O_{i'_1}}$ into two parts
\begin{equation}
    \T_{O_{i'_1}}=\T_{O_{i'_1}, \tang}\bigcup \T_{O_{i'_1}, \trans},
\end{equation}
where 
\begin{equation}\label{220429e6_26}
    \T_{O_{i'_1}, \tang}:=\set{T\in \T_{O_{i'_1}}: T \text{ is } \rho_1^{-\frac{1}{2}+\delta_{m-1}}\text{-tangent to } Z \text{ on } B_{\rho_1}
    },
\end{equation}
where $m=\dim(\mf{n}^*_0)$. We refer the definition of $\T_{O_{i'_1}, \tang}$ to Definition 9.3 in \cite[page 257]{HR2019}; it needs some clarification as $T$ is a wave packet at the scale $\rho_0$ and we are talking about tangency at the smaller scale $\rho_1$.

After defining $\T_{O_{i'_1}, \tang}$, we will just set 
\begin{equation}
    \T_{O_{i'_1}, \trans}:=\T_{O_{i'_1}}\setminus \T_{O_{i'_1}, \tang}.
\end{equation}
Moreover, define 
\begin{equation}\label{220610e6_39}
    g_{\iota, O_{i'_1}, \tang}:=\sum_{T\in \T_{O_{i'_1}, \tang}}(g_{\iota, O_{i'_1}})_{T}
\end{equation}
and 
\begin{equation}\label{220610e6_40}
    g_{\iota, O_{i'_1}, \trans}:=\sum_{T\in \T_{O_{i'_1}, \trans}}(g_{\iota, O_{i'_1}})_{T}.
\end{equation}
We continue to define the other two quantities. Define 
\begin{equation}
    C(\mf{M}_1):=\sum_{O_{i'_1}\in \mf{n}_{1, M}(\mf{n}^*_0)} \sum_{\iota} \Norm{T^{\lambda} g_{\iota, O_{i'_1}, \trans} 
    }^p_{\BLka^p(O_{i'_1}\cap B_{\iota})}
\end{equation}
and 
\begin{equation}
    C(\mf{R}_1):=\sum_{O_{i'_1}\in \mf{n}_{1, R}(\mf{n}^*_0)} \sum_{\iota} \Norm{T^{\lambda} g_{\iota, O_{i'_1}, \tang} 
    }^p_{\BLka^p(O_{i'_1}\cap B_{\iota})}.
\end{equation}
In the end, we compare $C(\mf{L}_1), C(\mf{M}_1), C(\mf{R}_1)$ and see which one is the largest. For the one that is the largest, its node $\mf{n}_1$ will be called $\mf{n}^*_1$. This finishes the first step. 

Before we proceed to the next step, we introduce more notations which will be used later and also in the forthcoming broom estimates. If $\mf{n}^*_1\in \mc{L}_1$, then 
\begin{equation}\label{220610e6_43a}
    g^*_{\iota, O_{i_1}}:=g_{\iota, O_{i_1}},
\end{equation}
for which we refer to \eqref{220429e6_24}. If $\mf{n}^*_1\in \mc{M}_1$, then 
\begin{equation}\label{220610e6_44a}
     g^*_{\iota, O_{i_1}}:=g_{\iota, O_{i_1}, \trans}
\end{equation}
for which we refer to \eqref{220610e6_39}. If $\mf{n}^*_1\in \mc{R}_1$, then 
\begin{equation}\label{220610e6_45a}
     g^*_{\iota, O_{i_1}}:=g_{\iota, O_{i_1}, \tang}
\end{equation}
for which we refer to \eqref{220610e6_40}. \\

\noindent \underline{Step $2\le \ell\le \ell_0$.} Here $\ell_0\in \N$ is the last step in the algorithm in Subsection \ref{220607subsection6_2}. Step $\ell$ will be similar to Step 1. Our goal is to define $C(\mf{L}_{\ell}), C(\mf{M}_{\ell}), C(\mf{R}_{\ell})$. We consider the case $\mf{n}_{\ell, L}(\mf{n}^*_{\ell-1})$ and the case $\mf{n}_{\ell, M}(\mf{n}^*_{\ell-1}), \mf{n}_{\ell, R}(\mf{n}^*_{\ell-1})$ separately. 

Take $\mf{n}_{\ell}=\mf{n}_{\ell, L}(\mf{n}^*_{\ell-1})$ and $O_{i_{\ell}}\in \mf{n}_{\ell}$. Suppose that $O_{i_{\ell}}\subset B_{\rho_{\ell}}\cap O_{i_{\ell-1}}$ with $\rho_{\ell}=\rho(\mf{n}_{\ell})$, $O_{i_{\ell-1}}\in \mf{n}^*_{\ell-1}$ and $O_{i_{\ell-1}}\subset B_{\rho_{\ell-1}}$,  $\rho_{\ell-1}=\rho(\mf{n}^*_{\ell-1})$. For a given $\iota$, denote 
\begin{equation}\label{220430e6_36}
    g_{\iota, O_{i_{\ell}}}=\sum_{\substack{T\in \T[B_{\rho_{\ell-1}}]\\ T\cap O_{i_{\ell}}\neq \emptyset}
    } (g^*_{\iota, O_{i_{\ell-1}}})_{T}. 
\end{equation}
Define 
\begin{equation}\label{220430e6_37}
    C(\mathfrak{L}_{\ell}):= \sum_{O_{i_{\ell}}\in \mathfrak{n}_{\ell}} \sum_{\iota} \Norm{T^{\lambda} g_{\iota, O_{i_{\ell}}} }_{\BLka^p(O_{i_{\ell}}\cap B_{\iota})}^p.
\end{equation}
Next, take $\mf{n}_{\ell}=\mf{n}_{\ell, M}(\mf{n}^*_{\ell-1})$ or $\mf{n}_{\ell, R}(\mf{n}^*_{\ell-1})$ and $O_{i'_{\ell}}\in \mf{n}_{\ell}$. Suppose that $O_{i'_{\ell}}\subset B_{\rho_{\ell}}\cap O_{i'_{\ell-1}}$ with $\rho_{\ell}=\rho(\mf{n}_{\ell})$, $O_{i'_{\ell-1}}\in \mf{n}^*_{\ell-1}$ and $O_{i'_{\ell-1}}\subset B_{\rho_{\ell-1}}$,  $\rho_{\ell-1}=\rho(\mf{n}^*_{\ell-1})$. Let $\T_{O_{i'_{\ell}}}$ denote the collection of all $T\in \T[B_{\rho_1}]$  for which 
\begin{equation}
    T\cap B_{\rho_{\ell}}\cap W_{m-1}\neq \emptyset,
\end{equation}
with $m=\dim(\mf{n}^*_{\ell-1})$. Moreover, we will partition $\T_{O_{i'_{\ell}}}$ into two parts
\begin{equation}
    \T_{O_{i'_{\ell}}}=\T_{O_{i'_{\ell}}, \tang}\bigcup \T_{O_{i'_{\ell}}, \trans},
\end{equation}
where 
\begin{equation}\label{220429e6_26z}
    \T_{O_{i'_{\ell}}, \tang}:=\set{T\in \T_{O_{i'_{\ell}}}: T\text{ is } \rho_{\ell}^{-\frac{1}{2}+\delta_{m-1}}\text{-tangent to } W_{m-1} \text{ on } B_{\rho_{\ell}}
    },
\end{equation}
and 
\begin{equation}
    \T_{O_{i'_{\ell}}, \trans}:=\T_{O_{i'_{\ell}}}\setminus \T_{O_{i'_{\ell}}, \tang}.
\end{equation}
We continue to define the other two quantities. Define 
\begin{equation}\label{220430e6_42}
    C(\mf{M}_{\ell}):=\sum_{O_{i'_{\ell}}\in \mf{n}_{{\ell}, M}(\mf{n}^*_{\ell-1})} \sum_{\iota} \Norm{T^{\lambda} g_{\iota, O_{i'_{\ell}}, \trans} 
    }^p_{\BLka^p(O_{i'_{\ell}}\cap B_{\iota})}
\end{equation}
and 
\begin{equation}\label{220430e6_43}
    C(\mf{R}_{\ell}):=\sum_{O_{i'_{\ell}}\in \mf{n}_{\ell, R}(\mf{n}^*_{\ell-1})} \sum_{\iota} \Norm{T^{\lambda} g_{\iota, O_{i'_{\ell}}, \tang} 
    }^p_{\BLka^p(O_{i'_{\ell}}\cap B_{\iota})}.
\end{equation}
In the end, we compare $C(\mf{L}_{\ell}), C(\mf{M}_{\ell}), C(\mf{R}_{\ell})$ and see which one is the largest. For the one that is the largest, its node $\mf{n}_{\ell}$ will be called $\mf{n}^*_{\ell}$. 

In the end, we define $g^*_{\iota, O_{i_{\ell}}}$ in the same way as in  \eqref{220610e6_43a}--\eqref{220610e6_45a}. \\

The above algorithm outputs a sequence of nodes 
\begin{equation}\label{220708e5_78}
    \mf{n}^*_0, \mf{n}^*_1, \dots, \mf{n}^*_{\ell_0}. 
\end{equation}
The parameter $\dim(\mf{n}^*_{\ell})$ is non-increasing in $\ell$. If $\mf{n}^*_{\ell}$ is an $R$-child, then 
\begin{equation}
    \dim(\mf{n}^*_{\ell})=\dim(\mf{n}^*_{\ell-1})-1;
\end{equation}
otherwise the dimension does not decrease. Denote $m:=\dim(\mf{n}^*_{\ell_0})$. We know that $m\ge k$, where $k$ is as in Theorem \ref{201204thm5_1}, as otherwise the desired estimate \eqref{main-esti} there would be trivial. Let 
\begin{equation}\label{220711e5_88}
    \mf{S}_n, \mf{S}_{n-1}, \dots, \mf{S}_m
\end{equation}
denote the nodes from \eqref{220708e5_78} that are $R$-children, where $\mf{S}_n:=\mf{n}^*_0$ is also included. Here $\mf{S}$ is short for ``surface", as elements in $\mf{S}_{n'}$ are neighborhoods of algebraic varieties for each $n'$. We therefore have 
\begin{equation}
    \dim(\mf{S}_{n'})=n', \ \ \forall m\le n'\le n. 
\end{equation}
Moreover, denote 
\begin{equation}\label{220718e5_90}
    r_{n'}:=\rho(\mf{S}_{n'}), \ \ \forall m\le n'\le n, \ \ r_{m-1}:=1. 
\end{equation}
Elements in $\mf{S}_{n'}$ are of the form $B_{r_{n'}}\cap \mc{N}_{r_{n'}^{1/2+\delta_{n'}}}(S_{n'})$ where $S_{n'}$ is some algebraic variety of dimension $n'$. To simplify notation, we will often identify $B_{r_{n'}}\cap \mc{N}_{r_{n'}^{1/2+\delta_{n'}}}(S_{n'})$ with $S_{n'}$ if it is clear from the context that we are talking about the node $\mf{S}_{n'}$. We follow \cite{HR2019} and introduce a few new notions.  

The pair $(S_{n'}, B_{r_{n'}})$ is called a grain, with its dimension given by $n'$ and degree given by the degree of $S$. 

A multigrain $\vec{S}_{n'}$ is a tuple of grains
$$
\vec{S}_{n'}=\left(\mathcal{G}_{n}, \ldots, \mathcal{G}_{n'}\right), \quad \mathcal{G}_{i}=\left(S_{i}, B_{r_{i}}\right) \quad \text { for } n' \leqslant i \leqslant n
$$
satisfying
\begin{enumerate}
    \item $\operatorname{dim}\left(S_{i}\right)=i$ for $n' \leqslant i \leqslant n$;
    \item $S_{n} \supset S_{n-1} \supset \cdots \supset S_{n'}$;
    \item $B_{r_{n}} \supset B_{r_{n-1}} \supset \cdots \supset B_{r_{n'}}$.
\end{enumerate}
Sometimes we also write $\vec{S}_{n'}=(S_n, \dots, S_{n'})$. The parameter $n-n'$ is referred to as the level of the multigrain $\vec{S}_{n'}$. The complexity of the multigrain is defined to be the maximum of the degrees $\operatorname{deg} S_{i}$ over all $n' \leqslant i \leqslant n$.

\begin{definition}[Nested tubes, \cite{hickman2020note}]\label{220726def5_10}
Let $\vec{S}_{n'}=\left(\mathcal{G}_{n}, \ldots, \mathcal{G}_{n'}\right)$ be a multigrain and
$$
\mathcal{G}_{i}=\left(S_{i}, B_{r_i}\right) \quad \text { for } n' \leqslant i \leqslant n .
$$
Define $\mathbb{T}_{r_i}[\vec{S}_{n'}]$ to be the set of length $r_i$ tubes $T\in \mathbb{T}[B_{r_i}]$ that are tangent to $S_i$ and  satisfy  that there exists $T_j \in \mathbb{T}[B_{r_j}]$ for $n'\le j< i$ such that
\begin{equation}
    T_j \subset \mc{N}_{r_{j}^{1 / 2+\delta_{j}}} S_{j}, \ \ \operatorname{dist}\left(\theta(T_i), \theta(T_j)\right) \lesssim r_{j}^{-1 / 2},
\end{equation}   
and 
\begin{equation}
    \operatorname{dist}\left(T_j, T_i \cap B_{r_j}\right) \lesssim r_{i}^{(1+\delta) / 2}
\end{equation}
hold true for all $i, j$ with $n' \leqslant j <i$.
The direction set of $\mathbb{T}_{r_i}[\vec{S}_{n'}]$ is defined to be
\begin{equation}
    \Theta_{r_i}[\vec{S}_{n'}]:=\{
    \theta(T): T\in \T_{r_i}[\vec{S}_{n'}]
    \}. 
\end{equation}
\end{definition}

For each $m\le n'<n$, define 
\begin{equation}
    D_{n'}=d^{\#_c(
    \mf{j}(\mf{n})
    )},
\end{equation}
where $\mf{n}$ is the parent node of $\mf{S}_{n'}$. Moreover, define 
\begin{equation}
    D_{m-1}=d^{
    \#_c
    (
    \mf{j}(n_{\ell_0}^*)
    )
    }, \ \ D_n=1.
\end{equation}
This defines the same quantity as $D_{\ell-1}$ in \cite[page 265]{HR2019}.

\begin{lemma}\label{220708lemma5_9}
For each $n'\ge m-1$, it holds that 
\begin{equation}
    r_{n'} \prod_{i=n'}^n D_{i}\le R.
\end{equation}
\end{lemma}
\begin{proof}[Proof of Lemma \ref{220708lemma5_9}]
It suffices to show that 
\begin{equation}\label{220717e5_94}
    D_i\le r_{i+1}/r_i, \ \forall i<n.
\end{equation}
Recall that $r_{i+1}=\rho(\mf{S}_{i+1})$ and $r_{i}=\rho(\mf{S}_{i})$. As $\mf{S}_{i+1}$ is an $R$-child, by definition (see equation \eqref{220717e5_22} and the line below), we have 
\begin{equation}
    \mf{j}(\mf{S}_{i+1})=\#_a(\mf{j}(\mf{S}_{i+1}))=\#_c(\mf{j}(\mf{S}_{i+1}))=0.
\end{equation}
Let $\mf{n}$ be the parent node of $\mf{S}_i$, and therefore $D_i=d^{\#_c(\mf{j}(\mf{n}))}$. When the algorithm runs from $\mf{S}_{i+1}$ to $\mf{n}$, the radius parameter $\rho$ decreases to $\rho/d$ each time $\#_c$ increases by $1$, and \eqref{220717e5_94} follows immediately. 
\end{proof}

In the end of this subsection, we describe a few output functions of the above algorithm. Take $n'$ with $m\le n'\le n$ and consider the node $\mf{S}_{n'}$. As $\mf{S}_{n'}$ is an $R$-child, it means the tangential case $C(\mf{R}_{\ell})$, which was defined in \eqref{220430e6_43}, dominates. Here $\ell$ is the level that $\mf{S}_{n'}$ belongs to. As elements in $\mf{S}_{n'}$ are neighborhoods of algebraic varieties, from now on we will always use $S_{n'}$ to refer to an element in $\mf{S}_{n'}$. Consequently, $g^*_{\iota, O_{i'_{\ell}}}$ will be called $g^*_{\iota, S_{n'}}$, and its wave packets in the ball $B_{r_{n'}}$ with $S_{n'}\subset B_{r_{n'}}$ are all tangent to $S_{n'}$. 

Regarding these functions, we have the following properties. Let $p_{n'}$ with $n'\ge m\ge k$ be a Lebesgue exponent that is fixed later. These exponents satisfy 
\begin{equation}
    p_m\ge p_{m+1}\ge \dots\ge p_n=p\ge 2,
\end{equation}
where $p$ is the exponent in Theorem \ref{201204thm5_1}. Define $\alpha_{n'}, \beta_{n'}\in [0, 1]$ by 
\begin{equation}
    \frac{1}{p_{n'}}=\frac{1-\alpha_{n'-1}}{2}+\frac{\alpha_{n'-1}}{p_{n'-1}}, \ \ \beta_{n'}=\prod_{i=n'}^{n-1} \alpha_i, 
\end{equation}
for $m+1\le n'\le n-1$, and $\alpha_n=\beta_n=1$. We have \\

\noindent \underline{Property 1.} The inequality
\begin{align}
\|T^{\lambda} g\|_{\mathrm{BL}_{k, A}^{p}\left(B_{R}\right)} & \lessapprox M(\vec{r}_{n'}, \vec{D}_{n'})\|g\|_{L^{2}}^{1-\beta_{n'}}\\
& \left(
\sum_{S_{n'} \in \mf{S}_{n'}}
\sum_{\iota}
\left\|
T^{\lambda} g^*_{\iota, S_{n'}}
\right\|
_{
\mathrm{BL}_
{k, A_{n'}}
^{p_{n'}}(B_{r_{n'}})
}
^
{p_{n'}}
\right)^{\frac{\beta_{n'}}{p_{n'}}},
\end{align}
where $B_{r_{n'}}$ is the ball of radius $r_{n'}$ that contains $S_{n'}$ and 
\begin{equation}
    \vec{r}_{n'}:=(r_n, r_{n-1}, \dots, r_{n'}), \ \ \vec{D}_{n'}:=(D_n, D_{n-1}, \dots, D_{n'}),
\end{equation}
holds for
\begin{align}
M(\vec{r}_{n'}, \vec{D}_{n'}):=\left(\prod_{i=n'}^{n-1} D_{i}\right)^{(n-n') \delta}
\left(\prod_{i=n'}^{n-1} r_{i}^{\left(\beta_{i+1}-\beta_{i}\right) / 2} D_{i}^{\left(\beta_{i+1}-\beta_{n'}\right) / 2}\right)
\end{align}

\noindent \underline{Property 2.} For $n'\le n-1$, we have 
\begin{align}
\sum_{S_{n'} \in \mf{S}_{n'}}
\left\|g^*_{\iota, S_{n'}}\right\|_{2}^{2}
\lessapprox
D_{n'}^{1+\delta} 
\sum_{
S_{n'+1}\in \mf{S}_{n'+1}
}
\norm{
g^*_{
\iota, S_{n'+1}
}
}_2^2,
\end{align}
for every $B_{\iota}\subset B_R$ of radius $R^{1-\delta}$. Here when $n'=n-1$, $g^*_{\iota, S_{n'+1}}$ was not defined before and we simply set $g^*_{\iota, S_{n'+1}}=g$. \\

\noindent \underline{Property 3.} For $n'\le n-1$, we have 
\begin{align}
\max _{S_{n'} \in \mf{S}_{n'}}
\left\|
g^*_{\iota, S_{n'}}
\right\|_{2}^{2}
\lessapprox
\pnorm{
\frac{r_{n'+1}}{r_{n'}}
}^{-\frac{n-n'-1}{2}}
D_{n'}^{-n'+\delta}
\max_{
S_{n'+1}\in \mf{S}_{n'+1}}
\norm{
g^*_{\iota, S_{n'+1}}
}_2^2
\end{align}
and 
\begin{align}
& \max _{S_{n'} \in \mf{S}_{n'}}
\max_{\theta}
\left\|
g^*_{\iota, S_{n'}}
\right\|_{\avetwo(\theta)}^{2}\\& 
\lessapprox
\pnorm{
\frac{r_{n'+1}}{r_{n'}}
}^{-\frac{n-n'-1}{2}}
D_{n'}^{\delta}
\max_{
S_{n'+1}\in \mf{S}_{n'+1}}
\max_{\theta}
\norm{
g^*_{\iota, S_{n'+1}}
}_{\avetwo(\theta)}^2,
\end{align}
where $\theta$ is a frequency cap defined in Subsection \ref{220619subsection4_2} of side length $\rho^{-1/2}$, hold for all $1\le \rho\le r_{n'}$. \\

\noindent \underline{Property 4.} For $n'\le n''\le n$, it holds that 
\begin{equation}
    \norm{
    g^*_{\iota, S_{n'}}
    }_2^2 
    \lesim_{\epsilon} 
    r_{n'}^{\frac{n-n'}{2}} 
    \pnorm{
    \prod_{i=n'}^{n-1} r_i^{-\frac{1}{2}}
    }
    r_{n''}^{-\frac{n-n''}{2}}
    \pnorm{
    \prod_{i=n''}^{n-1}
    r_i^{\frac{1}{2}}
    }
    R^{O(\epsilon_{\circ})}
    \norm{
    g_{\iota, S_{n'}}^{*(n'')}
    }_2^2,
\end{equation}
where 
\begin{equation}\label{220726e4_109}
    g_{\iota, S_{n'}}^{*(n'')}:=
    \sum_{
    T\in 
    \T_{
    r_{n''}
    }[\vec{S}_{n'}]
    }
    (g_{\iota})_T,
\end{equation}
and $\T_{r_{n''}}[\vec{S}_{n'}]$ is from Definition \ref{220726def5_10}. 

If one takes $n''=n$, then $g_{\iota, S_{n'}}^{*(n)}$ becomes $g_{S_{n'}}^{\#}$ in the last equation in \cite[page 9]{hickman2020note}. If $n''=n'$, then $g^*_{\iota, S_{n'}}=g^{*(n'')}_{\iota, S_{n'}}$. \\

These four properties are taken essentially from \cite[page 9]{hickman2020note}. The only main difference is that Hickman and Zahl \cite{hickman2020note} only introduced and used $n''=n$ in \eqref{220726e4_109}. 
The proofs of the first three properties are given in \cite{HR2019}, the proof of the fourth property is the same as that of Property iv) in \cite[page 10]{hickman2020note} and relies on transverse equidistribution properties (for our setting, the needed property is in \cite{MR4047925}), and therefore we will not repeat. 

The only explanation that is needed is as follows. In the last equation in \cite[page 255]{HR2019}, the authors there used the fact that $\rho_{j+1}\simeq \rho_j$, where the implicit constant is universal. In our case, these two radius parameters differ by $d$, which is a large constant. Consequently, Property $(\mathrm{III})_j$ in \cite[page 250]{HR2019} may not hold as is written there. However, we can still obtain some good substitute for it. For a node $\mf{n}$, let $\mf{n}^{\Uparrow}$ denote the closest ancestor (itself included) that is an $R$-child. Let $\mf{n}_{\ell}$ be a node in $\mf{M}'_{\ell}\cup \mf{M}_{\ell}$ with $\mf{n}_{\ell}^{\Uparrow}=\mf{S}_{n'}$. We will prove 
\begin{equation}\label{220710e5_101}
    \norm{
    g_{\iota, O_{\ell}}}_2^2 \le 
    C^{\mathrm{III}}_{\ell, \delta}\cdot 
    \pnorm{
    \frac{r_{n'}}{\rho(\mf{n}_{\ell})}
    }^{-\frac{n-n'}{2}}
    d^{-\#_c(
    \mf{j}(\mf{n}_{\ell})
    )
    (n'-1)
    }
    \norm{g}_2^2,
\end{equation}
for every $O_{\ell}\in \mf{n}_{\ell}$ and ball $B_{\iota}$ of radius $R^{1-\delta}$,  where 
\begin{equation}\label{220710e5_102}
    C^{\mathrm{III}}_{\ell, \delta}:=d^{
    \#_c(
    \mf{j}(\mf{n}_{\ell})
    ) \delta
    +
    \#_a(
    \mf{j}(\mf{n}_{\ell})
    ) \delta
    }
    (r_{n'})^
    {
    \#_a(
    \mf{j}(\mf{n}_{\ell})
    ) 
    O(\delta_{n'-1/2})
    + 
    \#_{a'}(
    \mf{j}(\mf{n}_{\ell})
    ) O(\delta_{n'})
    }
\end{equation}
and 
\begin{equation}\label{220710e5_103}
    \#_{a'}(\mf{j}(
    \mf{n}_{\ell}
    )):=\#\{
    \mf{n}: \mf{n}_{\ell}\preccurlyeq \mf{n}\preccurlyeq \mf{S}_{n'}, \mf{n}\in \mf{M}'_{\ell'} \text{ for some } \ell'
    \}.
\end{equation}
Note that 
\begin{equation}
    \#_a(\mf{j}
    (\mf{n}_{\ell})
    ) \lesim \frac{|\log \delta_{n'-1}|}{\delta_{n'-1}}, \ \ \#_{a'}(\mf{j}
    (\mf{n}_{\ell})
    ) \lesim \frac{|\log \delta_{n'-1/2}|}{\delta_{n'-1/2}},
\end{equation}
and therefore by applying Lemma \ref{220708lemma5_9}, we always have $\eqref{220710e5_102} \lessapprox 1$. \\

Let us prove \eqref{220710e5_101}. Let $\ell_1$ be the largest integer smaller than $\ell$ such that there exists $\mf{n}_{\ell_1}\in \mf{M}'_{\ell_1}\cup \mf{M}_{\ell_1}$ with $\mf{n}_{\ell}\preccurlyeq \mf{n}_{\ell_1}$ and  $\mf{n}_{\ell_1}^{\Uparrow}=\mf{S}_{n'}$; if no such nodes exist, then we simply take $\mf{n}_{\ell_1}=\mf{S}_{n'}$. Assume that \eqref{220710e5_101} has been proved for $\mf{n}_{\ell_1}$, and we will prove it for $\mf{n}_{\ell}$. There are two cases: $\mf{n}_{\ell}\in \mf{M}_{\ell}$ or $\mf{n}_{\ell}\in \mf{M}'_{\ell}$. We only prove the latter case, and the former case can be done in a similar way. List all the nodes between $\mf{n}_{\ell_1}$ and $\mf{n}_{\ell}$ in a descending order:
\begin{equation}
    \mf{n}_{\ell_1}, \mf{n}_{\ell_1+1}, \dots, \mf{n}_{\ell-1}, \mf{n}_{\ell}. 
\end{equation}
Note that $\mf{n}_{\ell'}$ is an $L$-child for every $\ell_1<\ell'< \ell$. By orthogonality between wave packets and the fundamental theorem of algebra, we have 
\begin{equation}
    \norm{g_{\iota, O_{\ell'}}}_2^2 \lesim d^{-(n'-1)} \norm{
    g_{\iota, O_{\ell'-1}}
    }_2^2,
\end{equation}
for every $\ell'<\ell$, $O_{\ell'}\in \mf{n}_{\ell'}, O_{\ell'-1}\in \mf{n}_{\ell'-1}$ and $O_{\ell'}\subset O_{\ell'-1}$. This further implies 
\begin{equation}\label{220711e5_107}
    \norm{
    g_{\iota, O_{\ell-1}}}_2^2 \le 
    C^{\mathrm{III}}_{\ell-1, \delta}\cdot 
    \pnorm{
    \frac{r_{n'}}{\rho(\mf{n}_{\ell_1})}
    }^{-\frac{n-n'}{2}}
    d^{-\#_c(
    \mf{j}(\mf{n}_{\ell-1})
    )
    (n'-1)
    }
    \norm{g}_2^2.
\end{equation}
Here note that the denominator is $\rho(
\mf{n}_{\ell_1}
)$
instead of $\rho(
\mf{n}_{\ell-1}
)$, as remarked at the beginning of Subsection \ref{220607subsection6_2}. When passing from $\mf{n}_{\ell-1}$ to $\mf{n}_{\ell}$, recall that in \eqref{220707e5_43} and \eqref{220707e5_49z}, we cut the neighborhood $\mc{N}_{
\rho(
\mf{n}_{\ell-1}^{\uparrow}
)^{1/2+\delta_m}
}
=
\mc{N}_{
\rho(
\mf{n}_{\ell_1}
)^{1/2+\delta_m}
}$
into thinner layers $\mc{N}_{
\rho(
\mf{n}_{\ell}
)^{1/2+\delta_m}
}$. Therefore by transverse equidistribution properties (for instance Lemma 8.4 in  \cite{MR4047925}), we have 
\begin{equation}
    \norm{g_{\iota, O_{\ell}}}_2^2 \lesim (r_{n'})^{O(\delta_{n'})} d^{-(n'-1)} \pnorm{
    \frac{
    \rho(\mf{n}_{\ell_1})
    }
    {
    \rho(\mf{n}_{\ell})
    }
    }^{-\frac{n-n'}{2}} \norm{
    g_{\iota, O_{\ell-1}}
    }_2^2,
\end{equation}
for $O_{\ell}\in \mf{n}_{\ell}$ with $O_{\ell}\subset O_{\ell-1}$. This, combined with \eqref{220711e5_107} and the choice of the constant in \eqref{220710e5_102}, gives us the desired bound.

\section{Strong polynomial Wolff axioms}\label{220704section4}

\begin{lemma}\label{220223lemma6.4}
Let $\vec{S}_{n'}$ be the multi-grain given in Definition \ref{220726def5_10} with complexity at most $d$. We have 
\begin{equation}\label{SPWAestimateineq}
    \# \Theta_{r_i}[\vec{S}_{n'}] \lesssim_{\epsilon_{\circ}, d}
    \left(\prod_{j=n'}^{i-1} 
    r_{j}^{-1 / 2}\right) 
    r_i^{\frac{i-1}{2}+\epsilon_{\circ}},
\end{equation}
for all $n'\le i\le n$, where $\epsilon_{\circ}$ is given in \eqref{admissable_parameters}. 
\end{lemma}

As with Theorem \ref{main_thm_2}, we will deduce Lemma \ref{220223lemma6.4} from a geometric theorem.

\begin{theorem}\label{nestedgeomlemofPhi}[Strong Polynomial Wolff Axiom for our $\phi$]
Let $n\ge 3$. If Bourgain's condition holds for the phase function $\phi$ at every $(\bfx_0; \xi_0)\in \supp(a)$, then the following strong polynomial Wolff axiom for $\phi$ holds: Let $E\ge 2$ be an integer and fix an integer $k$. For every $\epsilon>0$, there exists $C(n, E, k, \epsilon)>0$ such that for balls
\begin{equation}
    B(\bfx_1, s_1) \subset \frac{1}{2} B(\bfx_2, s_2) \subset \cdots \subset \frac{1}{2^{k-1}}  B(\bfx_k, s_k) \subset \frac{1}{2^k}   B^n
\end{equation}
numbers $\kappa^{C_0} \leq \kappa_1 \leq \ldots \leq \kappa_k \leq \kappa$ and for $S_j \subset B(\bfx_j, s_j)$, $j = 1, 2, \ldots, k$ satisfying:
\begin{itemize}
    \item $\kappa_j \leq s_j, \forall 1 \leq j \leq k$ and $\varphi_j := \frac{s_j}{\kappa_j}$ satisfy $\varphi_1 \geq \varphi_2 \geq \cdots \geq \varphi_k$,
    \item $S_j$ is a semialgebraic set of complexity $\le E$ whose $\kappa_j$-neighborhood has volume $\simeq |S_j|$,
    \item The intersection between $S_j$ and any ball of radius $r\in[\kappa_j, s_j]$ has volume $\leq C_j r^{d_j} \kappa_j^{n-d_j}$,
\end{itemize}
the following holds uniformly:

For every collection $\T$ of $\kappa$-tubes pointing in different directions (defined before Theorem \ref{main_thm_2}), if we use $c(T)$ to denote the core curve of $T$, then
\begin{align}\label{concluofSPWA}
& \#\{T\in \T: c(T) \bigcap \mc{B}(\bfx_j, \frac{1}{2} s_j) \subset S_j, \forall 1 \leq j \leq k\}\nonumber\nonumber\\
\le & C(C_0, n, E, k, \epsilon)\prod_{j=1}^k (\frac{\varphi_{j-1}}{\varphi_{j}})^{d_j-1}(\frac{\varphi_k}{\kappa})^{n-1}   \kappa^{-\epsilon}
\end{align}
where $\varphi_0 = 1$ and the \emph{horizontal slab}\footnote{In general, we call a set to be a \emph{horizontal slab} if it is $\pi_t^{-1} (I)$ for some interval $I$.} $\mc{B}(\bfx_j, \frac{1}{2} s_j)$ is defined to be $\pi_t^{-1} (\pi_t (B(\bfx_j, \frac{1}{2} s_j)))$. Here $\pi_t: \R^n \to \R$ is the orthogonal projection to the $t$-variable.

Moreover the implied constant only depends on bounds of finitely many (depending on $C_0, n, E, k, \epsilon$) derivatives of $\phi$.
\end{theorem}

Like in Section \ref{220728section3}, we are going to deduce Lemma \ref{nestedgeomlemofPhi} when the phase function satisfies a concrete derivative condition. Then we simply check that the condition is satisfied by our phase function.

\begin{theorem}
\label{nestedgeomlemwithderivativecond}[Generalized Strong Polynomial Wolff Axiom]
Let $n\ge 3$. Suppose for a $\Phi$ as in the beginning of Section \ref{220728section3}:

\begin{enumerate}[label=({\alph*})]
    \item For every choice of $v \in [-1, 1]^{n-1}$, $\xi \in [-1, 1]^{n-1}$, subinterval $I \subset [-1, 1]$ and $t \in [-1, 1]$,  we have both \begin{align}\label{averagecompareofJacobianeq}
    & |\det (\nabla_{v} \Phi(v, t, \xi)\cdot M+\nabla_{\xi} \Phi(v, t, \xi))|\nonumber\\
    \lesssim & \left(1+ \frac{\mathrm{dist}(t, I)}{|I|}\right)^{n-1}\frac{1}{|I|}\int_{I} |\det (\nabla_{v} \Phi(v, s, \xi)\cdot M+\nabla_{\xi} \Phi(v, s, \xi))|\mathrm{d} s%\nonumber\\
    %\gtrsim_{\epsilon} & \kappa^{C\epsilon}, \forall M \in \mathrm{Mat}_{(n-1)\times (n-1)} (\R)
    \end{align}
    and \eqref{averagelargenessofJacobianeq} for some implied constants independent of the choices of $v, \xi, I$ and $M$.
    \item If $t_1 \neq t_2$, $\Phi(v, t_1, \xi) = x_{1}$, and  $\Phi(v', t_1, \xi')) = x_{1}$, then
    \begin{equation}
        |\Phi(v', t_2, \xi')) - \Phi(v, t_2, \xi))| \lesssim |t_1 - t_2|\cdot |\xi - \xi'|.
    \end{equation}
    \item If $t_1 \neq t_2$ and $\Phi(v, t_1, \xi)) = x_{1}$, $\Phi(v, t_2, \xi)) = x_{2}$, then for $x_{2}'$ with distance $\mu |t_1-t_2|$ from $x_{2}$ $(\mu \leq 10)$, there are $v'$ and $\xi'$ with $\Phi(v', t_1, \xi')) = x_{1}$, $\Phi(v', t_2, \xi')) = x_2 '$ and $|\xi' - \xi| \lesssim \mu$.
\end{enumerate}

Then the conclusion of Lemma \ref{nestedgeomlemofPhi} (with the notion ``pointing in different directions'' now defined as in the beginning of Section \ref{220728section3}) holds with the implied constant only depends on bounds of finitely many (depending on $C_0, n, E, k, \epsilon$) derivatives of $\Phi$ and the implied constants in (a) - (c).
\end{theorem}

\begin{remark}\label{explainingb}
We explain the intuition behind (b) and (c) a bit. For convenience we introduce the following notation. For fixed $v \in \R^{n-1}$ and $\xi \in \R^{n-1}$, we call the curve
\begin{equation}
    c_{v, \xi} = \{(x, t) \in \R^{n-1} \times [-1, 1]: x = \Phi(v, t, \xi)\}
\end{equation}
to be a \emph{$\Phi$-curve}. Intuitively, if we know a $\Phi$-curve passes through $(x_1, t_1)$ and want to perturb the ``direction variable'' $\xi$ and the ``initial position variable'' $v$ so that $(x_1, t_1)$ is still on the curve, then (b) says whenever the perturbation on the direction $\xi$ is $O(\mu)$ we always have the perturbation of the curve at time $t= t_2$ is $O(|t_1 - t_2| \mu)$, and (c) says if we want the $x$ coordinate at time $t=t_2$ to be shifted by a distance $\simeq \mu |t_1 - t_2|$, we can always succeed with the amount of the perturbation needed on $\xi$ being $O(\mu)$.
\end{remark}

\begin{proof}[Proof of Theorem \ref{nestedgeomlemwithderivativecond}]

Similarly to the proof of Theorem \ref{GPWAthm}, we only do the proof when $\Phi$ is fixed and after seeing the proof it will be clear that the estimate only depends on finitely many derivatives of $\Phi$ (in particular since the smallest scale ($\kappa^{C_0}$) we consider is polynomial in $\kappa$, in the approximation argument described below one only needs a Taylor approximation of order $O_{C_0, n, E, k, \epsilon} (1)$).

Like the proof of Theorem \ref{GPWAthm}, we can reduce the situation to the case where $\Phi$ is a polynomial of degree $O_{C_0, n, E, k, \epsilon} (1)$ by a Taylor approximation argument. For general $\Phi$ by this argument one reduces to a problem with two modified conditions: (a') a slightly weaker condition than \eqref{averagecompareofJacobianeq} in (a) (similar to \eqref{avgbigJacrestreq} versus \eqref{averagelargenessofJacobianeq}) that has very small error term (of the form $\kappa^{-1000n(1+C_0)}$) which makes no difference (see the proof below and how to deal with this issue in the similar situation in the proof of Theorem \ref{GPWAthm}), and (b') (c') two slightly weaker conditions than (b) and (c) with an error term $\kappa^{1+C_0}$ for $|t| \leq \kappa^{\epsilon}$, both not affecting our framework (for (b) and (c), note that the only place they are needed is the verification of a claim in the beginning of the induction step). From now on we always assume $\Phi$ is a polynomial of degree $O_{C_0, n, E, k, \epsilon} (1)$.

Let us assume $\kappa$ and $\kappa_j$ all all sufficiently small (allowed to depend on derivatives of $\Phi$). Otherwise we simply ignore some constraints. This assumption will enable us to use the implicit function theorem at scales $\kappa$ or $\kappa_j$ freely.

We make one more comment before starting: the present theorem is a generalization of Lemma 3.7 in \cite{hickman2020note}, which was in turn developed based on Theorem 1.4 in \cite{HRZ} or Theorem 1.9 in \cite{MR4205111}). Our proof will have a lot in common with these, and will be a natural generalization of Theorem \ref{GPWAthm}.

We also refine the slabs $\mc{B}(\bfx_j, \frac{1}{2}s_j)$ a bit before starting. Since each $B (\bfx_j, s_j)$ is contained in $\frac{1}{2}B (\bfx_{j+1}, s_{j+1})$ ($\forall j < k$), we can take horizontal slabs $\mc{B}_1, \ldots, \mc{B}_k$ such that:
\begin{enumerate}[label=(\roman*)]
    \item $\mc{B}_j \subset \mc{B}(\bfx_j, \frac{1}{2}s_j)$.
    \item The thickness of $\mc{B}_j$ is $\simeq s_j$.
    \item The distance between $\mc{B}_{j_1}$ and $\mc{B}_{j_2}$ is $\gtrsim s_{j_2}$ for all pairs $j_1 < j_2$.
\end{enumerate}

For $1 \leq l \leq k$ and $t \in [-1, 1]$, we define
\begin{equation}\label{defnofSlt}
    S_{l, t} = \{y \in \R^{n-1}: \exists \text{ a } \Phi-\text{curve } c_{v, \xi} \text{ s.t. } c_{v, \xi} \bigcap \mc{B}_j \subset S_j, \forall 1 \leq j \leq l \text{ and } (y, t) \in c_{v, \xi}\}
\end{equation}
and we are going to prove inductively that
\begin{equation}\label{inductivenbhdvolineq}
m_{n-1}^*(S_{l, t}) \le C(C_0, n, E, k, \epsilon) (d_l (t))^{n-1} \prod_{j=1}^l (\frac{\varphi_{j-1}}{\varphi_{j}})^{d_j-1}\varphi_l^{n-1}   \kappa^{-2^{l-k-1}\epsilon}, \forall t \in [-1, 1]
\end{equation}
where $d_l (t)$ is defined to be $s_l$ plus the distance between $B(\bfx_l, s_l)$ and the hyperplane $\{x_n = t\}$, and $m_{n-1}^*(\cdot)$ is the $(n-1)$-dimensional Lebesgue outer measure.

For convenience, we choose a large constant $K$ depending on the implied constant in (b) and (c) and also consider a companion set
\begin{equation}\label{defneqoftildeSl}
    \tilde{S}_{l, t} = \{y \in \R^{n-1}: \exists \text{ a } \Phi-\text{curve } c_{v, \xi} \text{ s.t. } c_{v, \xi} \bigcap \mc{B}_j \subset \mc{N}_{K\kappa_j} (S_j), \forall 1 \leq j \leq l \text{ and } (y, t) \in c_{v, \xi}\}
\end{equation}

Note that once we prove \eqref{inductivenbhdvolineq}, by the same reasoning and the assumption about $\mc{N}_{\kappa_j} (S_j)$, we also prove the same upper bound for $m_{n-1}^* (\tilde{S}_{l, t})$.

\noindent \underline{Base case.} Our base case is when $l=0$. Since $\Phi$ has a $C^1$-derivative bound, we observe that all $S_{l, t}$ lie in a uniformly bounded set $\Omega_0$. We make the convention that $S_{0, t}$ is the part of the set defined in \eqref{defnofSlt} with $l=0$ (hence with a vacuous condition) lying in $\Omega_0$. \eqref{inductivenbhdvolineq} then trivially holds for $l=0$ with the convention $d_0 (t) = 1$. We will see the  first induction step is similar to the steps afterward with this setup.

\noindent \underline{The inductive step.}  Suppose we have  \eqref{inductivenbhdvolineq} for some $l$ in $[1, k)$. Next we prove it for $l+1$.

Integrate the induction hypothesis over $t$ and temporally ignore measurability issues, by Fubini we formally deduce
\begin{eqnarray}
& m_{n}^*((\bigcup_{t} S_{l, t})\bigcap \mc{B}_{l+1}) \le  C(C_0, n, E, k, \epsilon) \prod_{j=1}^l (\frac{\varphi_{j-1}}{\varphi_{j}})^{d_j-1}\varphi_l^{n-1} s_{l+1}^{n}  \kappa^{-2^{l-k-1}\epsilon}.
\end{eqnarray}

We assert a stronger conclusion: In fact, $(\bigcup_{t} S_{l, t})\bigcap \mc{B}_{l+1}$ is contained in a set $U_l$ such that $U_l$ is a union of ($n$-dim) balls of radii $\varphi_l s_{l+1}$ and that 
\begin{eqnarray}\label{Ulvolbound}
m_{n}^*(U_l) \le C(C_0, n, E, k, \epsilon) \prod_{j=1}^l (\frac{\varphi_{j-1}}{\varphi_{j}})^{d_j-1}\varphi_l^{n-1} s_{l+1}^{n}  \kappa^{-2^{l-k-1}\epsilon}.
\end{eqnarray}

To construct such a $U_l$, we take the $\varphi_l s_{l+1}$-neighborhood of $(\bigcup_{t} S_{l, t})\bigcap \mc{B}_{l+1}$ and cover it by a finitely overlapping collection of balls of radii $\varphi_l s_{l+1}$. Define the union of these balls to be $U_l$.

It remains to derive the volume bound \eqref{Ulvolbound}.  We claim that we can take $K$ in \eqref{defneqoftildeSl} large (and the constraint here will be the only one affecting the choice of $K$) such that this $U_l$ is contained in $\bigcup_{t} \tilde{S}_{l, t}$. 

To verify  this claim, by definition we see $U_l$ is contained in the $3\varphi_l s_{l+1}$-neighborhood of $(\bigcup_{t} S_{l, t})\bigcap \mc{B}_{l+1}$. This means for every point $(\tilde{y}, t)$ in $U_l$, we can find a $\Phi$-curve that is $3\varphi_l s_{l+1}$-close to this point and the part of that curve in $\mc{B}_j$ completely lies in $S_j$, $\forall 1 \leq j \leq l$. Now keep a point of that $\Phi$-curve in $\mc{B}_1$ fixed and by assumption (c) (see Remark \ref{explainingb} for more intuition), one can change the ``direction'' $\xi$ by up to $O(\varphi_l)$ so that the new $\Phi$-curve now passes through $(\tilde{y}, t)$. By assumption (b) for each $t \in \pi_t \mc{B}_j (1 \leq j \leq l)$, the perturbation amount of the $x$ variable is $\lesssim \varphi_l s_j \lesssim \varphi_j s_j = \kappa_j$. Hence the intersection between the new $\Phi$-curve and $\mc{B}_j$ lies in $\mc{N}_{K\kappa_j} (S_j)$ if $K$ is sufficiently large depending on the implied constants in (b) and (c). For this choice of $K$ we just proved that the claim holds. Applying the induction hypothesis to each $t$-slice of $\tilde{S}_{l, t}$ and integrate, we deduce \eqref{Ulvolbound}. %Now by the non-degeneracy condition \eqref{Jacobianvlowerbdeqn} we can find a $\Phi$-curve through $(\tilde{\bfy}, t)$ by changing the parameter $v$ by $O(\phi_l s_{l+1})$. The new $\Phi$-curve and the old $\Phi$-curve will have distance $O(\|\Phi\|_{C^1} \phi_l s_{l+1})$ on all $t$-hyperplane ($\forall t \in [-1, 1]$). Hence if we choose $K$ large, proportionally to $\|\Phi\|_{C^1}$, the part of the new $\Phi$-curve in $\mc{B}_j$ completely lies in $\mc{N}_{K\kappa_j} (S_j)$, $\forall 1 \leq j \leq l$

Our $U_l$ is a union of $\varphi_l s_{l+1}$-balls that contains $(\bigcup_{t} S_{l, t})\bigcap \mc{B}_{l+1}$ and obeys the volume bound \eqref{Ulvolbound}. By a covering lemma we may assume without loss of generality that the $\varphi_l s_{l+1}$-balls are finite-overlapping. Now we use the volume upper bound of the intersection between $S_{l+1}$ and $r$-balls in the assumption of Theorem \ref{nestedgeomlemofPhi}. We deduce
\begin{align}\label{Ulplus1volbound}
& m_{n}^*((\bigcup_{t} S_{l+1, t})\bigcap \mc{B}_{l+1})\nonumber\\ \le & C(C_0, n, E, k, \epsilon) \prod_{j=1}^l (\frac{\varphi_{j-1}}{\varphi_{j}})^{d_j-1}\varphi_l^{n-1} s_{l+1}^{n} (\frac{\kappa_{l+1}}{\varphi_l s_{l+1}})^{n-d_{l+1}}  \kappa^{-2^{l-k-1}\epsilon}\nonumber\\
= & C(C_0, n, E, k, \epsilon) \prod_{j=1}^l (\frac{\varphi_{j-1}}{\varphi_{j}})^{d_j-1}\varphi_l^{n-1} s_{l+1}^{n} (\frac{\varphi_{l+1}}{\varphi_l})^{n-d_{l+1}}  \kappa^{-2^{l-k-1}\epsilon}\nonumber\\
= & C(C_0, n, E, k, \epsilon) \prod_{j=1}^{l+1} (\frac{\varphi_{j-1}}{\varphi_{j}})^{d_j-1}\varphi_{l+1}^{n-1} s_{l+1}^{n} \kappa^{-2^{l-k-1}\epsilon}.
\end{align}

We will use \eqref{Ulplus1volbound} to close the induction step by an argument developed in \cite{HRZ} and \cite{MR4205111}. Below we fix an arbitrary $t=t_0$ to do the proof. 
This step is a lot similar to the proof of Theorem \ref{GPWAthm}. So we will present some steps in sketch only. 

Define the set
\begin{equation}
    L_{l+1, t} = \{(v, \xi):  c_{v, \xi} \bigcap \mc{B}_j \subset S_j, \forall 1 \leq j \leq l+1\}.
\end{equation}

We already assumed $\Phi$ is a polynomial of degree $O_{C_0, n, E, k, \epsilon} (1)$. Thus in the expression of a $\Phi$-curve, $x$ is polynomial in $v, t, \xi$ with the same degree bound. For simplicity we use $b_{t_0}$ to denote the hyperplane $\{x_n = t_0\}$.

By effective quantifier elimination (i.e.  the Tarski-Seidenberg theorem) we find a  semialgebraic subset $L_{l+1, t} ' \subset L_{l+1, t}$ of complexity $O_{n, E, \epsilon_1, K} (1)$ such that all $c_{v, \xi} \bigcap b_{t_0}$ are distinct for $(v, \xi) \in L_{l+1, t} '$, and that $\{c_{v, \xi} \bigcap b_{t_0}: (v, \xi) \in L_{l+1, t} '\} = \{c_{v, \xi} \bigcap b_{t_0}: (v, \xi) \in L_{l+1, t}\}$. (To see this one can first add $(n-1)$ more coordinates to each $(\xi, v) \in L_{l+1, t}$ denoting the ``position'', i.e. the first $(n-1)$ coordinates, of the intersection $c_{v, \xi} \bigcap b_{t_0}$. This is still a semialgebraic set of bounded complexity. Then one applies the quantifier elimination to find a subset such that the last $(n-1)$ coordinates are distinct among different points in the subset and the set of the last$(n-1)$ coordinates does not change. From the construction we see easily that $L_{l+1, t} '$ has dimension $\leq n-1$)

Using Gromov's lemma to approximate $L_{l+1, t} '$ by images of smooth maps in the same way as we did to prove Theorem \ref{GPWAthm}, %(and notice that $\mc{N}_{O(\kappa)} (S_j)$ and $S_j$ satisfy the same volume upper bound), 
we see (for arbitrary $\epsilon_1>0$) there exist two polynomial maps $F$ and $G: [0, \kappa^{\epsilon_1}]^{n-1} \to \R^{n-1}$ (whose images are the $v$ and the $\xi$ variables, respectively) with $\deg F, \deg G = O_{n, E, \epsilon_1} (1)$ and $\|F\|_{C^1}, \|G\|_{C^1} \leq 1$ such that
\begin{equation}\label{curveinvarietySPWA}
    c_{{F(x)}, G(x)} \bigcap \mc{B}_j \subset \mc{N}_{\kappa_j} (S_j), \forall x, \forall 1 \leq j \leq l+1
\end{equation}
and that
\begin{equation}\label{imofcubelargeeqinSPWA}
    \mc{H}^{n-1} (\{c_{{F(x)}, G(x)} \bigcap b_{t_0}: x \in  [0, \kappa^{\epsilon_1}]^{n-1}\}) \gtrsim_{n, E, \epsilon_1, K} \kappa^{C\epsilon_1} m_{n-1}^*(S_{l+1, t_0})
\end{equation}
where $\mc{H}^{n-1} (\cdot)$ stands for the $(n-1)$-dimensional Hausdorff measure on the hyperplane $b_{t_0}$.

Now look at the $n$-dimensional volume of
\begin{equation}
    M = \{(\Phi (F(x), t, G(x)), t): x \in [0, \kappa^{\epsilon_1}]^{n-1}\}\bigcap \mc{B}_{l+1} = (\bigcup c_{{F(x)}, G(x)}) \bigcap \mc{B}_{l+1}
\end{equation}
and the $(n-1)$-dimensional volume of
\begin{equation}
    H = \{(\Phi (F(x), t_0, G(x)), t_0): x \in [0, \kappa^{\epsilon_1}]^{n-1}\} = (\bigcup c_{{F(x)}, G(x)}) \bigcap b_{t_0}
\end{equation}
and compare them.

Suppose the time interval (i.e. the range of the last coordinate) of $\mc{B}_{l+1}$ is $I_{l+1}$. Then use B\'{e}zout like in the proof of Theorem \ref{GPWAthm}, we have
\begin{align}\label{integralofDelta}
    & |M| \sim_{n, E, \epsilon_1} \int_{I_{l+1}} \int_{[0, \kappa^{\epsilon_1}]^{n-1}} |\nabla_x (\Phi (F(x), t, G(x)))|  \mathrm{d}x\mathrm{d}t\nonumber\\
    = & \int_{I_{l+1}} \int_{[0, \kappa^{\epsilon_1}]^{n-1}} |\det(\nabla_v \Phi (F(x), t, G(x))\cdot \nabla_x F + \nabla_{\xi} \Phi (F(x), t, G(x))\cdot \nabla_x G)|  \mathrm{d}x\mathrm{d}t\nonumber\\
    = & \int_{[0, \kappa^{\epsilon_1}]^{n-1}} |\det(\nabla_x G)|\nonumber\\
    \cdot & \int_{I_{l+1}}  |\det(\nabla_v \Phi (F(x), t, G(x))\cdot (\nabla_x F\cdot (\nabla_x G)^{-1}) + \nabla_{\xi} \Phi (F(x), t, G(x)))|  \mathrm{d}t\mathrm{d}x.
\end{align}

On the other hand,
\begin{align}\label{integralofH}
    & \mc{H}^{n-1} (H) \sim_{n, E, \epsilon_1} \int_{[0, \kappa^{\epsilon_1}]^{n-1}} |\nabla_x (\Phi (F(x), t_0, G(x)))|  \mathrm{d}x\nonumber\\
    = & \int_{[0, \kappa^{\epsilon_1}]^{n-1}} |\det(\nabla_v \Phi (F(x),  t_0, G(x))\cdot \nabla_x F + \nabla_{\xi} \Phi (F(x),  t_0, G(x))\cdot \nabla_x G)|  \mathrm{d}x\nonumber\\
    = & \int_{[0, \kappa^{\epsilon_1}]^{n-1}} |\det(\nabla_x G)|\cdot  |\det(\nabla_v \Phi (F(x), t_0, G(x))\cdot (\nabla_x F\cdot (\nabla_x G)^{-1}) + \nabla_{\xi} \Phi (F(x), t_0, G(x)))|\mathrm{d}x.
\end{align}

Now by assumption (a), the left hand side of \eqref{integralofDelta} is $\gtrsim |I_{l+1}|^n \cdot d_{l+1} (t_0)^{-(n-1)} \simeq s_{l+1}^n \cdot d_{l+1} (t_0)^{-(n-1)}$ times the left hand side of \eqref{integralofH}. Note that (the $\tilde{S}_{l+1, t}$ version of) \eqref{Ulplus1volbound} gives an upper bound of the left hand side of \eqref{integralofDelta}.  Moreover \eqref{imofcubelargeeqinSPWA} gives a lower bound of the left hand side of \eqref{integralofH} in terms of $m_{n-1}^*(S_{l+1, t_0})$. Combining everything, we can take $\epsilon_1$ to be a sufficiently small multiple of $\epsilon$ to finish the induction step and \eqref{inductivenbhdvolineq} is proved.

From \eqref{inductivenbhdvolineq} the conclusion will follow easily. Take the union of all $S_{k, t}$ for $t \in [-1, 1]$. We notice by the definition and effective quantifier elimination that this is a semialgebraic set of complexity $O_{n, E, \epsilon} (1)$. Use  \eqref{inductivenbhdvolineq}, we see its measure is
\begin{equation}
    \leq C(C_0, n, E, k, \epsilon)  \prod_{j=1}^k (\frac{\varphi_{j-1}}{\varphi_{j}})^{d_j-1}\varphi_k^{n-1}   \kappa^{-2^{-1}\epsilon}.
\end{equation}
 Note that we already have \eqref{averagelargenessofJacobianeq} holds. In exactly the same way as we proved Theorem \ref{GPWAthm} (the only difference is that Theorem \ref{GPWAthm} was stated for tubes and we need a version for $\Phi$-curves but notice that in Theorem \ref{GPWAthm} in fact a $\Phi$-curves version was proven), we can bound the left hand side of \eqref{concluofSPWA} by
 \begin{equation}
     C(C_0, n, E, k, \epsilon)  \prod_{j=1}^k (\frac{\varphi_{j-1}}{\varphi_{j}})^{d_j-1}\varphi_k^{n-1}   \kappa^{(1-n)-\epsilon}.
 \end{equation}
 This concludes the proof.
\end{proof}

\begin{proof}[Proof of Theorem \ref{nestedgeomlemofPhi}]
The proof will be similar to the proof of Theorem \ref{main_thm_2} in \S\ref{220728section3}. As with that Theorem, take unique smooth $\Phi = \Phi(v, t, x)$ near $0$ such that \eqref{defnofPhi} holds.
We can assume the above can be done for all $(v, t, \xi)\in [-1.5, 1.5]^{2n-1}$ without loss of generality like in the other proof. %The non-degenerate assumption of $\phi$  also  implies \eqref{Jacobianvlowerbdeqn}.
It suffices to show that our $\Phi$ satisfies conditions (a)-(c) in Theorem \ref{nestedgeomlemwithderivativecond}. That Theorem then immediately leads to the desired conclusion. When reading the proof one naturally sees the implicit constants in (b) and (c) only depend on finitely many derivatives of $\phi$. We also note that (b) and (c) comes from the non-degeneracy property of $\phi$, and (a) comes from Bourgain's condition.

For (a), \eqref{averagelargenessofJacobianeq}   is already verified in the proof of Theorem \ref{main_thm_2}. Recall we defined $A(t; \xi)= \nabla^2_{\xi} \phi(X_t(\xi), t; \xi)$ in \eqref{defnofA} and deduced $A(t; \xi) = f(t; \xi)B(\xi) + A(0; \xi)$ and the time derivative of $f$ is $1$ at $t=0$ around \eqref{scalarcondofA}. We make one more harmless assumption that time derivative of $f$ is always in $(\frac{1}{2}, 2)$ since otherwise we can do a constant rescaling that only causes loss of a constant. Now we can do the reduction to both sides of \eqref{averagecompareofJacobianeq} like in the proof of Theorem  \ref{main_thm_2} and reduces \eqref{averagecompareofJacobianeq} to proving that for every polynomial $P(t)$ of degree $n-1$,
\begin{equation}
    |P(t)| \lesssim \left(1+ \frac{\mathrm{dist}(t, I)}{|I|}\right)^{n-1}\frac{1}{|I|}\int_{I} |P(s)|\mathrm{d}s,
\end{equation}
which is an elementary Theorem proved in e.g. Lemma 3.8 in \cite{HRZ}. We have completed the verification of (a). 

As before, (b) and (c) are more general properties that do not depend on Bourgain's condition. Next we verify them. %To verify them, recall that $\Phi$ is defined in \eqref{defnofPhi}. Without loss of generality we only need to check (b) and (c) for $t_1 = 0$ (otherwise we do a translation on $t$ and the proof will be identical). Assume $t_1 = 0$ henceforth.

%Since \eqref{defnofPhi} asserts $(\nabla_{\xi} \phi)(\Phi, t; \xi)=v$. We see around the origin the taylor expansion of $v$ in terms of 
Differentiate \eqref{defnofPhi}  with respect to $t$, we see
\begin{equation}
    \nabla_x \nabla_{\xi} \phi \cdot \partial_t \Phi + \partial_t\nabla_{\xi} \phi = 0.
\end{equation}

Hence $\Phi$-curves can be viewed as integral curves of the vector field $V_{\xi} (\bfx, t) = (\nabla_x \nabla_{\xi} \phi (\bfx, t, \xi)^{-1}\cdot \partial_t\nabla_{\xi} \phi (\bfx, t, \xi), 1)$ parameterized by $\xi$. Note that the non-degeneracy of $\phi$ implies $|\nabla_x \nabla_{\xi} \phi| \simeq 1$ and a $\mu$-perturbation on $\xi$ only causes a $O(\mu)$-perturbation of the above vector field. We see (b) follows from stability of ODE solutions.

Next we check (c), which is the ``opposite direction'' to (b). When checking it we can assume $t_1$ and $t_2$ are sufficiently close and that $x_1$ is sufficiently close to $0$ (and in application the honest (c) will always be satisfied after a harmless constant-rescaling of $(x, t)$). (c) basically asks: if we start from 
\begin{equation}
    x_0 = \Phi(v_0, t_1, \xi_0)
\end{equation}
and start to change $\xi$ and solve $v$ from the equation
\begin{equation}\label{constraintofvandxi}
x_0 = \Phi(v, t_1, \xi),
\end{equation}
how would $y = \Phi(v, t_2, \xi)$ change? For convenience denote $y_0 = \Phi(v_0, t_2, \xi_0)$. We use differentiation to compute this change. Differentiate \eqref{constraintofvandxi}, we see
\begin{equation}
    \nabla_v \Phi|_{(v_0, t_1, \xi_0)} \cdot \nabla_{\xi} v|_{v_0} + \nabla_{\xi} \Phi|_{(v_0, t_1, \xi_0)} = 0.
\end{equation}

Hence by the chain rule,
\begin{equation}
    \nabla_{\xi} y = \nabla_v \Phi|_{(v_0, t_2, \xi_0)} \cdot (-\nabla_v \Phi|_{(v_0, t_1, \xi_0)})^{-1} \cdot \nabla_{\xi} \Phi|_{(v_0, t_1, \xi_0)} + \nabla_{\xi} \Phi|_{(v_0, t_2, \xi_0)}.
\end{equation}

By \eqref{nablaxiPhieq} and \eqref{nablavPhieq}, this simplifies to
\begin{equation}
    \nabla_{\xi} y = (\nabla x \nabla_{\xi} \phi|_{y_0, t_2, \xi_0})^{-1}\cdot (\nabla_{\xi}^2 \phi|_{x_0, t_1, \xi_0} - \nabla_{\xi}^2 \phi|_{y_0, t_1, \xi_0}).
\end{equation}

The first factor $(\nabla x \nabla_{\xi} \phi|_{y_0, t_2, \xi_0})^{-1}$ has entries $\lesssim 1$ and determinant $\simeq 1$ and is harmless. We focus on the second factor. It is equal to $(t_2-t_1)$ times some  $(\sum_{j=1}^{n-1} c_j\partial_{x_j}\nabla_{\xi}^2 \phi + \partial_t\nabla_{\xi}^2 \phi)|_{x_0, t_1, \xi_0}$ plus a higher order term in $(t_2 - t_1)$, where each $|c_j|\lesssim 1$. By a familiar technique of parabolic rescaling (see for example the reduction Lemmas 4.1-4.3 in \cite{MR4047925}), one can assume all $\|\partial_{x_j}\nabla_{\xi}^2 \phi\|$ are uniformly very small and since we have the nondegeneracy condition on $\partial_t\nabla_{\xi}^2 \phi$ from \eqref{211003e1.7}, we see  $\nabla_{\xi} y$ is equal to $(t_2 - t_1)$ times a nondegenerate matrix of bounded entries. From here we see (c) holds by an application of the implicit function theorem.

Now that (a)-(c) are all verified, we apply Theorem \ref{nestedgeomlemwithderivativecond} and conclude the proof.
\end{proof}

\begin{proof}[Proof of Lemma \ref{220223lemma6.4}]
At this point, the Lemma is just a straightforward consequence of Theorem \ref{nestedgeomlemofPhi}. We rescale the whole $B_{r_{n'}}$ to the unit ball and rescale all $\mc{N}_{r_{j}^{1 / 2+\delta_{j}}} S_{j}$ in Definition \ref{220726def5_10} accordingly (to be our $S_j$ in Theorem \ref{nestedgeomlemofPhi}). For each possible $\theta(T)$ in $\Theta_{r_i}[\vec{S}_{n'}]$, we pick the core curve of the corresponding $T_{n'}$, rescale it into the unit ball and extend it into a $\Phi$-curve ($\Phi$ defined from $\phi$ as in the beginning of \S\ref{220728section3}). Then we choose $\kappa = R^{-\frac{1}{2}}$ and see by Definition \ref{220726def5_10} that the set of $\kappa$-tubes around all above $\Phi$-curves satisfy the assumption of Theorem \ref{nestedgeomlemofPhi} with the balls having radii $s_j = \frac{r_{i-1+j}}{r_{n'}}$ and corresponding $\kappa_j = \frac{r_{i-1+j}^{\frac{1}{2}+\delta_{i-1+j}}}{r_{n'}}$. Hence $\varphi_j = r_{n'+1-j}^{-\frac{1}{2}+\delta_{i-1+j}}$ By Wongkew's theorem \cite{Wongkew} of intersections between neighborhood of algebraic varieties we can take $d_j = i-1+j$. Lastly we can surely take $E = O_{d, n}(1)$ to be a constant by the definition of $S_j$ in Definition \ref{220726def5_10}.

By Theorem \ref{nestedgeomlemofPhi}, we see the left hand side of \eqref{SPWAestimateineq} is
\begin{equation}
    \lesssim r_i^{\delta_i} r_{n'}^{\frac{n'-1}{2}}\prod_{j=n'+1}^{i} \left(\frac{r_j}{r_{j-1}}\right)^{\frac{j-1}{2}}
\end{equation}
and is thus bounded by the right hand side. 
\end{proof}

As a corollary of Lemma \ref{220223lemma6.4}, we obtain 
\begin{corollary}\label{220711lemma5_5}
For $m\le n'\le n''$, we have 
\begin{equation}
    \Norm{
    g_{\iota, S_{n'}}^{*(n'')}
    }_2^2 \lessapprox 
    \pnorm{
    \prod_{j=n'}^{n''} r_j^{-\frac{1}{2}}
    }
    r_{n''}^{-\frac{n-n''-1}{2}}
    \max_{\tau: \ell(\tau)=r_{n''}^{-1/2}}
    \Norm{
    g_{\iota, S_{n'}}^{*(n'')}
    }_{
    L^2_{\mathrm{avg}}(\tau)
    }^2
\end{equation}
\end{corollary}

\section{Brooms}\label{220706section6}

\subsection{Definition of brooms}\label{220926sub7_1}

Let $(S, B(\bfx_0, r))$ be a grain of dimension $n'$ with $S\in \mf{S}_{n'}$ and assume that it is the last entry of a multi-grain $\vector{S}$. Throughout this section, we always assume that 
\begin{equation}
    r\ge \sqrt{R}.
\end{equation}
Recall Definition \ref{220613definition6_2} and Definition \ref{220726def5_10}. Define $\T[S]\subset \T[B(\bfx_0, r)]$ to be the collection of tubes that are tangent to $S$ in the ball $B(\bfx_0, r)$. Define 
\begin{equation}
    \Theta[S]:=\{\theta(T):  T\in \T[S]\}. 
\end{equation}
Moreover, define $\T_R[S]:=\T_R[\vec{S}]$.

Before we define brooms, we cut each $S$ into $O_{d, n}(1)$ many pieces so that the tangent spaces of each piece form a small angle with each other. Let us be more precise. For each $\bfz\in S$, let $T_{\bfz} S$ denote the tangent space of $S$ at $\bfz$. We cut $S$ into $O_{d, n}(1)$ many pieces $\{S', S'', \dots\}$ so that for each such piece, say $S'$, it holds that 
\begin{equation}\label{220613e8_1}
    \measuredangle(T_{\bfz_1} S', T_{\bfz_2} S')\le \frac{1}{100n}, 
\end{equation}
for $\bfz_1, \bfz_2\in S'$. Similarly, we define $\T[S'], \Theta[S']$ and $\T_R[S']$.  Such a decomposition only appears in this section. To simplify notation, in the rest of this section we will still use $S$ to refer to each such piece, and still call it a grain. \\

Fix $\tau\in \Theta[S]$ and a grain $S$ satisfying \eqref{220613e8_1}. Define 
\begin{equation}
    \T_{\tau, R}[S]:=\{T\in \T_R[S]: \theta(T)\subset \tau\}.
\end{equation}
We let $R$-tubes $T\in \T_{\tau, R}[S]$ intersect $S$. Morally speaking, $T\cap S$ can be thought of as a ``curved'' rectangular box of dimensions 
\begin{equation}
    r\times \underbrace{R^{1/2+\delta}\times\dots\times R^{1/2+\delta}}_{(n'-1) \mathrm{ copies}}\times \underbrace{r^{1/2+\delta_m}\times\dots\times r^{1/2+\delta_m}}_{(n-n') \mathrm{ copies}}.
\end{equation}
For two tubes $T_1, T_{2}\in \T_{\tau, R}[S]$, we say that 
\begin{equation}
    T_1\cap S \cap B(\mathbf{x}_0, r)\approx T_{2}\cap S\cap B(\mathbf{x}_0, r)
\end{equation}
if 
\begin{equation}
    T_1\cap S\cap B(\mathbf{x}_0, r) \subset (10 n T_{2})\cap S \cap B(\mathbf{x}_0, r),
\end{equation}
or the other way around. Before we study the geometry of $S_{\Box}$, let us assume without loss of generality that the tangent space $T_{\bfz}(S)$ forms an angle $\le 1/(100n)$ with the subspace spanned by $\{\vector{e}_1, \dots, \vector{e}_{n'-1}, \vector{e}_n\}$, the first $(n'-1)$ vectors from the orthonormal basis and the vertical $t$ coordinate direction $\vector{e}_n$, for every $\bfz\in S$.  
\begin{lemma}\label{220608lemma8_1}
\begin{enumerate}
    \item[(1)] We can write 
\begin{equation}\label{eq: SBox}
    S\supseteq \bigcup_{\Box}S_{\Box}
\end{equation}
where $S_{\Box}=T\cap S\cap B(\mathbf{x}_0, r)$ for some $T\in \T_{\tau, R}[S]$ and    $\{S_{\Box}\}_{\Box}$ is a disjoint collection. 
%finitely overlapping (bounded by $R^{O(\delta)}$).
Moreover, for every $T\in \T_{\tau, R}[S]$, we can find $S_{\Box}$ such that $T\cap S \cap B(\mathbf{x}_0, r)\approx S_{\Box}$. 
\item[(2)] Take $(x_1, t_1)\in S$. For each $S_{\Box}$, we can find an algebraic variety $Z\subset \{t=t_1\}$ of dimension $n'-1$ and complexity $O(\deg(S))$ satisfying that the angle between $T_{\bfz}(Z)$ and the subspace $\{\vec{e}_1, \dots, \vec{e}_{n'-1}\}$ is $\le 1/(100n)$ for every $\bfz\in Z\cap S_{\Box}$, such that 
\begin{equation}
    (B(x_1, r)\times \{t_1\})\cap S_{\Box}\subset \mc{N}_{r^{1/2}}(Z).
\end{equation}
Here $B(x_1, r)$ is the ball in $\R^n$ of radius $r$ centered at $x_1$. 
\end{enumerate}
\end{lemma}
\begin{proof}[Proof of Lemma \ref{220608lemma8_1}]
If $\mathbb{T}_{\tau, R}[S]$ is empty, then define the right-hand side of \eqref{eq: SBox} as an  empty set.  Now assume that $\mathbb{T}_{\tau, R}[S]$ is nonempty. Pick  $T\in \mathbb{T}_{\tau, R}[S]$ and define $S_{\Box}= T\cap S\cap B(\mathbf{x}_0, r)$. If there exists $T'\in \mathbb{T}_{\tau, R}[S]$ and $T\cap T'\cap S\cap B(\mathbf{x}_0, r)\neq \emptyset$, then $T\cap S\cap B(\mathbf{x}_0, r) \approx T'\cap S \cap B(\mathbf{x}_0, r)$.  To see this, if $\mathbf{x}\in T\cap T'$ and $\theta(T), \theta(T')\subset \tau$, then %\todo{Shaoming: did we prove this somewhere?}
\[T\cap B(\mathbf{x}, 10 r) \approx T'\cap B(\mathbf{x}, 10r)\]
and $B(\mathbf{x}_0, r)\subset B(\mathbf{x},r)$. 

If there exists $T'\in \mathbb{T}_{\tau, R}[S]$, then we add $S_{\Box}':=T'\cap S \cap B(\mathbf{x}_0, r)$ to the right-hand side of \eqref{eq: SBox}. Continue until for every $T\in \mathbb{T}_{\tau, R}[S]$, there exists $S_{\Box}$ from the right-hand side of \eqref{eq: SBox} such that $S_{\Box}=T\cap S \cap B(\mathbf{x}_0, r)$.  

For each $S_{\Box}$, define $Z=S\cap \{t=t_1\}$. \end{proof}

To define brooms, we fix $(S, B(\bfx_0, r)), \vector{S}$, $\tau$ and $S_{\Box}$. Write $\bfx_0=(x_0, t_0)$. Define 
\begin{equation}\label{220615e8_8}
    \T_{\tau, R}[S_{\Box}]:=\{T\in \T_{\tau, R}[S]: T\cap S_{\Box}\neq \emptyset\}. 
\end{equation}
Let us record the following lemma that will be useful later. 
\begin{lemma}\label{220608lemma8_2}
Under the above notation, we have that
\begin{equation}
    \bigcup_{T\in \T_{\tau, R}[S_{\Box}]} \Big(T\cap \{(x, t_1)\in \R^n: x\in \R^{n-1}\}\Big)
\end{equation}
is contained in an $(n-1)$ dimensional ball of radius $R^{1+\delta} r^{-1/2}$, for every $|t_1-t_0|\le R$. 
\end{lemma}
\begin{proof}[Proof of Lemma \ref{220608lemma8_2}]
By translation, we assume that $S_{\Box}$ contains the origin. Let $\omega_0$ be the center of $\tau$. For $\omega\in \tau$, let $x=X_{\omega}(t)$ denote the solution to 
\begin{equation}
    \nabla_{\omega} \phi(x, t; \omega)=0,
\end{equation}
for $|t|\le 1$. Then we need to show that 
\begin{equation}
    |X_{\omega_0}(t)-X_{\omega}(t)|\lesim r^{-1/2},
\end{equation}
for every $t$. Note that 
\begin{equation}
    \nabla_{\omega} \phi(X_{\omega}(t), t; \omega)-\nabla_{\omega} \phi(X_{\omega_0}(t), t; \omega_0)=0.
\end{equation}
By Taylor's expansion, this further implies 
\begin{equation}
    \nabla_x \nabla_{\omega} \phi(x', t; \omega)(X_{\omega}(t)-X_{\omega_0}(t))+\nabla^2_{\omega} \phi(X_{\omega_0}(t), t; \omega')(\omega-\omega_0)=0 
\end{equation}
for some $x', \omega'$. The desired bound follows from the fact that $\nabla_x \nabla_{\omega} \phi$ is non-degenerate. 
\end{proof}

We apply the following algorithm. Initialize 
\begin{equation}
    \T_0:=\T_{\tau, R}[S_{\Box}], \ 
    \D_0=\{T\cap \{t=t_0+R\}: T\in \T_{\tau, R}[S_{\Box}]\}. 
\end{equation}
Suppose we are at the $\ell'$-th step of the algorithm. Let $Z\subset \{t=t_0+R\}$ be an algebraic variety of dimension $n'-1$ and complexity $O(\deg(S))$ and satisfy that the angle between $T_{\bfz}(Z)$ and the subspace spanned by $\{\vector{e}_1, \dots, \vector{e}_{n'-1}\}$ is $\le 1/(100n)$, for every $\bfz\in Z$. Find such a $Z$ that maximizes 
\begin{equation}\label{220608e8_10}
    \#\{D_0\in \D_0: D_0\subset \mc{N}_{10n R^{1/2+\delta}}(Z)\};
\end{equation}
use $b_{\ell'}$ to refer to the number in \eqref{220608e8_10}. 
Remove the discs \eqref{220608e8_10} from $\D_0$, use $\T_{\ell'}$ to collect the tubes $T$ from $\T_0$ for which $T\cap \{t=t_0+R\}$ is removed from $\D_0$ in this step, and repeat this process until there is no any discs left. 

Suppose this algorithm terminates after $L$ steps. We obtain a collection of positive integers 
\begin{equation}
    b_1\ge b_2\ge \dots \ge b_{L}
\end{equation}
and a collection of tubes 
\begin{equation}\label{220608e8_13}
    \T_{1}, \T_{2}, \dots, \T_{L}.
\end{equation}
%By losing a $\log$ factor, we can assume that $b_{\ell'}/2\le b_{\ell''}\le 2 b_{\ell'}$ for two arbitrary $\ell', \ell''$. 
We group $\{b_{\ell'}\}_{\ell'}$ by checking which interval from
\begin{equation}
    [1, R^{\delta}), (R^{\delta}, R^{2\delta}], (R^{2\delta}, R^{3\delta}], \dots
\end{equation}
they belong to: 
\begin{equation}
    \{b_{\ell_0}, \dots, b_{\ell_1}\}, \{b_{\ell_1+1}, \dots, b_{\ell_2}\}, \dots
\end{equation}
with $\ell_0=1$, and therefore two $b_{\ell'}, b_{\ell''}$ in the same group are comparable up to factor $R^{\delta}$. Now we are ready to define brooms.

\begin{definition}[Brooms]
Fix $(S, B(\bfx_0, r)), \vector{S}, \tau$ and $S_{\Box}$. Each 
\begin{equation}\label{220608e8_14}
    \mc{B}_{\ell, b}:=\bigcup_{\ell_{m}+1\le \ell'\le \ell_{m+1}}\{\T_{\ell'}\},
\end{equation}
(see \eqref{220608e8_13} for definition of $\T_{\ell'}$), with level 
\begin{equation}\label{220717e6_19}
    \ell:=R^{w\delta}, \text{ with } w\in \N, R^{w\delta}\le \ell_{m+1}-\ell_m< R^{(w+1)\delta},
\end{equation}
is called a broom. Here
\begin{equation}
    b:=R^{w' \delta} \text{ with } w'\in \N, R^{w'\delta}\le b_{\ell_m+1}< R^{(w'+1)\delta},
\end{equation}
will be called the length of the broom, and in the definition of $\ell$, we used $[a]$ to denote the largest integer $\le a$.  For the broom $\mc{B}_{\ell, b}$ in \eqref{220608e8_14}, we say that it is rooted at $S_{\Box}$. 
\end{definition}

In the previous definition, we used tubes from $\T_R[S]$. For some perhaps technical reasons, we need to introduce the notion of brooms by using a sub-collection of tubes from $\T_R[S]$. Fix $(S, B(\bfx_0, r)), \vector{S}, \tau$ and $S_{\Box}$, and a sub-collection $\T'_R[S]\subset \T_R[S]$. We repeat the above definition of brooms with $\T_R[S]$ replaced by $\T'_R[S]$, and obtain a unique decomposition
\begin{equation}\label{220608e8_16}
    \T'_R[S]=\bigcup_{\ell, b} \mc{B}_{\ell, b}(\T'_R[S]),
\end{equation}
where each $\mc{B}_{\ell, b}(\T'_R[S])$ is called a broom of level $\ell$ and length $b$, and generated by tubes from $\T'_R[S]$.

\subsection{Definition of the two-ends relation}

Recall the algorithm in Subsection \ref{220607subsection6_2}. For each node $\mf{n}\in \cup_{\iota} \mf{R}_{\iota}$,  we will define a relation $\sim_{\mf{n}}$; Lemma \ref{220608lemma6_5} guarantees that we have a small number of these relations. \\

We define a few auxiliary functions $\chi_{\mf{n}, \kappa}=\chi_{\kappa}$, taking values $0$ or $1$, where $\kappa=((\ell_1, b_1), \mu_1, \dots, (\ell_{\iota}, b_{\iota}), \mu_{\iota})$, $\iota\in \N$ and $\ell_{\iota'}, b_{\iota'}, \mu_{\iota'}\in \{R^{w\delta}: w\in \N\}$ for every $1\le \iota'\le \iota$. Denote 
\begin{equation}
    r=\rho(\mf{n}), \ \ m=\dim(\mf{n}). 
\end{equation}
\noindent \underline{Step 1.} For $S\in \mf{n}$ and a tube $T\in \T[B_R]$, 
we say that 
\begin{equation}
    \chi_{(\ell_1, b_1)}(S, T)=1
\end{equation}
if $T$ belongs to a broom rooted at some $S_{\Box}\subset S$ with level $\ell_1$ and length $b_1$. Moreover, we say that 
\begin{equation}
    \chi_{(\ell_1, b_1), \mu_1}(S, T)=1
\end{equation}
if 
\begin{equation}
    \chi_{(\ell_1, b_1)}(S, T)=1
\end{equation}
and 
\begin{equation}
     \mu_1\le \sum_{S'\in \mf{n}}\chi_{(\ell_1, b_1)}(S', T)< \mu_1 R^{\delta}. 
\end{equation}

\noindent \underline{A general step.}  Suppose we have defined $\chi_{\kappa}$ for $\kappa=((\ell_1, b_1), \mu_1, \dots, (\ell_{\iota}, b_{\iota}), \mu_{\iota})$, and $\iota\ge 1$. Let us define 
\begin{equation}
    \chi_{\kappa, (\ell_{\iota+1}, b_{\iota+1})}, \ \  \chi_{\kappa, (\ell_{\iota+1}, b_{\iota+1}), \mu_{\iota+1}}. 
\end{equation}
For fixed $S\in \mf{n}$, define 
\begin{equation}\label{220615e8_26}
    \T_{S, \kappa}:=\{T'\in \T[B_R]: \chi_{\kappa}(S, T')=1\}.
\end{equation}
Recall \eqref{220608e8_16}. Write 
\begin{equation}
    \T_{S, \kappa}=\bigcup_{\ell_{\iota+1}, b_{\iota+1}} \{ \mc{B}_{\ell_{\iota+1}, b_{\iota+1}, \tau, S_{\Box}}( \T_{S, \kappa})\}_{\tau, S_{\Box}},
\end{equation}
where $\tau$ runs through all frequency caps of side length $\rho(\mf{n})^{-1/2}$,  $S_{\Box}$ is as given in Lemma \ref{220608lemma8_1}, $\mc{B}_{\ell_{\iota+1}, b_{\iota+1}, \tau, S_{\Box}}( \T_{S, \kappa})$ is a broom of level $\ell_{\iota+1}$, length $b_{\iota+1}$, rooted at $S_{\Box}$ and generated by tubes from $\T_{S, \kappa}$. We then say that 
\begin{equation}
    \chi_{\kappa, (\ell_{\iota+1}, b_{\iota+1})}(S, T)=1 \text{ if } T\in \mc{B}_{\ell_{\iota+1}, b_{\iota+1}, \tau, S_{\Box}}( \T_{S, \kappa, \tau}),
\end{equation}
for some $\tau$ and $S_{\Box}$. Next, set 
\begin{equation}
    \chi_{\kappa, (\ell_{\iota+1}, b_{\iota+1}), \mu_{\iota+1}}=1
\end{equation}
if 
\begin{equation}
    \chi_{\kappa, (\ell_{\iota+1}, b_{\iota+1})}(S, T)=1
\end{equation}
and 
\begin{equation}
    \mu_{\iota+1}\le \sum_{S'\in \mf{n}} \chi_{\kappa, (\ell_{\iota+1}, b_{\iota+1}), \mu_{\iota+1}}(S', T)< \mu_{\iota+1} R^{\delta}.
\end{equation}
This finishes the definition of the auxiliary functions we need. \\

For $\kappa=((\ell_1, b_1), \mu_1, \dots, (\ell_{\iota}, b_{\iota}), \mu_{\iota})$, we say that $\kappa$ is admissible if there exists exactly one pair $(\iota_1, \iota_2)$ with $\iota_1\neq \iota_2$ such that 
\begin{equation}
    ((\ell_{\iota_1}, b_{\iota_1}), \mu_{\iota_1})=((\ell_{\iota_2}, b_{\iota_2}), \mu_{\iota_2}).
\end{equation}
\begin{lemma}\label{220608lemma8_5}
The number of admissible $\kappa$ is $O_{\delta}(1)$. 
\end{lemma}
\begin{proof}[Proof of Lemma \ref{220608lemma8_5}]
Note that by \eqref{220717e6_19}, the number of values that $\ell_{\iota'}$ can take is $\delta^{-1}$; the same is true is for $b_{\iota'}$ and $\mu_{\iota'}$, for each $\iota'$. The lemma follows. 
\end{proof}

\begin{definition}
For a ball $B\subset B_R$ of radius $R^{1-\delta}$, a tube $T\in \T[B_R]$, a node $\mf{n}\in \cup_{\iota}\mf{R}_{\iota}$ and an admissible multi-index $\kappa$, we say that $B\sim_{\mf{n}, \kappa} T$ if $B$ maximizes 
\begin{equation}
    \#\{S'\in \mf{n}: S'\subset B', \chi_{\mf{n}, \kappa}(S', T)=1\},
\end{equation}
among all $B'$ of radius $R^{1-\delta}$. 
\end{definition}

\begin{definition}[Relation]
For a ball $B$ of radius $R^{1-\delta}$ and a tube $T\in \T[B_R]$, we say that $B\sim T$ if 
\begin{equation}
    B\sim_{\mf{n}, \kappa} T,
\end{equation}
for some node $\mf{n}\in \cup_{\iota} \mf{R}_{\iota}$ and admissible $\kappa$. 
\end{definition}

By a simple inductive argument on $\kappa$, we have 
\begin{lemma}\label{220615lemma8_7}
Given a multi-grain $\vector{S}$ with the last component $S$. For every $T\in \T_R[S]$, there exists exactly one admissible $\kappa$ such that $T\in \T_{S, \kappa}$. 
\end{lemma}

\subsection{Broom estimates}\label{220926sub7_3}

Let $\vector{S}_{n'}$ be a multigrain from Definition \ref{220726def5_10} with the last component given by $S_{n'}$. Let $B_{\iota}\subset B_{R}$ be the ball of radius $R^{1-\delta}$ that contains $S_{n'}$. Recall the definition of $f^*_{\iota, S_{n'}}$ from Subsection \ref{220610subsection6_4} and the definition of $f_{\iota, S_{n'}}^{*(n'')}$ with $n'\le n''\le n$ from \eqref{220726e4_109}. The notation $*(n'')$ means that we start with the function $f^*_{\iota, S_{n'}}$, which is defined via wave packets from $\T[B_{r_{n''}}]$, and ``trace" back by Definition \ref{220726def5_10} along the nodes in \eqref{220711e5_88} to wave packets in $\T[B_{r_{n''}}]$. 
Note that
\begin{equation}
    f_{\iota, S_{n'}}^{*(n'')}=f^*_{\iota, S_{n'}}, \text{ when } n''=n'.
\end{equation}
Define 
\begin{equation}\label{220617e7_36}
    f^{\not\sim}_{S_{n'}, \tau}:=(f_{\tau})^*_{\iota, S_{n'}}. 
\end{equation}
Here we have suppressed the dependence on $\iota$ and replaced it by $\not\sim$, as $B_{\iota}$ is uniquely determined by $S$, and we would also like to emphasize that we are in the non-related case. Moreover, define 
\begin{equation}\label{220617e7_36zz}
    f^{\not\sim(n'')}_{S_{n'}, \tau}:=(f_{\tau})^{*(n'')}_{\iota, S_{n'}}.
\end{equation}
The main goal of this subsection is to prove the following broom estimate. 

\begin{theorem}\label{220615thm8_7}
Let
$\vector{S}_{n'}=(S_n, \dots, S_{n'})$ be a multi-grain and
$S_{n'}=\mc{N}_{r_{n'}^{1/2+\delta_{n'}}}(Z_{n'})\cap B_{r_{n'}}$ with $r_{n'}\ge \sqrt{R}$, $Z_{n'}$ is an $n'$-dimensional algebraic variety of degree $\lesssim_{n'} 1$ in $\mathbb{R}^n$, then for $\tau$ of scale $r_{n'}^{-1/2}$ and $n'\le n''<n$, it holds that 
\begin{equation}
\|f^{\nsim(n'')}_{S_{n'}, \tau}\|_{L^2}^2 \lessapprox \Big(\frac{r_{n''}}{R}\Big)^{\frac{n-n'}{2}}
\|f_{\tau}\|_{L^2}^2.
\end{equation}
Here $r_i=\rho(S_i)$ for each $n'\le i\le n$. 
\end{theorem}

\begin{proof}[Proof of Theorem \ref{220615thm8_7}] We will first write down the details for the case $n''=n'$, as it requires less notation and contains all the ideas of the proof; the general case $n''\ge n'$ will be remarked in the end. Our goal is to prove 
\begin{equation}
\|f^{\nsim}_{S_{n'}, \tau}\|_{L^2}^2 \lessapprox \Big(\frac{r_{n'}}{R}\Big)^{\frac{n-n'}{2}}
\|f_{\tau}\|_{L^2}^2.
\end{equation}
To simplify notation, we will abbreviate $S_{n'}$ to $S$, $\vector{S}_{n'}$ to $\vector{S}$ and $r_{n'}$ to $r$. For the ball $B_{\iota}$ of radius $R^{1-\delta}$ containing $S$, the node $\mf{n}$ containing $S$ and each admissible multi-index $\kappa$, recall the notation in \eqref{220615e8_26} and 
define \begin{equation}\label{eq: Tnsimkappatau}
\mathbb{T}^{\nsim}_{S, \kappa,\tau}:=\{ T \in \mathbb{T}_{S, \kappa}: \theta(T)\subset \tau,  B_{\iota}\nsim_{\mf{n}, \kappa} T\},
\end{equation}
and 
\[
f^{\nsim}_{\kappa, S, \tau}:= \sum_{T\in \mathbb{T}^{\nsim}_{S, \kappa,\tau}}f_T \quad \text{and} \quad 
f^{\nsim, *}_{\kappa, S, \tau} := (f^{\nsim}_{\kappa, S, \tau} )^*_{\iota, S}.
\]
Then we claim that
\begin{equation}\label{220615e8_36}
    f^{\nsim}_{S,\tau} = \sum_{\kappa}  f^{\nsim, *}_{\kappa, S, \tau}.
\end{equation}
To see this, note that for every $T\in \T[B_R]$ that is tangent to $S$, by Lemma \ref{220615lemma8_7}, there exists exactly one admissible $\kappa$ such that $T\in \T_{S, \kappa}$. Moreover, if $T\not\sim B_{\iota}$, then for the above given $\mf{n}$ and $\kappa$, we also have $T\not\sim_{\mf{n}, \kappa} B_{\iota}$. \\

By the Cauchy-Schwarz inequality, 
\begin{equation}\label{220617e7_38}
    \norm{f^{\nsim}_{S,\tau}}^2_2 \lesim_{\delta}  \sum_{\kappa} \norm{ f^{\nsim, *}_{\kappa, S, \tau}}^2_2.
\end{equation}
Here we used Lemma \ref{220608lemma8_5}. Next, we localize $f^{\sim}_{\kappa, S, \tau}$ further in space. Recall from Lemma \ref{220608lemma8_1} that $S=\cup_{\Box} S_{\Box}$. Denote 
\begin{equation}
    \mathbb{T}^{\nsim}_{S_{\Box}, \kappa,\tau}:=\{ T \in \mathbb{T}_{S, \kappa}: T\cap S_{\Box}\neq\emptyset, \theta(T)\subset \tau,  B_{\iota}\nsim_{\mf{n}, \kappa} T\},
\end{equation}
and then 
\begin{equation}
    f^{\nsim}_{\kappa, S_{\Box}, \tau}:= \sum_{T\in \mathbb{T}^{\nsim}_{S_{\Box}, \kappa,\tau}}f_T \quad \text{and} \quad 
f^{\nsim, *}_{\kappa, S_{\Box}, \tau} := (f^{\nsim}_{\kappa, S_{\Box}, \tau} )^*_{\iota, S}.
\end{equation}
By spatial disjointness of $\{S_{\Box}\}_{\Box}$, we have 
\begin{equation}
    \eqref{220617e7_38}\lesim \sum_{\kappa} 
    \sum_{\Box} \norm{
    f^{\nsim, *}_{\kappa, S_{\Box}, \tau}
    }_2^2.
\end{equation}
Let $(x_1, t_1)$ be a point in $S$, then 
\begin{equation}\label{220627e7_42}
\|f^{\nsim, *}_{\kappa, S_{\Box},  \tau}\|_{L^2}^2 \lesim 
\|T^{\lambda} f^{\nsim, *}_{\kappa, S_{\Box},\tau}\|_{L^2( (B(x_1, 2r_{n'}) \times \{t_1\})\cap S_{\Box})}^2
\end{equation}
Here $B(x_1, 2r_{n'})$ is an $n-1$ dimensional ball. Indeed, inequalities of the form \eqref{220627e7_42} hold for more general data. Let $h_{\tau}$ be supported on $\tau$. Assume that $T^{\lambda} h_{\tau}(\cdot, t_1)$ is essentially supported on $B_{\sqrt{R}}\subset \R^{n-1}$, a ball of radius $\sqrt{R}$. Then 
\begin{equation}\label{220627e7_44}
    \norm{h_{\tau}}_2^2 \lesim \norm{T^{\lambda} h_{\tau}}_{L^2
    (
    B_{\sqrt{R}}\times \{t_1\}
    )
    }.
\end{equation}
To see this, we first apply the change of variables 
\begin{equation}\label{220627e7_45}
    x\mapsto x+x_0, t\mapsto t+t_1, \xi\mapsto \xi+\xi_0,
\end{equation}
where $x_0$ is the center of $B_{\sqrt{R}}$ and $\xi_0$ is the center of $\tau$. We obtain 
\begin{equation}
    T^{\lambda} h_{\tau}(x+x_0, t+t_1)=\int h_{\tau}(\xi+\xi_0) e^{i \phi^{\lambda}(x+x_0, t+t_1; \xi+\xi_0)}a^{\lambda}(x+x_0, t+t_1; \xi+\xi_0) d\xi.
\end{equation}
Its absolute value can be written as the absolute value of 
\begin{equation}
    \begin{split}
        & \int \widetilde{h}_{\tau}(\xi) e^{i\phi^{\lambda}_0(x, t; \xi)} a^{\lambda}(x+x_0, t+t_1; \xi+\xi_0) d\xi
    \end{split}
\end{equation}
where 
\begin{equation}\label{220612e8_43}
\begin{split}
    \phi^{\lambda}_0(x, t; \xi):=
    & \phi^{\lambda}(x+x_0, t+t_1; \xi+\xi_0)-\phi^{\lambda}(x_0, t_1; \xi+\xi_0)\\
    & - \phi^{\lambda}(x+x_0, t+t_1; \xi_0)+ \phi^{\lambda}(x_0, t_1; \xi_0)
\end{split}
\end{equation}
and 
\begin{equation}
    \widetilde{h}_{\tau}(\xi):=h_{\tau}(\xi+\xi_0) e^{i\phi^{\lambda}(x_0, t_1; \xi+\xi_0)}
\end{equation}
This change of variables tells us that, if we denote 
\begin{equation}\label{220611e8_44}
    \widetilde{T}^{\lambda} h_{\tau}(x, t):=\int h_{\tau}(\xi) e^{i \phi_0^{\lambda}(x, t; \xi)}a^{\lambda}(\bfx; \xi) d\xi,
\end{equation}
then to prove \eqref{220627e7_44}, it suffices to prove it for the operator as defined in \eqref{220611e8_44} with $t_1=0, x_0=0, \xi_0=0$. 

We apply Taylor expansion to $\phi_0^{\lambda}$ in the $x$ variable about the origin, and write 
\begin{equation}\label{220627e7_51}
    \phi_0^{\lambda}(x, 0; \xi)=\nabla_{x}\phi_0^{\lambda}(0; \xi)\cdot x+ w_1(x; \xi).
\end{equation}
Note that 
\begin{equation}
    |\nabla_x^2 \phi_0^{\lambda}| \lesim 1/\lambda, \ \ |x|\lesim \sqrt{R},
\end{equation}
and therefore $|w_1(x; \xi)|\lesim 1$. We apply Taylor expansion in the $\xi$ variable about the origin, and obtain 
\begin{equation}\label{220627e7_53}
    \phi_0^{\lambda}(x, 0; \xi)=\inn{x}{\xi}+\underbrace{w_2(x; \xi)+ w_1(x; \xi)}_{=: w(x; \xi)}.
\end{equation}
Note that 
\begin{equation}
    |\nabla_x\nabla^{\beta}_{\xi} \phi_0^{\lambda}|\lesim_{\beta} 1, 
\end{equation}
for all multi-indices $\beta$ and therefore $|w_2(x; \xi)|\lesim r^{-1} \sqrt{R}\lesim 1$. Now the claimed estimate \eqref{220627e7_44} follows from Plancherel's theorem and Taylor's expansion. \\

Suppose that
\begin{equation}
    \kappa=\{ (\ell_1, b_1), \mu_1, \dots, (\ell_{j}, b_{j}), \mu_{j}).
\end{equation}
For each plank $S_{\Box}$, 
\[
\mathbb{T}_{\tau, R}[S_{\Box}] \cap \mathbb{T}_{S, \kappa},
\]
where $\T_{\tau, R}[S_{\Box}]$ was defined in \eqref{220615e8_8}, is contained in a broom $\mathcal{B}_{\ell_{j}, b_{j}}$. Recall the definition of brooms and the notation in \eqref{220608e8_13}. Write 
\begin{equation}\label{220628e7_57}
    \mathcal{B}_{\ell_{j}, b_{j}}=\bigcup_{1\le \ell'\le \ell_j} \T_{\ell'},
\end{equation}
and 
\begin{equation}\label{220628e7_58}
    \bigcup_{T\in \T_{\ell'}} T \cap \{t=t_1+R\}\subset \mathcal{N}_{10nR^{1/2+\delta}}(Z_{\ell'})
\end{equation}
for an algebraic variety $Z_{\ell'}$  of dimension $n'-1$ and complexity $O(\deg(S))$ satisfying that the angle between $T_{\mathbf{z}}(Z_{\ell'})$ and the space spanned by $\{\vec{e}_1, \dots, \vec{e}_{n'-1}\}$ is $\leq 1/(100n)$, for every $\mathbf{z}\in Z_{\ell'}$.  Write 
\[
f^{\nsim}_{\kappa, S_{\Box}, \tau, \ell'}=\sum_{T\in \mathbb{T}^{\nsim}_{S_{\Box}, \kappa, \tau}\cap \mathbb{T}_{\ell'}} f_T
\]
and 
\[
f^{\nsim, *}_{\kappa, S_{\Box}, \tau, \ell'} :=  (f^{\nsim}_{\kappa, S_{\Box}, \tau, \ell'})^*_{\iota, S}. 
\]
Then by the triangle inequality and Cauchy-Schwarz inequality, 
\begin{equation}\label{220616e8_41}
\begin{split}
& \| T^{\lambda} f^{\nsim, *}_{\kappa, S_{\Box}, \tau} \|_{L^2( (B(x_1, 2r) \times \{t_1\}) \cap S_{\Box})}^2  \\
& \lesssim \ell_{j} \sum_{1\le \ell'\le \ell_j} \|T^{\lambda} f^{\nsim, *}_{\kappa, S_{\Box}, \tau, \ell'} \|_{L^2( (B(x_1, 2r) \times \{t_1\}) \cap S_{\Box})}^2.
\end{split}
\end{equation}
\begin{claim}\label{220616claim8_9}
For each $\ell'$, it holds that 
\begin{equation}\label{220616e8_42}
\begin{split}
& \|T^{\lambda} f^{\nsim, *}_{\kappa, S_{\Box}, \tau, \ell'}\|_{L^2( (B(x_1, 2r) \times \{t_1\}) \cap S_{\Box})}^2 \\
& \lessapprox  (\frac{r}{R})^{\frac{n-n'}{2}}
\|T^{\lambda}f^{\nsim}_{\kappa, S_{\Box}, \tau, \ell'}\|_{L^2(\{ t=t_2\})}^2
\end{split}
\end{equation}
where $t_2:=t_1+R$. 
\end{claim}
We first accept Claim \ref{220616claim8_9}, and continue with the $L^2$ estimate: 
\begin{equation}
    \eqref{220616e8_42} \lessapprox
    (\frac{r}{R})^{\frac{n-n'}{2}} 
\|f^{\nsim}_{\kappa, S_{\Box}, \tau, \ell'}\|^2_{L^2}.
\end{equation}
Summing over all $\ell'$ and $S_{\Box}$, we obtain 
\[
\|T^{\lambda} f^{\nsim}_{S, \tau}\|_{L^2 (B(x_1, 2r)\times \{t_1\}) }^2 \lessapprox \ell_{j}  (\frac{r}{R})^{\frac{n-n'}{2}}
\|f^{\nsim}_{\kappa, S, \tau}\|_{L^2}^2. 
\]
By Lemma~\ref{lem: counting} below and the assumption that all wave packets have comparable coefficients, we conclude that 
\[
\|f^{\nsim}_{\kappa, S, \tau}\|_{L^2}^2 \lesssim \ell_{j}^{-1} R^{O(n\delta)}\|f_{\tau}\|_{L^2}^2. 
\]
This finishes the proof of the theorem, modulo the proofs of Lemma \ref{lem: counting} and Claim \ref{220616claim8_9}. 
\end{proof}

\begin{lemma}\label{lem: counting}
Let $\kappa= ( (\ell_1, b_1), \mu_1, \dots, (\ell_{j}, b_{j}),\mu_{j}) $ be an admissible multi-index and 
$\mathbb{T}^{\nsim}_{S, \kappa, \tau}$ be defined as in \eqref{eq: Tnsimkappatau}. We have  
\[
| \mathbb{T}^{\nsim}_{S, \kappa, \tau} |\lesssim  \ell_{j}^{-1} R^{O(n\delta)} |\mathbb{T}_{\tau}|,
\]
where $\T_{\tau}$ collects tubes $T\in \T[B_R]$ with $\theta(T)\subset \tau$. 
\end{lemma}
\begin{proof}[Proof of Lemma \ref{lem: counting}]
Since $\kappa$ is admissible, there exists
\begin{equation}
    \kappa'= ( (\ell_1, b_1), \mu_1, \dots, (\ell_{j}, b_{j'}),\mu_{j'})
\end{equation}
such that
\begin{equation}
    \kappa= (\kappa', (\ell_{j'+1}, b_{j'+1}), \mu_{j'+1}, \dots (\ell_{j}, b_{j}), \mu_j )
\end{equation}
and 
\begin{equation}\label{eq: admissible}
((\ell_{j'}, b_{j'}),\mu_{j'})=((\ell_{j}, b_{j}), \mu_{j}).
\end{equation}
Let $B$ be a ball of radius $R^{1-\delta}$ containing $S$ and $\mf{n}$ be the node containing $S$, then for each $S\in \mf{n}$ and $S'\nsubset 2B$, we have 
\begin{equation}
\sum_{T\nsim B, T\in \mathbb{T}_{\tau}} \chi_{\mf{n}, \kappa'} (S', T)\chi_{\mf{n}, \kappa}(S, T) \lesssim R^{O(n \delta)}  \ell_{j'}^{-1} \sum_{ T\in \mathbb{T}_{\tau}} \chi_{\mf{n}, \kappa'} (S', T)
\end{equation}
Summing over all $S'\in \mf{n}, S'\nsubset 2B$,
\begin{equation}
    \begin{split}
            & \sum_{S\in \mf{n}, S\not\subset 2B} \sum_{ T\nsim B,  T\in \mathbb{T}_{\tau}} \chi_{\mf{n}, \kappa}(S, T)\chi_{\mf{n}, \kappa'}(S', T)  \\
            & \leq  R^{O(n\delta)}\ell_{j'}^{-1} \sum_{S\in \mf{n}, S\not\subset 2B}  \sum_{  T\in \mathbb{T}_{\tau} }\chi_{\mf{n}, \kappa'} (S', T).
    \end{split}
\end{equation}
On the other hand, for each  $T\nsim B$ and $\chi_{\mf{n, \kappa}}(S, T)=1$, 
\[
\sum_{S'\in \mf{n}, S'\not\subset 2B} \chi_{\mf{n}, \kappa'} (S', T) \geq \mu_{j} R^{-\delta}. 
\]
As a consequence, 
\[
|\mathbb{T}^{\nsim}_{S, \kappa, \tau}| \leq R^{\delta} \mu_{j}^{-1}  \sum_{S'\in \mf{n}, S'\not\subset 2B} \sum_{ T\nsim B,  T\in \mathbb{T}_{\tau}} \chi_{\mf{n}, \kappa}(S, T)\chi_{\mf{n}, \kappa'}(S', T).
\]
Moreover, note that 
\[
\sum_{S'\in \mf{n}, S'\not\subset 2B}  \sum_{  T\in \mathbb{T}_{\tau}}\chi_{\mf{n}, \kappa'} (S', T)
\lesssim R^{\delta} \mu_{j'}  |\mathbb{T}_{\tau}|
\]
The conclusion follows from \eqref{eq: admissible}.
\end{proof}

In the rest of this section, we will prove Claim \ref{220616claim8_9}. We will start with the proof of Lemma \ref{220614lemma8_8}, and this lemma will be an important ingredient in the forthcoming proof of Claim \ref{220616claim8_9}.

\begin{proof}[Proof of Lemma \ref{220614lemma8_8}]
Denote $\sigma:=\sqrt{R_1}/\sqrt{R_2}$, $\Omega'_1=\mc{N}_{1}(Z_1)$ and $\Omega'_2:=\mc{N}_{\sigma}(Z_2)$. By scaling, let us assume that $\supp(F)\subset \Omega'_2$, and we need to prove 
\begin{equation}
    \norm{\widehat{F}}_{L^2(\Omega'_1)}^2 \lesim \sigma^{n-m-\delta} \norm{F}_{L^2}^2.
\end{equation}
Let $K$ be a large number depending on $n, \deg(Z_2), \delta$ and is to be determined.  We cut $Z_2$ into $O_{n, \deg(Z_2), \delta}(1)$ many pieces so that for each piece $Z'_2\subset Z_2$, there exists a linear subspace of dimension $m-1$ satisfying that the angle between $T_{\bfz}(Z'_2)$ and the linear subspace is $\le 1/K$ for every $\bfz\in Z'_2$. As our constant is allowed to depend on $K$, we only need to prove Lemma \ref{220614lemma8_8} for each $Z'_2$.
\begin{claim}\label{221010claim7_11}
Fix $Z'_2$ as above and $K'\ge K$. Let $V\subset \R^{n-1}$ be an $(n-m)$-dimensional affine subspace such that the angle between $T_{\bfz}(Z'_2)$ and $V^{\perp}$ is $\le 1/K$ for every $\bfz\in Z'_2$. Let $S_V$ denote the $1/K'$-neighbourhood of $V$. Then  $Z'_2\cap S_V$ is contained in a union of $O(\deg(Z_2)^{n-1})$ many rectangular boxes of dimensions 
\begin{equation}
    \underbrace{\frac{1}{K'}\times \dots \frac{1}{K'}}_{(m-1) \text{ copies}} \times \underbrace{\frac{1}{K K'}\times \dots \frac{1}{K K'}}_{(n-m) \text{ copies}}
\end{equation}
whose long sides are parallel to $V^{\perp}$. 
\end{claim}

\begin{proof}[Proof of Claim \ref{221010claim7_11}]
Without loss of generality, we may assume $K = e^{-n}$ (all we need here is a small constant)  and $K' = 1$ since we can do an (anisotropic) rescaling otherwise. Since the angle between every $T_{\bfz}(Z'_2)$ and $V^{\perp}$ is $\le e^{-n}$, we see the angle between every $T_{\bfz}(Z'_2)$ and $V$ is $\gtrsim 1$. Hence if we take the union of all points $\bfz \in Z_2 \bigcap S_V$ such that the angle between $T_{\bfz}(Z_2)$ and $V$ is $\gtrsim 1$, it suffices to prove this whole set can be contained in a union of $O(\deg(Z_2)^{n-1})$ many unit balls. Note that when $n-m = 1$, this is proved by Guth as a special case of Lemma 5.7 in \cite{guth2018} (when one takes $r \simeq \alpha \simeq 1$ there). For general $n$ and $m$ this can also be proved by induction on dimension exactly in the same way as in the proof of Lemma 5.7 in \cite{guth2018} and  we thus omit the details.
\end{proof}

We continue to prove Lemma \ref{220614lemma8_8}. We start with a trivial estimate: 
\begin{equation}\label{220614e8_42}
    \norm{\widehat{F}}_{L^2(\mc{N}_{1}(Z_1))} \le \norm{\widehat{F}}_{L^2(\mc{N}_{K}(Z_1))}.
\end{equation}
For a given $K'\ge 1$, let $\mc{P}_{K'}$ be the partition of $\Omega'_2$ into disjoint pieces $\{\Omega'_{2, K'}\}$ and the orthogonal projection of each $\Omega'_{2, K'}$ to $\spa\{\vector{e}_1, \dots, \vector{e}_{m-1}\}$ is a dyadic cube of side length $1/K'$. Denote \begin{equation}
    F_{\Omega'_{2, K'}}=F\cdot \mathbbm{1}_{\Omega'_{2, K'}}.
\end{equation}
By $L^2$ orthogonality and Claim \ref{221010claim7_11}, 
\begin{equation}\label{220614e8_43}
    \eqref{220614e8_42}\lesim \sum_{\Omega'_{2, K}} 
    \norm{
    \widehat{F_{\Omega'_{2, K}}}
    }_{L^2(\mc{N}_K(Z_1))}^2
\end{equation}
By the assumption that the direction of $T_{\bfz}(Z'_2)$ falls in a small angle of size $1/K$ for every $\bfz$, we see that $\Omega'_{2, K}$ is contained in a rectangular box of dimension 
\begin{equation}\label{220614e8_44}
    \underbrace{\frac{1}{K}\times \dots \frac{1}{K}}_{(m-1) \text{ copies}} \times \underbrace{\frac{1}{K^2}\times \dots \frac{1}{K^2}}_{(n-m) \text{ copies}}.
\end{equation}

We now use the uncertainty principle to analyze $\widehat{F_{\Omega'_{2, K}}}$ again with the help of (the $K\simeq K'\simeq 1$ version of)  Claim \ref{221010claim7_11}. Each $|\widehat{F_{\Omega'_{2, K}}}|$ is essentially a constant on dual boxes of dimensions $\underbrace{K\times \dots K}_{(m-1) \text{ copies}} \times \underbrace{K^2\times \dots K^2}_{(n-m) \text{ copies}}$. If we use  Claim \ref{221010claim7_11}, set all parameters in that claim to be $\simeq 1$ and rescale the conclusion by $K$ times, we see that if we tile any given $\underbrace{K\times \dots K}_{(m-1) \text{ copies}} \times \underbrace{K^2\times \dots K^2}_{(n-m) \text{ copies}}$ box as above by $K$-balls, then $Z_1$ only intersects $\simeq 1$ of them.

Therefore, the uncertainty principle and the above geometric observation imply that
\begin{equation}
   \begin{split}
        \eqref{220614e8_43}& \lesim \frac{1}{K^{n-m}} \sum_{\Omega'_{2, K}} \norm{
        \widehat{F_{\Omega'_{2, K}}}
        }_{L^2(\mc{N}_{K^2}(Z_1))}^2\\
        & \lesim \frac{1}{K^{n-m}} \sum_{\Omega'_{2, K^2}} \norm{
        \widehat{F_{\Omega'_{2, K^2}}}
        }_{L^2(\mc{N}_{K^2}(Z_1))}^2,
   \end{split} 
\end{equation}
where in the second inequality we use $L^2$ orthogonality. Similarly to \eqref{220614e8_44}, $\Omega'_{2, K^2}$ is contained in a rectangular box of dimension 
\begin{equation}\label{220614e8_46}
    \underbrace{\frac{1}{K^2}\times \dots \frac{1}{K^2}}_{(m-1) \text{ copies}} \times \underbrace{\frac{1}{K^3}\times \dots \frac{1}{K^3}}_{(n-m) \text{ copies}}
\end{equation}
We continue this process repeatedly, and in the end arrive at 
\begin{equation}
    \norm{
    \widehat{F}
    }_{L^2(\mc{N}_1(Z))}\le C^W \pnorm{\frac{1}{K^W}}^{n-m} \norm{F}_{L^2}^2,
\end{equation}
where $K^W=1/\sigma$. In the end, we pick $K$ to be large enough.  
\end{proof}

Next, we will prove a simpler version of Claim \ref{220616claim8_9}, see Lemma \ref{220608lemma8_9} below. The proof of this lemma contains the main idea of that of Claim \ref{220616claim8_9}, and indeed we will apply Lemma \ref{220608lemma8_9} iteratively to prove Claim \ref{220616claim8_9}.  To simplify notation, let us denote 
\begin{equation}
    F(x, t):=\int f_{\tau}(\xi) e^{i \phi^{\lambda}(x, t; \xi)}a^{\lambda}(\bfx; \xi) d\xi,
\end{equation}
where $\tau$ is a frequency cap.

\begin{lemma}\label{220608lemma8_9}
Let $\sqrt{R}\le r\le R'_1\le R'_2\le R$. Let $\tau$ be a frequency cap of side length $r^{-1/2}$. Given $t_1, t_2\in [0, \lambda]$ with $R= |t_1-t_2|$. Given two $(m-1)$-dimensional algebraic varieties $Z_1\subset \{t=t_1\}$ and $Z_2\subset \{t=t_2\}$ satisfying that the angle formed by $T_{\bfz_i}(Z_i)$ and the space spanned by $\{\vector{e}_1, \dots, \vector{e}_{m-1}\}$ is $\le 1/(100n)$, for every $i=1, 2$ and every $\bfz_i\in Z_i$. Here $T_{\bfz_i}(Z)$ refers to the tangent space. Denote 
\begin{equation}
    \Omega_1=\mc{N}_{\sqrt{R'_1}}(Z_1)\cap B_{\sqrt{R}},
\end{equation}
for a given ball $B_{\sqrt{R}}$ of radius $\sqrt{R}$; denote 
\begin{equation}
    \Omega_2=\mc{N}_{R/\sqrt{R'_2}}(Z_2)\cap B_{R/\sqrt{r}}. 
\end{equation}
Assume that $F(x, t_1)$ is essentially supported on $B_{\sqrt{R}}$ and $F(x, t_2)$ is essentially supported on $\Omega_2$, then 
\begin{equation}
    \norm{F(x, t_1)}_{L^2(\Omega_1)}^2 \lesim \pnorm{\frac{R'_1}{R'_2}}^{\frac{n-m}{2}-O(\delta)}\norm{F(x, t_2)}_{L^2(\Omega_2)}^2.
\end{equation}
\end{lemma}
\begin{proof}[Proof of Lemma \ref{220608lemma8_9}] Let $x_0$ be the center of $B_{\sqrt{R}}$ and $\xi_0$ the center of $\tau$. We apply the same change of variables as in \eqref{220627e7_45}. Recall the new phase function in \eqref{220612e8_43}. If we denote \footnote{Here we still use $F$ just to avoid new notation.}
\begin{equation}\label{220611e7_79}
    F(x, t):=\int f_{\tau}(\xi) e^{i \phi_0^{\lambda}(x, t; \xi)}a^{\lambda}(\bfx; \xi) d\xi,
\end{equation}
then to prove the lemma, we need to prove it for the function $F(x, t)$ as in \eqref{220611e7_79} with $t_1=0$, $x_0=0$ and $\xi_0=0$. Before we finish this reduction step, let us do another linear change of variables so that $\nabla_x\nabla_{\xi} \phi^{\lambda}_0(0; 0)$ is an identity matrix. \\

We claim that $\widecheck{f}_{\tau}$ is essentially supported on $2B_{\sqrt{R}}$. By the same Taylor expansion as in \eqref{220627e7_51}--\eqref{220627e7_53}, we can write
\begin{equation}
    \begin{split}
        \widecheck{f}_{\tau}(x)& =\int f_{\tau}(\xi) e^{ix\cdot \xi}d\xi=\int f_{\tau}(\xi) e^{i\phi_0^{\lambda}(x, 0; \xi)} e^{-iw(x; \xi)}d\xi\\
        &=\sum_{k\in \N} \frac{(-i)^k}{k!} \int f_{\tau}(\xi) e^{i\phi_0^{\lambda}(x, 0; \xi)} w^k(x; \xi) d\xi
    \end{split}
\end{equation}
Next, we do Fourier expansion for $w^k(x; \xi) a^{\lambda}(x, 0; \xi)$ in the $\xi$ variable at the unit scale, and write it as 
\begin{equation}
    \sum_{\beta\in \N^{n-1}} c_{k, \beta}(x) e^{i\beta\cdot \xi},
\end{equation}
where the coefficients satisfy 
\begin{equation}\label{220626e7_80}
    |c_{k, \beta}(x)| \lesim_N 2^k |\beta|^{-N},
\end{equation}
for every large $N$, uniformly in $x$. Now the claim that $\widecheck{f}_{\tau}$ is essentially supported on $2B_{\sqrt{R}}$ follows from the assumption on $F$ and the rapid decay in \eqref{220626e7_80}. Moreover, by a similar Taylor expansion, we obtain 
\begin{equation}
    \norm{F(x, t_1)}_{L^2(\Omega_1)}^2 \lesim \norm{\widecheck{f}_{\tau}}_{L^2(\Omega_1)}^2. 
\end{equation}
It therefore remains to control $\norm{\widecheck{f}_{\tau}}_{L^2(\Omega_1)}^2$ by $\norm{F(x, t_2)}_{L^2(\Omega_2)}^2$. \\

Consider $t=t_2$, and write 
\begin{equation}
\begin{split}
    F(x, t_2) &= \int f_{\tau}(\xi) e^{i\phi^{\lambda}(x, t_2; \xi)} a^{\lambda}(x, t_2; \xi) d\xi\\
    & =\int \widecheck{f}_{\tau}(y) \big( \int  e^{-iy\cdot \xi}  e^{i\phi_0^{\lambda}(x, t_2; \xi)}a^{\lambda}(x, t_2; \xi) d\xi  \big) dy
\end{split}
\end{equation}
Consider the critical point of the phase function: 
\begin{equation}\label{220612e8_55}
\nabla_\xi \phi_0^{\lambda} (x, t_2; \xi)=y. 
\end{equation}
Let $\xi_c=\xi_c(x, t_2; y)$ denote the critical point. Note that by \eqref{220612e8_43} and the assumption that $\phi^{\lambda}$ is in its normal form, we see that 
\begin{equation}\label{220612e8_56}
    \nabla_{\xi}^2 \phi_0^{\lambda}(x, t_2; \xi)=t_2\cdot  I_{(n-1)\times (n-1)}+ \text{small perturbation}
\end{equation}
where $I_{(n-1)\times (n-1)}$ is the identity matrix of order $(n-1)$.  Denote 
\begin{equation}\label{eq: critical point}
\psi_{t_2}^{\lambda}(x, y) = \phi_0^{\lambda}(x, t_2; \xi_c) -y\cdot \xi_c.
\end{equation}
Then by the stationary phase principle (see for instance Sogge \cite[Theorem 1.2.1]{MR1205579}), we obtain 
% Recall stationary phase lemma.  Let $\phi(t)$ be such that $\phi'(t_0)-0$ and $\phi''(t_0)\neq 0$, then 
% \[
% \int e^{i \lambda \phi(t) } a(t) dt = \frac{a(t_0) e^{i\phi(t_0)}}{\sqrt{\lambda} \sqrt{ \phi''(t_0)}} + Error.
% \]
\begin{equation}\label{eq: stationary phase}
F(x, t_2) = t_2^{-\frac{n-1}{2}}\int \widecheck{f}_{\tau}(y) e^{i \psi_{t_2}^{\lambda}(x, y)} a_{t_2}^{\lambda}(x, y) dy
\end{equation}
where 
\begin{equation}
    a_{t_2}^{\lambda}(x, y):=a_{t_2}(\frac{x}{\lambda}, \frac{y}{\lambda}),
\end{equation}
and $a_{t_2}$ is a compactly supported smooth function in both variables. To continue, we apply Taylor expansion of $\psi_{t_2}^{\lambda}(x, y) $ in $y$. 
% \[
% \psi^{\lambda}(x, T; y)= \psi^{\lambda}(x, T; 0) - \nabla_y \psi^{\lambda}(x, T; y) + \nabla_y^2 \psi^{\lambda}(x, T;y) |y|^2 + \cdots
% \]
Take $\nabla_y$ on both sides of \eqref{eq: critical point}:
\begin{equation}\label{220627e7_87z}
\nabla_y \psi_{t_2}^{\lambda}(x, y)= \nabla_{\xi} \phi_0^{\lambda}(x, t_2; \xi_c) \cdot \nabla_y \xi_c -\xi_c -y \cdot \nabla_y \xi_c=-\xi_c.
\end{equation}
Hence
\[
\nabla_y^2 \psi_{t_2}^{\lambda} =-\nabla_y \xi_c.
\]
We claim that 
\begin{equation}\label{220612e8_60}
    |\nabla_y \xi_c|\leq \frac{1}{t_2}.
\end{equation}
Here by $|\cdot|$ of the matrix $\nabla_{y}\xi_c$, we mean the maximum of all the entries. To see this, we go back to the definition of $\xi_c$ in \eqref{220612e8_55}, and differentiate both side: 
\begin{equation}\label{220612e8_61}
\nabla_\xi^2 \phi_0^{\lambda}(x, t_2; \xi_c) \nabla_y \xi_c = I_{(n-1)\times (n-1)}.
\end{equation}
The claim now follows from \eqref{220612e8_56}. Moreover, taking $\nabla_x$ on both sides of $\nabla_{\xi} \phi_0^{\lambda}(x, t_2; \xi_c) =y$, we obtain
\begin{equation}\label{220612e8_64}
\nabla_x \nabla_{\xi} \phi_0^{\lambda} + \nabla_{\xi}^2 \phi_0^{\lambda} \cdot \nabla_x \xi_c =0.
\end{equation}
If we keep differentiating both sides of \eqref{220612e8_61} and \eqref{220612e8_64} in $x, y$, then we will be able to obtain
\begin{equation}\label{220612e8_63}
    |\nabla^{\alpha}_x\nabla^{\alpha'}_y \xi_c| \lesim_{\alpha} t_2^{-\alpha} t_2^{-\alpha'},
\end{equation}
for every $\alpha, \alpha'\in \N$. The proof is left out.\\

From \eqref{220612e8_60} and the fact that $\widecheck{f}_{\tau}$ is essentially supported on $B_{\sqrt{R}}$, we see that we can ``ignore" the quadratic term in the Taylor expansion of $\psi^{\lambda}_{t_2}(x, y)$ in the $y$ variable. Write
\begin{equation}\label{220626e7_91}
F(x, t_2)= e^{i\psi_{t_2}^{\lambda} (x, 0)} \int \widecheck{f}_{\tau}(y) e^{i y \cdot \nabla_y \psi_{t_2}^{\lambda}(x, 0)+iw_{3, t_2}(x, y)} a^{\lambda}_{t_2}(x, y) dy,
\end{equation}
for some error function $w_{3, t_2}(x, y)$.  Next, we do Taylor's expansion for $\nabla_y \psi^{\lambda}_{t_2}(x, 0)$ in the $x$ variable. Recall our assumption that $F(x, t_2)$ is essentially supported on a ball of radius $R/\sqrt{r}$; let $x_0$ denote its center. Write 
\begin{equation}
\begin{split}
        y\cdot \nabla_y \psi^{\lambda}_{t_2}(x, 0)
    & =
    y\cdot \nabla_y \psi^{\lambda}_{t_2}(x_0, 0)+
    y\cdot \nabla_x \nabla_y \psi^{\lambda}_{t_2}(x_0, 0)x+w_{4, t_2}(x, y),
\end{split}
\end{equation}
for some error function $w_{4, t_2}(x, y)$. In particular, 
\begin{equation}
    |w_{4, t_2}(x, y)|\lesim \frac{1}{R^2} \pnorm{\frac{R}{\sqrt{r}}}^2 |y|\lesim 1. 
\end{equation}
Next, by \eqref{220627e7_87z}, \eqref{220612e8_64} and \eqref{220612e8_56}, we see that $\nabla_x\nabla_y \psi^{\lambda}_{t_2}(x_0, 0)$ is a small perturbation of $\frac{1}{t_2}I_{(n-1)\times(n-1)}$. The desired bound 
\begin{equation}
    \norm{\widecheck{f}_{\tau}}_{L^2(\Omega_1)}^2 \lesim 
    \pnorm{
    \frac{R'_1}{R'_2}
    }^{\frac{n-m}{2}-O(\delta)} \norm{F(x, t_2)}_{L^2(\Omega_2)}^2
\end{equation}
now follows from Taylor's expansion and Lemma \ref{220614lemma8_8}. 
\end{proof}

In the end of this section, we prove Claim \ref{220616claim8_9}. As mentioned above, the main idea is already explained in Lemma \ref{220614lemma8_8} and the proof of Lemma \ref{220608lemma8_9}. The extra work we need to do is to take care of the refinement process of wave packets in the polynomial partitioning algorithm. In other words, in the partitioning algorithm, each time we have an algebraic dominating case, we will need to remove certain wave packets, and therefore the input function $f^{\not\sim}_{\kappa, S_{\Box}, \tau, \ell'}$ also changes as the algorithm proceeds. 

\begin{proof}[Proof of Claim \ref{220616claim8_9}] To simplify notation, let us write 
\begin{equation}
    f_R:=f^{\not\sim}_{\kappa, S_{\Box}, \tau, \ell'}.
\end{equation}
Here we use $R$ to emphasize that the function $f_{R}$ is built on wave packets from $\T[B_R]$. Our goal is to prove 
\begin{equation}\label{220627e7_102}
    \|T^{\lambda} (f_R)^*_{\iota, S}\|_{L^2( (B(x_1, 2r) \times \{t_1\}) \cap S_{\Box})}^2 
    \lessapprox 
    (\frac{r}{R})^{\frac{n-n'}{2}} \|T^{\lambda}f_R \|_{L^2(\{ t=t_2\})}^2
\end{equation}
We apply Lemma \ref{220608lemma8_1} and find an algebraic variety $Z_1\subset \{t=t_1\}$ satisfying that the angle between $T_{\bfz_1}(Z_1)$ and $\{\vec{e}_1, \dots, \vec{e}_{n'-1}\}$ is $\le 1/(100n)$, such that 
\begin{equation}
    (B(x_1, 2r)\times \{t_1\})\cap S_{\Box}\subset \mc{N}_{\sqrt{r}}(Z_1)\cap B_{\sqrt{R}}\cap \{t=t_1\}=:\Omega_1,
\end{equation}
for some ball $B_{\sqrt{R}}$ of radius $\sqrt{R}$. Moreover, by Lemma \ref{220608lemma8_2} and the definition of brooms, we can find an algebraic variety $Z_2\subset \{t=t_1+R\}$ satisfying that the angle between $T_{\bfz_2}(Z_2)$ and $\{\vec{e}_1, \dots, \vec{e}_{n'-1}\}$ is $\le 1/(100n)$ for every $\bfz_2\in Z_2$, such that $T^{\lambda} f_R(\cdot, t_2)$ is essentially supported on \begin{equation}
    \mc{N}_{\sqrt{R}}(Z_2)\cap B_{R^{1+\delta}/\sqrt{r}}\cap \{t=t_2\}=:\Omega_2. 
\end{equation} 
Under the above notation, \eqref{220627e7_102} can be written as
\begin{equation}\label{220627e7_105}
    \|T^{\lambda} (f_R)^*_{\iota, S}\|_{L^2(\Omega_1)}^2 
    \lessapprox (\frac{r}{R})^{\frac{n-n'}{2}} \|T^{\lambda}f_R \|_{L^2(\Omega_2)}^2.
\end{equation}
To proceed, we need to recall notation and definitions from Subsection \ref{220610subsection6_4}. Let $\mf{n}_r^*\in \{\mf{n}^*_0, \mf{n}^*_1, \dots\}$ be such that $S\in \mf{n}_r^*$. Collect all the ancestors of $\mf{n}^*_r$ that are in $\cup_j (\mf{M}_j\cup \mf{R}_j)$, and list them in descending order 
\begin{equation}
    \mf{n}^*_{R_1}, \mf{n}_{R_2}^*, \dots, \mf{n}^*_{R_{W-1}},
\end{equation}
where $R>R_1> R_2>\dots> R_{W-1}> r$. Moreover, denote $R_0:=R$, $\mf{n}^*_{R_0}:=\mf{n}^*_0$, $R_W:=r$ and $\mf{n}^*_{R_W}:=\mf{n}^*_r$. Note that we have the trivial bound $W\le \delta_n^{-10}$. Next,  find $S_{R_w}\in \mf{n}^*_{R_w}$ for each $1\le w\le W$ such that 
\begin{equation}
    S=S_{R_W}\subset S_{R_{W-1}}\subset \dots\subset S_{R_1}\subset S_{R_0}=B_R,
\end{equation}
and each $S_{R_w}$ is contained in a ball $B_{R_w}$ of radius $R_w$. \\

We will prove Claim \ref{220616claim8_9} by applying Lemma \ref{220608lemma8_9} iteratively. Denote 
\begin{equation}
\mc{N}_{\sqrt{R_{w}}}(\Omega_1)\cap \{t=t_1\}=:\omegaone_{R_w}, \ \ 
    \mc{N}_{R/\sqrt{R_{w}}}(\Omega_2)\cap \{t=t_1+R\}=:\omegatwo_{R_w},
\end{equation}
for every $w$. Recall that $T^{\lambda} f_{R_0}(\cdot, t_2)$ is supported on $\omegatwo_{R_0}$. By Lemma \ref{220608lemma8_9} with $m=n'$, $R'_1=R_{1}$ and $R'_2=R_0$, we obtain 
\begin{equation}
    \norm{T^{\lambda} f_{R_0}}_{
    L^2(
    \omegaone_{R_{1}}
    )
    }
    \lesim 
    \pnorm{
    \frac{R_{1}}{R_0}
    }^{\frac{n-n'}{2}-O(\delta)}
    \norm{
    T^{\lambda} f_{R_0}
    }_{
    L^2(
    \omegatwo_{R_{0}}
    )
    }
\end{equation}
On the ball $B_{R_{1}}$, we have the new wave packet decomposition 
\begin{equation}
    f_{R_0}=\sum_{T\in \T[B_{R_{1}}]} (f_{R_0})_T.
\end{equation}
Let $\T'[B_{R_{1}}]$ be a subset of $\T[B_{R_{1}}]$ such that 
\begin{equation}
    (f_{R_0})^*_{\iota, S_{R_{1}}}=\sum_{T\in \T'[B_{R_{1}}]} (f_{R_0})_T,
\end{equation}
that is, we throw away certain wave packets, depending on whether we are in the tangential case or the transverse case. Denote 
\begin{equation}\label{220728e6_120}
    f_{R_1}:=
    \sum_{
    \substack{
    T\in \T'[B_{R_{1}}]\\
    T\cap \omegaone_{R_{1}}\neq \emptyset
    }
    } (f_{R_0})_T.
\end{equation}
Therefore by $L^2$ orthogonality, we have 
\begin{equation}
    \norm{
    T^{\lambda} f_{R_1}
    }_{
    L^2(\omegaone_{R_{1}})
    } \lesim 
    \norm{
    T^{\lambda} f_{R_0}
    }_{
    L^2(\omegaone_{R_{1}})
    }
\end{equation}
Note that by the construction of $f_{R_1}$, the function $T^{\lambda} f_{R_1}(\cdot, t_1)$ is essentially supported on $\omegaone_{R_{1}}$. Therefore, by \eqref{220627e7_44} and $L^2$ bounds for H\"ormander's operator at fixed time (see for instance H\"ormander \cite{MR340924}), we have  
\begin{equation}
    \norm{
    T^{\lambda} f_{R_1}(x, t_2)
    }_{L^2_x(\R^{n-1})} \lesim \norm{f_{R_1}}_2 \lesim \norm{
    T^{\lambda} f_{R_1}
    }_{
    L^2(\omegaone_{R_{1}})
    }. 
\end{equation}
The main observation is that $T^{\lambda} f_{R_1}(\cdot, t_2)$ is essentially supported on $\omegatwo_{R_{1}}$. Once this is proven, we see that we can repeat the above process: By Lemma \ref{220608lemma8_9} with $R'_1=R_{2}$ and $R'_2=R_{1}$, we obtain 
\begin{equation}
    \begin{split}
        \norm{T^{\lambda} f_{R_1}}_{
    L^2(
    \omegaone_{R_{2}}
    )
    }
    & 
    \lesim 
    \pnorm{
    \frac{R_{2}}{R_{1}}
    }^{\frac{n-n'}{2}-O(\delta)}
    \norm{
    T^{\lambda} f_{1}
    }_{
    L^2(
    \omegatwo_{R_{1}}
    )
    }\\
    & \lesim 
    \pnorm{
    \frac{R_{2}}{R_{0}}
    }^{\frac{n-n'}{2}-O(\delta)}
    \norm{
    T^{\lambda} f_{R_0}
    }_{
    L^2(
    \omegatwo_{R_{0}}
    )
    }.
    \end{split}
\end{equation}
We define $f_{R_2}$ similarly as above, and then observe that $T^{\lambda} f_{R_2}(\cdot, t_2)$ is essentially supported on $\omegatwo_{R_{2}}$. This allows us to repeat the above iteration one more time. In the end, we will obtain the desired bound \eqref{220627e7_105}.\\

It remains to prove that $T^{\lambda} f_{R_w}(\cdot, t_2)$ is essentially supported on $\omegatwo_{R_w}$ for every $w$. We will only prove the case $w=1$, and the other cases are the same. Recall the definition of $f_{R_1}$ from \eqref{220728e6_120}. Write 
\begin{equation}
    f_{R_0}=\sum_{T_0\in \T'[B_{R_0}]} f_{T_0},
\end{equation}
for some $\T'[B_{R_0}]\subset \T[B_{R_0}]$. Denote 
\begin{equation}\label{220728e6_120zz}
    f_{T_0, R_1}:=
    \sum_{
    \substack{
    T\in \T'[B_{R_{1}}]\\
    T\cap \omegaone_{R_{1}}\neq \emptyset
    }
    } (f_{T_0})_T.
\end{equation}
It suffices to prove that for each $T_0$, we have 
\begin{equation}\label{220728e6_126}
    \supp{
    T^{\lambda} f_{T_0, R_1}(\cdot, t_2)
    } \subset 
    \mc{N}_{R/\sqrt{R_1}}(
    \supp{
    T^{\lambda} f_{T_0}(\cdot, t_2)
    }
    ),
\end{equation}
where by $\mc{N}$ we mean neighborhood in $(n-1)$ dimensions. Note that $f_{T_0, R_1}$ consists of wave packets of frequency scale $R_1^{-1/2}$. In order to see where $T^{\lambda} f_{T_0, R_1}(\cdot, t_2)$ is supported, we do a wave packet decomposition for $f_{T_0, R_1}$ by using wave packets of frequency scale $R_0^{-1/2}$: 
\begin{equation}
    f_{T_0, R_1}=\sum_{T'_0\in \T[B_{R_0}]} (f_{T_0, R_1})_{T'_0}.
\end{equation}
In order for $T'_0\in \T[B_{R_0}]$ to have non-trivial contribution, we need that $T_0\cap T'_0\neq \emptyset$ and $\mathrm{dist}(\theta(T_0), \theta(T'_0))\lesim R_1^{-1/2}$. Under these two conditions, by the same Taylor expansion argument as in Lemma \ref{220608lemma8_2}, \eqref{220728e6_126} follows immediately. 
\end{proof}

\section{Bushes: Small grains}\label{220717section7}

Let $\vector{S}$ be a grain with its last component given by $(S, B(\bfx_0, r))$. In the previous section, we considered the case $r\ge \sqrt{R}$. In this section, we will consider the case $r\le \sqrt{R}$. As the grain $S$ is small, in particular, it is smaller than the scale $\sqrt{R}$ of a wave packet $T\in \T[B_R]$, we will see that this case is much easier to handle.

The goal of this section is to prove the following result. 
\begin{theorem}\label{220617theorem8_1}
Let $\vector{S}_{n'}=(S_n, \dots, S_{n'})$ be a multi-grain with $S_{n'}=\mc{N}_{r_{n'}^{1/2+\delta_{n'}}}(Z_{n'})\cap B_{r_{n'}}$, $r_{n'}\le \sqrt{R}$ and $Z_{n'}$ an $n'$-dimensional algebraic variety of degree $\lesim_{n'} 1$ in $\R^n$. Then for frequency caps $\tau$ of side length $r_{n'}^{-1/2}$ and $n'\le n''\le n$, we have
\begin{equation}
\|f^{\nsim(n'')}_{S_{n'}, \tau}\|_{L^2}^2 \lessapprox \Big(\frac{r_{n''}}{R}\Big)^{\frac{n-n'}{2}}
\|f_{\tau}\|_{L^2}^2,
\end{equation}
where $f^{\not\sim(n'')}_{S_{n'}, \tau}$ is defined in \eqref{220617e7_36} and \eqref{220617e7_36zz}, with the definition of $\not\sim$ given in Definition \ref{220617definition8_3} below.  
\end{theorem}

In Theorem \ref{220617theorem8_1}, the scale $r_{n'}$ is so small that we do not see the broom structure; we will replace brooms by bushes. 

\begin{definition}[Bushes]
Given a grain $(S, B(\bfx_0, r))$ with $r\le \sqrt{R}$, a collection of tubes $\T'[B_R]\subset \T[B_R]$ and a frequency cap $\tau$ of side length $r^{-1/2}$, the collection of tubes 
\begin{equation}
    \mc{B}(\T'[B_R]):=\{T\in \T'[B_R]: \theta(T)\subset \tau, T\cap S\neq\emptyset\}
\end{equation}
is called a bush generated by $\T'[B_R]$. The size $b$ of the bush $\mc{B}(\T'[B_R])$ is defined to be $R^{w\delta}$ with $w\in \N$, $R^{w\delta}\le \#\mc{B}(\T'[B_R])< R^{(w+1)\delta}$. Often $\mc{B}(\T'[B_R])$ will be written as $\mc{B}_b(\T'[B_R])$. In particular, if $T'[B_R]=\T[B_R]$, then we will simply write $\mc{B}_{b}$ for a bush. 
\end{definition}

Next, we define a two-ends relation $\sim_{\mf{n}}$ for nodes $\mf{n}\in \cup_{\iota}\mf{R}_{\iota}$ with $\rho(\mf{n})\le \sqrt{R}$.  Similarly as above, we start by introducing a few auxiliary functions $\chi_{\mf{n}, \kappa}=\chi_{\kappa}$, taking values $0$ or $1$, where $\kappa=(b_1, \mu_1, \dots, b_{\ell}, \mu_{\ell})$ and $b_{\ell'}, \mu_{\ell'}\in \{R^{w\delta}: w\in \N\}$ for every $1\le \ell'\le \ell$. Denote $r=\rho(\mf{n})$.  For $S\in \mf{n}$ and a tube $T\in \T[B_R]$, we say that 
\begin{equation}
    \chi_{b_1}(S, T)=1,
\end{equation}
if $T$ belongs to a bush of size $b_1$ rooted at $S$. Moreover, we say that 
\begin{equation}
    \chi_{b_1, \mu_1}(S, T)=1
\end{equation}
if 
\begin{equation}
    \chi_{b_1}(S, T)=1, \ \ \mu_1\le \sum_{S'\in \mf{n}} \chi_{b_1}(S', T)<\mu_1 R^{\delta}.
\end{equation}
Now let us assume that we have defined $\chi_{\kappa}$ with $\kappa=(b_1, \mu_1, \dots, b_{\ell}, \mu_{\ell})$ already, and we would like to define $\chi_{\kappa, b_{\ell+1}}$ and $\chi_{\kappa, b_{\ell+1}, \mu_{\ell+1}}$. For a fixed $S\in \mf{n}$, define 
\begin{equation}
    \T_{S, \kappa}:=\{T'\in \T[B_R]: \chi_{\kappa}(S, T')=1\}.
\end{equation}
We write $\T_{S, \kappa}$ as a disjoint union of bushes 
\begin{equation}
    \bigcup_{b_{\ell+1}} \{\mc{B}_{b_{\ell+1}, \tau}(\T_{S, \kappa})\}_{\tau} 
\end{equation}
where $\tau$ runs through all caps of side length $\rho(\mf{n})^{-1/2}$, and $\mc{B}_{b_{\ell+1}, \tau}(\T_{S, \kappa})$ is a bush of size $b_{\ell+1}$ with tubes coming from $\tau$. We say that 
\begin{equation}
    \chi_{\kappa, b_{\ell+1}}(S, T)=1, \text{ if } T\in \mc{B}_{b_{\ell+1}, \tau}(\T_{S, \kappa}), 
\end{equation}
for some $\tau$. Moreover, 
\begin{equation}
    \chi_{\kappa, b_{\ell+1}, \mu_{\ell+1}}(S, T)=1,
\end{equation}
if 
\begin{equation}
  \chi_{\kappa, b_{\ell+1}}(S, T)=1, \ \ \mu_{\ell+1}\le \sum_{S'\in\mf{n}} \chi_{\kappa, b_{\ell+1}}(S', T)<  \mu_{\ell+1}R^{\delta}.   
\end{equation}
This finishes the definition of the auxiliary functions. 

For $\kappa=(b_1, \mu_1, \dots, b_{\ell}, \mu_{\ell})$, we say that $\kappa$ is admissible if there exists exactly one pair $(\ell_1, \ell_2)$ with $\ell_1\neq \ell_2$ such that \begin{equation}
    (b_{\ell_1}, \mu_{\ell_1})=(b_{\ell_2}, \mu_{\ell_2}). 
\end{equation}
Similarly to Lemma \ref{220608lemma8_5}, it is elementary to see that the number of admissible $\kappa$ is $O_{\delta}(1)$. 
\begin{definition}\label{220617definition8_3}
For a ball $B\subset B_R$ of radius $R^{1-\delta}$, a tube $T\in \T[B_R]$, a node $\mf{n}\in \cup_{\iota} \mf{R}_{\iota}$ with $\rho(\mf{n})\le \sqrt{R}$ and an admissible $\kappa$, we say that $T\sim_{\mf{n}, \kappa} B$ if $B$ maximizes 
\begin{equation}
    \#\{S'\in \mf{n}: S'\subset B', \chi_{\mf{n}, \kappa}(S', T)=1\},
\end{equation}
among all $B'$ of radius $R^{1-\delta}$. Moreover, we say that $T\sim_{\mf{n}} B$ if $T\sim_{\mf{n}, \kappa}$ for some admissible $\kappa$; we say that $T\sim B$ if $T\sim_{\mf{n}} B$ for some node $\mf{n}\in \cup_{\iota} \mf{R}_{\iota}$. 
\end{definition}

Now we are ready to prove Theorem \ref{220617theorem8_1}. \\

\begin{proof}[Proof of Theorem \ref{220617theorem8_1}.] Similarly to what we did in the proof of Theorem \ref{220615thm8_7}, we will write down more details for the case $n''=n'$, and the case $n''>n'$ is essentially the same.

We abbreviate $S_{n'}$ to $S$, $\vec{S}_{n'}$ to $\vec{S}$ and $r_{n'}$ to $r$. Let $B_{\iota}$ be the ball of radius $R^{1-\delta}$ containing $S$; let $\mf{n}$ be the node such that $S\in \mf{n}$. For an admissible multi-index $\kappa$ of the form $(b_1, \mu_1, \dots, b_{\ell}, \mu_{\ell})$, denote 
\begin{equation}
\T_{S, \kappa, \tau}:=\{T\in \T_{S, \kappa}: \theta(T)\subset \tau\},  
\end{equation}
\begin{equation}
\T^{\not\sim}_{S, \kappa, \tau}:=\{T\in \T_{S, \kappa}: \theta(T)\subset \tau, B_{\iota}\not\sim_{\mf{n}, \kappa} T\},  
\end{equation}
and 
\begin{equation}
f^{\nsim}_{\kappa, S, \tau}:= \sum_{T\in \mathbb{T}^{\nsim}_{S, \kappa,\tau}}f_T \quad \text{and} \quad 
f^{\nsim, *}_{\kappa, S, \tau} := (f^{\nsim}_{\kappa, S, \tau} )^*_{\iota, S}.
\end{equation}
Then similarly to \eqref{220615e8_36}, we have 
\begin{equation}
    f^{\not\sim}_{S, \tau}=\sum_{\kappa} f^{\not\sim, *}_{\kappa, S, \tau}. 
\end{equation}
Next, similarly to \eqref{220617e7_38} and \eqref{220627e7_42}, we have 
\begin{equation}\label{220630e8_17}
    \norm{f^{\not\sim}_{S, \tau}}_2^2 \lesim_{\delta} \sum_{\kappa} \norm{T^{\lambda} f^{\not\sim, *}_{\kappa, S, \tau}}_{L^2
    (
    \{t=t_1\}
    \cap S
    )
    }^2,
\end{equation}
where $(x_1, t_1)$ is a point in $S$. By the definition of $\kappa$, $\T_{S, \kappa, \tau}$ is contained in a bush $\mc{B}_{b_{\ell}}$. Write \begin{equation}
    f^{\not\sim, *}_{\kappa, S, \tau}=
    \sum_{
    T\in \T^{\not\sim}_{S, \kappa, \tau}\cap \mc{B}_{b_{\ell}}
    } (f_T)^*_{\iota, S}. 
\end{equation}
By the Cauchy-Schwarz inequality, we have 
\begin{equation}
    \norm{T^{\lambda} f^{\not\sim, *}_{\kappa, S, \tau}}_{L^2
    (
    \{t=t_1\}
    \cap S
    )
    }^2
    \lesim b_{\ell} 
    \sum_{
    T\in \T^{\not\sim}_{S, \kappa, \tau}\cap \mc{B}_{b_{\ell}}
    }
    \norm{T^{\lambda} 
    (f_T)^*_{\iota, S}
    }_{L^2(\{t=t_1\}\cap S)}^2
\end{equation}
It is elementary to see that\footnote{This can be viewed as a trivial version of Claim \ref{220616claim8_9}. }
\begin{equation}
    \norm{T^{\lambda} 
    (f_T)^*_{\iota, S}
    }_{L^2(\{t=t_1\}\cap S)}^2 \lesim 
    \pnorm{
    \frac{r}{R}
    }^{\frac{n-n'}{2}}
    \norm{f_T}_{L^2}^2.
\end{equation}
By summing over $T$, we obtain 
\begin{equation}
    \norm{T^{\lambda} f^{\not\sim, *}_{\kappa, S, \tau}}_{L^2
    (
    \{t=t_1\}
    \cap S
    )
    }^2 \lesim b_{\ell} \pnorm{
    \frac{r}{R}
    }^{\frac{n-n'}{2}}
    \norm{f^{\not\sim}_{\kappa, S, \tau}}_{L^2}^2. 
\end{equation}
In the end, we just need to show that 
\begin{equation}\label{220703e7_22}
| \mathbb{T}^{\nsim}_{S, \kappa, \tau} |\lesssim  b_{\ell}^{-1} R^{O(n\delta)} |\mathbb{T}_{\tau}|,
\end{equation}
which is an analogue of Lemma \ref{lem: counting}.  As the proof of \eqref{220703e7_22} is also more or less the same as that of Lemma \ref{lem: counting} (indeed simpler), we leave it out. \end{proof}

\section{Finishing the proof of Theorem \ref{201204thm5_1}}\label{220706section8}

In the last section, we combine the polynomial Wolff axiom in Lemma \ref{220223lemma6.4} and Corollary \ref{220711lemma5_5}, Property 1-4 in Subsection  \ref{220610subsection6_4}, the broom estimate in Theorem \ref{220615thm8_7} and the bush estimate in Theorem \ref{220617theorem8_1} to finish the proof of Theorem \ref{201204thm5_1}. \\

First of all, by Property 1 and repeated application of Property 2 in Subsection  \ref{220610subsection6_4} in the same way as how \cite[page 269]{HR2019} obtain equation (56) and (57), we obtain 
\begin{align}\label{220714e8_1}
    \|T^{\lambda} g\|_{\mathrm{BL}_{k, A}^{p}\left(B_{R}\right)} & \lessapprox \prod_{i=m-1}^{n-1} r_{i}^{\frac{\beta_{i+1}-\beta_{i}}{2}} D_{i}^{\frac{\beta_{i+1}}{2}-\left(\frac{1}{2}-\frac{1}{p_{n}}\right)}\\
    & \|g\|_{2}^{\frac{2}{p_{n}}} 
    \max _{O \in \mf{n}_{\ell_0}^*}
    \left\|g_{\iota, O}\right\|_{2}^{1-\frac{2}{p_{n}}},
\end{align}
where $\iota$ refers to $B_{\iota}\subset B_R$, the ball of radius $R^{1-\delta}$ containing $O$, the parameter $m$ comes from \eqref{220711e5_88} and $\mf{n}_{\ell_0}^*$ is as in \eqref{220708e5_78}. Next, by repeated application of Property 3, we obtain 
    \begin{align}\label{220711e8_3}
\max _{O \in 
\mf{n}_{\ell_0}^*
}
\left\|g_{\iota, O}\right\|_{2}^{2} \lessapprox r_{n'}^{-\frac{n-n'}{2}} \prod_{i=m-1}^{n'-1} r_{i}^{-1 / 2} D_{i}^{-i+\delta} 
\max_{
S_{n'}\in \mf{S}_{n'}
}
\left\|g^*_{\iota, S_{n'}}\right\|_{2}^{2}.
\end{align}
By Property 4, 
\begin{equation}
    \norm{
    g^*_{\iota, S_{n'}}
    }_2^2 
    \lessapprox 
    r_{n'}^{\frac{n-n'}{2}} 
    \pnorm{
    \prod_{i=n'}^{n-1} r_i^{-\frac{1}{2}}
    }
    r_{n''}^{-\frac{n-n''}{2}}
    \pnorm{
    \prod_{i=n''}^{n-1}
    r_i^{\frac{1}{2}}
    }
    R^{O(\epsilon_{\circ})}
    \norm{
    g_{\iota, S_{n'}}^{*(n'')}
    }_2^2.
\end{equation}
We then apply Corollary \ref{220711lemma5_5} and obtain 
\begin{equation}\label{220711e8_4}
    \Norm{
    g_{\iota, S_{n'}}^{*(n'')}
    }_2^2 \lessapprox 
    \pnorm{
    \prod_{j=n'}^{n''} r_j^{-\frac{1}{2}}
    }
    r_{n''}^{-\frac{n-n''-1}{2}}
    \max_{\tau: \ell(\tau)=r_{n''}^{-1/2}}
    \Norm{
    g_{\iota, S_{n'}}^{*(n'')}
    }_{
    L^2_{\mathrm{avg}}(\tau)
    }^2
\end{equation}
Recall the notation \eqref{220617e7_36} and \eqref{220617e7_36zz}. By the broom estimate in Theorem \ref{220615thm8_7} and the bush estimate in Theorem \ref{220617theorem8_1}, we have
\begin{equation}
\Norm{
    g_{\iota, S_{n'}}^{*(n'')}
    }_{
    L^2_{\mathrm{avg}}(\tau)
    }^2
    \lessapprox
    \pnorm{
    \frac{R}{r_{n''}}
    }^{-\frac{n-n'}{2}} \norm{g}_{\infty}^2.
\end{equation}
Putting everything together, we obtain 
\begin{align}
    \max_{O}
    \norm{
    g_{\iota, O}
    }_2^2
    & \lessapprox
    R^{O(\epsilon_{\circ})}
    \pnorm{\prod_{i=m-1}^{n'-1} D_{i}^{-i}
    }
    \pnorm{
    \prod_{j=m-1}^{n'-1} r_j^{\frac{1}{2}}
    }
    \pnorm{
    \prod_{i=m-1}^{n-1} r_i^{-1}
    }\\
    & 
    r_{n''}^{-(n-n''-1)}
    \pnorm{
    \prod_{i=n''+1}^{n-1}
    r_i
    }
    \pnorm{
    \frac{R}{r_{n''}}
    }^{-\frac{n-n'}{2}} 
    \norm{g}_{\infty}^2.
\end{align}
We pick $n''$ such that 
\begin{equation}
    n-n''=[\frac{n-n'}{3}]+1,
\end{equation}
bound $r_i$ by $R$ and obtain
\begin{align}
    \max_{O}
    \norm{g_{\iota, O}}_2^2 
    \lessapprox
    & \pnorm{
    \prod_{i=m-1}^{n-1}
    r_i^{-1}
    }
    \pnorm{
    \prod_{i=m-1}^{n'-1}
    r_i^{1/2}
    }
    \pnorm{
    \prod_{i=m-1}^{n'-1}
    D_i^{-i}
    }
    \pnorm{
    \frac{R}{r_{n''}}
    }^{-\frac{n-n'}{6}} 
    \norm{g}_{\infty}^2.
\end{align}
Note that when
\begin{equation}
    n'\le n/100+99m/100=:W_{m}^n,
\end{equation}
it holds that 
\begin{equation}
    n''\le \frac{2n}{3}+\frac{n'}{3}\le \frac{67n}{100}+\frac{33m}{100}
    =:M,
\end{equation}
and therefore we have 
\begin{align}
    \max_{O}
    \norm{g_{\iota, O}}_2^2 
    \lessapprox
    & \pnorm{
    \prod_{i=m-1}^{n-1}
    r_i^{-1}
    }
    \pnorm{
    \prod_{i=m-1}^{n'-1}
    r_i^{1/2}
    }
    \pnorm{
    \prod_{i=m-1}^{n'-1}
    D_i^{-i}
    }\\
    & \times 
    \begin{cases}
    \pnorm{
    \frac{R}{r_M}
    }^{-\frac{33(n-m)}{200}} 
    \norm{g}_{\infty}^2 & \text{ if } n'\le W_m^n,\\
    \norm{g}_{\infty}^2, & \text{otherwise}
    \end{cases}
\end{align}
By taking a weighted geometric average in $n'\in \{m, m+1, \dots, n-1\}$, and substituting into \eqref{220714e8_1}, we obtain 
\begin{equation}\label{220714e8_8}
    \norm{T^{\lambda}g}_{
    \BLka^p(B_R)
    }\lessapprox
    R^{-\Lambda} \pnorm{
    \prod_{i=m-1}^{n-1} r_i^{X_i} D_i^{Y_i}
    }
    \norm{g}_2^{\frac{2}{p}} \norm{g}_{\infty}^{1-\frac{2}{p}},
\end{equation}
where 
\begin{align}
    & \Lambda=\pnorm{\sum_{j=m}^{W_m^n} 
    \gamma_j
    }  \pnorm{\frac{1}{2}-\frac{1}{p}}\frac{33(n-m)}{200},\\
    & X'_i=\frac{\beta_{i+1}-\beta_i}{2}-\pnorm{\frac{1}{2}-\frac{1}{p}}+\frac{1}{2}\pnorm{1-\sum_{j=m}^i \gamma_j}\pnorm{\frac{1}{2}-\frac{1}{p}},\\
    & X_i=X'_i +
    \begin{cases}
    0 \text{ if } i\neq M,\\
    \Lambda  \text{ if } i=M,
    \end{cases}\\
    & Y_i=\frac{\beta_{i+1}}{2}-\pnorm{
    1+i(1-\sum_{j=m}^i \gamma_j) \pnorm{\frac{1}{2}-\frac{1}{p}}
    }
\end{align}
and 
\begin{align}
    & \gamma_{m-1}:=0, \  0\le \gamma_m, \dots, \gamma_n\le 1, \ \gamma_m+\dots+\gamma_n=1,\\
    & r_{m-1}=1, \ D_n=1, \ \beta_n=1, \ \beta_n\ge \beta_{n-1}\ge \dots\ge \beta_{m}.\label{220727e8_20}
\end{align}
Lemma \ref{220708lemma5_9} suggests that we write the coefficients on the right hand side of \eqref{220714e8_8} as 
\begin{align}
    R^{-\Lambda} (D_{m-1})^
    {
    \frac{\beta_m}{2}-m(\frac{1}{2}-\frac{1}{p})
    } 
    \pnorm{
    \prod_{i=m}^{n-1}
    D_i^{
    Y_i-
    (\sum_{j=m}^i X_j)
    }
    }
    \prod_{i=m}^{n-1} 
    \Big(r_i \prod_{j=i}^{n-1} D_j\Big)^{X_i} 
\end{align}
We pick $\gamma_i$ and $\beta_i$ so that 
\begin{align}
    & Y_{n-1}-X_{n-1}-\dots-X_m=\dots=Y_m-X_m=\frac{\beta_m}{2}-m(\frac{1}{2}-\frac{1}{p})=0,\label{220727e8_22}\\
    & X_M=\Lambda, \ \ X_i=0, \text{ when } i\neq M.\label{220727e8_23}
    \end{align}
One can check directly that if we set $p_m=2m/(m-1)$, 
\begin{equation}\label{220727e8_24}
    \gamma'_i:=\frac{1}{2(i-1)}
    \prod_{j=m}^i \frac{2(j-1)}{2j+1}, \forall m\le i\le n-1,
\end{equation}
and let $\gamma_i$ be given by the solution to the system
\begin{equation}\label{220727e8_25}
    \begin{split}
        & \gamma_i=\gamma'_i, \ \forall m\le i\le M-1, \\
        & (M+\frac{1}{2})(\gamma_M-\gamma'_M)=\Gamma_m^n \frac{33(n-m)}{200}, \text{ with } \Gamma_m^n:=\sum_{j=m}^{W_m^n} \gamma_j,\\
        & (M+\frac{2i+1}{2})(\gamma_{M+i}-\gamma'_{M+i})+\frac{3}{2}\sum_{j=M}^{M+i-1} (\gamma_j-\gamma'_j)=0, \ \ \forall 1\le i\le n-M-1,
    \end{split}
\end{equation}
then \eqref{220727e8_22} is satisfied. 
\begin{claim}\label{220727claim8_1}
Let $\gamma_i$ be given as above. Then $ \gamma_i\ge 0$ for every $m\le i\le n-1$ and 
\begin{equation}
    \sum_{i=m}^{n-1} \gamma_i\le 1. 
\end{equation}
\end{claim}
\begin{proof}[Proof of Claim \ref{220727claim8_1}]
By taking the logarithm and Taylor's expansion, we see that 
\begin{equation}
    \sum_{i=m}^{n-1} \gamma'_i\le 1/3.
\end{equation}
Moreover, under the assumption that $k\ge 2n/5$ in Theorem \ref{201204thm5_1}, it holds that 
\begin{equation}
\begin{split}
    \gamma'_i& =\frac{1}{2(i-1)} \frac{2(m-1) 2m \dots 2(i-1)}{(2m+1)(2m+3)\dots(2i+1)}\\
    & \ge \frac{1}{2(i-1)} \frac{2(m-1) 2m}{(2i-1)(2i+1)}\ge \frac{2n}{25},
\end{split}
\end{equation}
and 
\begin{equation}
     \gamma_M-\gamma'_M=(M+1/2)^{-1} \Gamma_m^n \frac{33(n-m)}{200}\le  \frac{1}{24}.
\end{equation}
The claim now follows from checking the system \eqref{220727e8_25}. 
\end{proof}

So far we have picked the values for $\gamma_i$ with $m\le i< n-1$ and $\beta_m$. By Claim \ref{220727claim8_1}, we can choose $\gamma_n=1-\gamma_m-\dots-\gamma_{n-1}$. Now we pick $\beta_i$ satisfying \eqref{220727e8_20} so that \eqref{220727e8_23} is satisfied. Elementary computation shows that  
\begin{align}\label{220727e8_30}
    \frac{1}{2}=\pnorm{\frac{1}{2}-\frac{1}{p}}
    \pnorm{
    \frac{n+m}{2}+\frac{1}{2} \sum_{j=m}^{n-1} (n-j)\gamma_j
    }
\end{align}
Therefore, to prove Theorem \ref{201204thm5_1}, it remains to prove that $p\le p_n(k)$, where $p$ is given in \eqref{220727e8_30}. This can done by using the trick in the appendix of \cite{hickman2020note}, together with elementary but tedious computation, which will be skipped.

%%%%%%%%%%%%%%%%%%%%%%%%%%%%%%%%%%%%%%%%%%%%%%%%%%%%%%%%%%%%%%%%%%%%%%%%%%%%%%%%%%%%%%%%%%%%%%%%

%REFERENCES

%%%%%%%%%%%%%%%%%%%%%%%%%%%%%%%%%%%%%%%%%%%%%%%%%%%%%%%%%%%%%%%%%%%%%%%%%%%%%%%%%%%%%%%%%%%%%%%%

%\def\bibfont{\small}

	\bibliography{reference.bib}
	\bibliographystyle{alpha}

\vspace{1cm}
	
	\noindent Shaoming Guo. UW Madison.
	Email address: shaomingguo@math.wisc.edu\\

\noindent Hong Wang. UCLA. Email address: hongwang@math.ucla.edu\\

\noindent Ruixiang Zhang.  UC Berkeley. Email address: ruixiang@berkeley.edu

\end{document}